\documentclass[11pt]{article}

\usepackage{amsmath,amsfonts,amssymb,amsthm,graphicx,color,mathtools}
\usepackage[normalem]{ulem} 

\usepackage{times, color}
\usepackage{graphicx}
\usepackage[utf8]{inputenc}
\usepackage{paralist}
\usepackage[colorlinks,pdfdisplaydoctitle]{hyperref}

\usepackage{bm}
\newcommand{\supp}{\mathrm{supp}\,}
\newcommand{\tr}{\mathrm{tr}\,}
\newcommand{\mc}{\mathcal}

\DeclareMathOperator{\diag}{diag}
\DeclareMathOperator{\fraco}{frac}

\newcommand{\argmin}{\operatornamewithlimits{argmin}}

\newcommand{\bbN}{\mathbb{N}}
\newcommand{\bbR}{\mathbb{R}}

\newcommand{\bbP}{\mathbb{P}}

\DeclareMathOperator{\1}{\mathrm{ \mathbf{1}}}
\def\Q1{Q}

\renewcommand{\P}{\mathrm P}
\renewcommand{\S}{\mathrm S}
\newcommand{\PM}{\mathrm{P}}
\newcommand{\SP}{\mathrm{S}}

\newcommand{\LAn}{\mc L_{\mu_n}}

\DeclareMathOperator{\e}{\mathrm e}
\newcommand{\ev}[1]{ {\boldsymbol{\mathbb  E}} \left[  #1 \right]}
\newcommand{\evalt}[2]{ {\boldsymbol{\mathbb  E}}_{#1} \left[  #2 \right]}

\renewcommand{\d}{\mathrm{d}}

\DeclareMathOperator{\Var}{Var}

\newcommand{\valpha}{\boldsymbol{\alpha}}
\newcommand{\ve}{\boldsymbol{e}}
\newcommand{\vnu}{{\bm{\nu}}}

\theoremstyle{definition}

\newtheorem{example}{Example}
\newtheorem{assumption}{Assumption}
\theoremstyle{remark}
\newtheorem{remark}{Remark}
\theoremstyle{plain}
\newtheorem{proposition}{Proposition}
\newtheorem{theorem}{Theorem}
\newtheorem{lemma}{Lemma}
\newtheorem{corollary}{Corollary}

\newcommand{\rev}[1]{\textcolor{black}{#1}}

\definecolor{darkorange}{RGB}{255,140,0}
\title{On the Convergence of the Laplace Approximation and {Noise-Level-Robustness of Laplace-based Monte Carlo Methods} for Bayesian Inverse Problems}

\author{Claudia Schillings$^\dag$, Bj\"orn Sprungk$^\ddag$ and Philipp Wacker$^\ast$\\[10pt]
\small
$^\dag$ University of Mannheim, Institute for Mathematics, Mannheim, Germany,\\ \small c.schillings@uni-mannheim.de\\
\small
$^\ddag$ Technische Universit\"at Bergakademie Freiberg, Faculty of Mathematics and Computer Science, Freiberg, Germany\\ 
\small bjoern.sprungk@math.tu-freiberg.de\\
\small $^\ast$ Friedrich-Alexander-Universit\"at Erlangen-N\"urnberg, Department of Mathematics, Erlangen, Germany,\\ \small phkwacker@gmail.com\\
}

\date{}

\begin{document}
\maketitle

\begin{abstract}
The Bayesian approach to inverse problems provides a rigorous framework for the incorporation and quantification of uncertainties in measurements, parameters and models. 
We are interested in designing numerical methods which are robust w.r.t.~the size of the observational noise, i.e., methods which behave well in case of concentrated posterior measures.
The concentration of the posterior is a highly desirable situation in practice, since it relates to informative or large data.
However, it can pose a computational challenge for numerical methods based on the prior or reference measure. 
We propose to employ the Laplace approximation of the posterior as the base measure for numerical integration in this context.
The Laplace approximation is a Gaussian measure centered at the maximum a-posteriori estimate and with covariance matrix depending on the logposterior density.
We discuss convergence results of the Laplace approximation in terms of the Hellinger distance and analyze the efficiency of Monte Carlo methods based on it. 
In particular, we show that Laplace-based importance sampling and Laplace-based quasi-Monte-Carlo methods are robust w.r.t.~the concentration of the posterior for large classes of posterior distributions and integrands whereas prior-based importance sampling and plain quasi-Monte Carlo are not.
Numerical experiments are presented to illustrate the theoretical findings.

\end{abstract}



\section{Introduction}
\label{sec:intro}
The identification of unknown parameters from noisy observations arises in various areas of application, e.g., engineering systems, biological models, environmental systems. 
In recent years, Bayesian inference has become a popular approach to model inverse problems \cite{Stuart2010}, i.e., 
{noisy observations are used to update the knowledge of unknown parameters from a prior distribution to the posterior distribution. 
The latter is then the solution of the Bayesian inverse problem and obtained by conditioning the prior distribution on the data.}
This approach is very appealing in various fields of applications, since uncertainty quantification can be performed, once the prior distribution is updated---barring the fact that Bayesian credible sets are not in a one-to-one correspondence to classical confidence sets, see \cite{cox1993analysis,szabo2015frequentist}.

To ensure the applicability of the Bayesian approach to computationally demanding models, there has been a lot of research effort towards improved algorithms allowing for effective sampling or integration w.r.t.~the resulting posterior measure.
For example, the computational burden of expensive forward or likelihood models can be reduced by surrogates or multilevel strategies \cite{MarzoukXiu2009,DodwellEtAl2015,HoangEtAl2013,ScheichlEtAl2016} and for many classical sampling or integration methods such as Quasi-Monte Carlo \cite{KuoActa}, Markov chain Monte Carlo \cite{CotterEtAl2013,RudolfSprungk2018,Vollmer2015}, and numerical quadrature \cite{SchillingsSchwab2013,PengEtAl2017} we now know modifications and conditions which ensure a dimension-independent efficiency.

However, a completely different, but very common challenge for many numerical methods has drawn surprisingly less attention so far: the challenge of concentrated posterior measures such as
\begin{equation}\label{equ:post}
	\mu_n(\d x) = \frac 1{Z_n} \exp\left(-n \Phi_n(x)\right) \mu_0(\d x),
	\quad 
	Z_n \coloneqq \int_{\bbR^d} \exp\left(-n \Phi_n(x)\right) \mu_0(\d x),
\end{equation}
Here, $n\gg 1$ and $\mu_0$ denotes a reference or prior probability measure on $\bbR^d$ and $\Phi_n\colon \bbR^d\to[0,\infty)$ are negative log-likelihood functions resulting, e.g., from $n$ observations.

From a modeling point of view the concentration effect of the posterior is a highly desirable situation due to large data sets and less remaining uncertainty about the parameter to be inferred.
From a numerical point of view, on the other hand, this can pose a delicate situation, since standard integration methods may perform worse and worse if the concentration increases due to $n\to\infty$.
Hence, understanding how sampling or quadrature methods for $\mu_n$ behave as $n\to\infty$ is a crucial task with immediate benefits for purposes of uncertainty quantification.
Since small noise yields ``small'' uncertainty, one might be tempted to consider only optimization-based approaches in order to compute a point estimator (i.e., the maximum a-posteriori estimator) for the unknown parameter which is usually computationally much cheaper than a complete Bayesian inference.
However, for quantifying the remaining risk, e.g., computing the posterior failure probability for some quantity of interest, we still require efficient integration methods for concentrated posteriors as $\mu_n$.
Nonetheless, we will use well-known preconditioning techniques from numerical optimization in order to derive such robust integration methods for the small noise setting.

Numerical methods are often based on the prior $\mu_0$, since $\mu_0$ is usually a simple measure allowing for direct sampling or explicit quadrature formulas. 
However, for large $n$ most of the corresponding sample points or quadrature nodes will be placed in regions of low posterior importance missing the needle in the haystack---the minimizers of $\Phi_n$.
An obvious way to circumvent this is to use a numerical integration w.r.t.~another reference measure which can be straightforwardly computed or sampled from and concentrates around those minimizers and shrinks like the posterior measures $\mu_n$ as $n\to \infty$.
In this paper we consider numerical methods based on a Gaussian approximation of $\mu_n$---the \emph{Laplace approximation.} 

When it comes to integration w.r.t.~an increasingly concentrated function, the well-known and widely used \emph{Laplace's method} provides explicit asymptotics for such integrals, i.e., under certain regularity conditions \cite{Wong2001} we have for $n\to \infty$ that
\begin{equation}\label{equ:weak}
	\int_{\bbR^d} f(x) \ \exp(-n \Phi(x)) \d x
	\ = \ f(x_\star) \frac{(2\pi)^{d/2}\, \exp(-n \Phi(x_\star))}{n^{d/2} \sqrt{\det\left(\nabla^2 \Phi(x_\star) \right)}}\ \left(1+\mc O(n^{-1})\right)
\end{equation}
where $x_\star \in \bbR$ denotes the assumed unique minimizer of $\Phi\colon\bbR^d\to\bbR$.
This formula is derived by approximating $\Phi$ by its second-order Taylor polynomial at $x_\star$.
We could now use \eqref{equ:weak} and its application to $Z_n$ in order to derive that $\int_{\bbR^d} f(x) \ \mu_n(\d x) \to f(x_\star) $ as $n\to\infty$.
However, for finite $n$ this is only of limited use, e.g., consider the computation of posterior probabilities where $f$ is an indicator function.
Thus, in practice we still rely on numerical integration methods in order to obtain a reasonable approximation of the posterior integrals $\int_{\bbR^d} f(x) \ \mu_n(\d x)$.
Nonetheless, the second-order Taylor approximation employed in Laplace's method provides us with (a guideline to derive) a Gaussian measure approximating $\mu_n$.

This measure itself is often called the \emph{Laplace approximation} of $\mu_n$ and will be denoted by $\LAn$.
Its mean is given by the maximum a-posteriori estimate (MAP) of the posterior $\mu_n$ and its covariance is the inverse Hessian of the negative log posterior density.
Both quantities can be computed efficiently by numerical optimization and since it is a Gaussian measure it allows for direct samplings and easy quadrature formulas.
The Laplace approximation is widely used in optimal (Bayesian) experimental design to approximate the posterior distribution (see, for example, \cite{alexanderian2016fast}) and has been demonstrated to be particularly useful in the large data setting, see \cite{LongEtAl2013,reviewDOE} and the references therein for more details.
Moreover, in several recent publications the Laplace approximation was already proposed as a suitable reference measure for numerical quadrature \cite{SchillingsScaling,PengEtAl2017} or importance sampling \cite{BECK2018523}.
Note that preconditioning strategies based on Laplace approximation are also referred to as Hessian-based strategies due to the equivalence of the inverse covariance and the Hessian of the corresponding optimization problem, cp. \cite{PengEtAl2017}.
In \cite{SchillingsScaling}, the authors showed that a Laplace approximation-based adaptive Smolyak quadrature for Bayesian inference with affine parametric operator equations exhibits a convergence rate independent of the size of the noise, i.e., independent of $n$.

This paper extends the analysis in \cite{SchillingsScaling} for quadrature to the widely applied \emph{Laplace-based importance sampling} and \emph{Laplace-based quasi-Monte Carlo} (QMC) integration.

Before we investigate the scale invariance or robustness of these methods we examine the behaviour of the Laplace approximation and in particular, the density $\frac{\d \mu_n}{\d \LAn}$.
The reason behind is that, for importance sampling as well as QMC integration, this density naturally appears in the methods, hence, if it deteriorates as $n\to \infty$, this will be reflected in a deteriorating efficiency of the method.
For example, for $\Phi_n \equiv \Phi$ the density w.r.t.~the prior measure $\frac{\d \mu_n}{\d \mu_0} = \exp(-n\Phi)/Z_n$ deteriorates to a Dirac function at the minimizer $x_\star$ of $\Phi$ as $n\to\infty$ which causes the shortcomings of Monte Carlo or QMC integration w.r.t.~$\mu_0$ as $n\to\infty$.
However, for the Laplace approximation we show that the density $\frac{\d \mu_n}{\d \LAn}$ converges 
$\LAn$-almost everywhere to $1$ which in turn results in a robust---and actually improving---performance w.r.t.~$n$ of related numerical methods.
In summary, the main results of this paper are the following:
\begin{enumerate}
\item{\bf Laplace Approximation:}
Given mild conditions the Laplace approximation $\LAn$ converges in Hellinger distance to $\mu_n$:
\[
		d_\mathrm{H}(\mu_n, \LAn) \in \mc O(n^{-1/2}).
\]
This result is closely related to the well-known Bernstein--von Mises theorem for the posterior consistency in Bayesian inference  \cite{VanDerVaart1998}. 
The significant difference here is that the covariance in the Laplace approximation depends on the data and the convergence holds for the particularly observed data whereas in the classical Bernstein--von Mises theorem the covariance is the inverse of the expected Fisher information matrix and the convergence is usually stated in probability. 

\item{\bf Importance Sampling:} We consider integration w.r.t.~measures $\mu_n$ as in \eqref{equ:post} where $\Phi_n(x) = \Phi(x) - \iota_n$ for a $\Phi\colon \bbR^d \to [0,\infty)$ and $\iota_n \in \bbR$. \\[-0.35cm]
\begin{itemize}
\item \emph{Prior-based Importance Sampling:}
We consider the case of prior-based importance sampling, i.e., the prior $\mu_0$ is used as the importance distribution for computing the expectation of smooth integrands $f\in L^2_{\mu_0}(\bbR)$. \rev{Here, the asymptotic variance w.r.t.~such measures $\mu_n$ deteriorates like $n^{d/2-1}$.}

\item
\emph{Laplace-based Importance Sampling.}
The (random) error $e_{n,N}(f)$ of Laplace-based importance sampling for computing expectations of smooth integrands $f\in L^2_{\mu_0}(\bbR)$ w.r.t.~such measures $\mu_n$ using a fixed number of samples $N\in\bbN$ decays in probability almost like $n^{-1/2}$, i.e.,
\[
	n^\delta e_{n,N}(f) \xrightarrow[n\to\infty]{\bbP} 0,
	\qquad
	\delta < 1/2.
\]
\end{itemize}

\item{\bf Quasi-Monte Carlo:}
We focus for the analysis of the quasi-Monte Carlo methods on the bounded case of $\mu_0 = \mc U([\frac12,\frac12]^d)$. 
\begin{itemize}
\item\emph{Prior-based Quasi-Monte Carlo:}
{The root mean squared error estimate for computing integrals of the form \eqref{equ:weak} by QMC using randomly shifted Lattice rules deteriorates like $n^{d/4}$ as $n\to \infty$.} 

\item\emph{Laplace-based Quasi-Monte Carlo:}
If the lattice rule is transformed by an affine mapping based on the mean and the covariance of the Laplace approximation, then the resulting root mean squared error decays like $n^{-d/2}$ for integrals of the form \eqref{equ:weak}.
\end{itemize}
\end{enumerate}

The outline of the paper is as follows: in Section \ref{sec:Laplace} we introduce the Laplace approximation for measures of the form \eqref{equ:post} and the notation of the paper.
In Section \ref{sec:conv} we study the convergence of the Laplace approximation.  We also consider the case of singular Hessians or {perturbed Hessians} and provide some illustrative numerical examples. {At the end of the section, we shortly discuss the relation to the classical Bernstein--von Mises theorem.}
The main results about importance sampling and QMC using the prior measure and the Laplace approximation, respectively, are then discussed in Section \ref{sec:numerics}.
We also briefly comment on existing results for numerical quadrature and provide several numerical examples illustrating our theoretical findings.
The appendix collects the rather lengthy and technical proofs of the main results.

\section{Convergence of the Laplace Approximation}
\label{sec:Laplace}
We start with recalling the classical \emph{Laplace method} for the asymptotics of integrals.
\begin{theorem}[variant of {\cite[Section IX.5]{Wong2001}}]
\label{theo:laplace_method}
Set
\[
	J(n)
	\coloneqq
	\int_D f(x) \exp(-n \Phi(x)) \d x,
	\qquad
	n\in\bbN,
\]
where $D\subseteq \bbR^d$ is a possibly unbounded domain and let the following assumptions hold:
\begin{enumerate}
\item
The integral $J(n)$ converges absolutely for each $n\in\bbN$.

\item
There exists an $x_\star$ in the interior of $D$ such that for every $r > 0$ there holds
\[
	\delta_r \coloneqq \inf_{x \in B^c_r(x_\star)} \Phi(x) - \Phi(x_\star)  > 0,
\]
{where $B_r(x_\star) \coloneqq \{x \in \bbR^d\colon \|x-x_\star\| \leq r \}$ and $B^c_r(x_\star) \coloneqq \bbR^d \setminus B_r(x_\star)$.}

\item
\rev{In a neighborhood of $x_\star$ the function $f:D\to \mathbb R$ is $(2p+2)$ times continuously differentiable and $\Phi\colon \bbR^d \to \bbR$ is $(2p+3)$ times continuously differentiable for a $p\ge 0$}, i.e., and the Hessian $H_\star \coloneqq \nabla^2 \Phi(x_\star)$ is positive definite.
\end{enumerate}
\rev{
Then,  as $n\to \infty$, we have
\[
	J(n) 
	=
	\e^{-n \Phi(x_\star)}\,
	n^{-d/2}\,
	\left(\sum_{k=0}^p c_k(f) n^{- \rev{k}}
	+ \mc O\left(n^{-p-1}\right) \right)
\]
}
where $c_k(f) \in \bbR$ and, particularly, $c_0(f) = \rev{ \sqrt{\det(2\pi H^{-1}_\star)}\, f(x_\star)}$.
\end{theorem}

\begin{remark} \label{rem:laplace_method}
As stated in \cite[Section IX.5]{Wong2001} 
the asymptotic
\[
	\lim_{n\to\infty} \frac{J(n)}{c_0(f)\, \exp(-n \Phi(x_\star))\, n^{-d/2} } = 1 
\]
with $c_0(f)$ is as above, already holds for $f\colon \bbR^d\to\bbR$ being continuous and $\Phi\colon \bbR^d\to \bbR$ being twice continuously differentiable in a neighborhood of $x_\star$ with positive definite $\nabla^2 \Phi(x_\star)$---given that the first two assumptions of Theorem \ref{theo:laplace_method} are also satified.
\end{remark}

Assume that $\Phi(x_\star) = 0$, then the above theorem and remark imply 
\[
	\left| \int_{\bbR^d} f(x)\  \exp(-n\Phi(x)) \ \d x - \int_{\bbR^d} f(x)\  \exp(- \frac n2 \|x-x_\star\|^2_{H_\star}) \ \d x\right| \in o(n^{-d/2})
\]
for continuous and integrable $f\colon \bbR^d\to\bbR$, where $\|\cdot \|_A=\|A^{1/2}\cdot\|$ for a symmetric positive definite matrix $A\in\bbR^{d\times d}$.
This is similar to the notion weak convergence (albeit with two different non-static measures).
If we additionally claim that $f(x_\star) > 0,$ then also
\begin{equation*}
	\lim_{n\to\infty} \frac{\int_{\bbR^d} f(x)\  \exp(-n\Phi(x)) \ \d x}{ \int_{\bbR^d} f(x)\  \exp(- \frac n2 \|x-x_\star\|^2_{H_\star}) \ \d x}
	=
	1
	\qquad
	\text{ as } n\to\infty,
\end{equation*}
which is sort of a relative weak convergence. 
In other words, the asymptotic behaviour of $\int f\  \e^{-n\Phi}\ \d x$, in particular, its convergence to zero, is the same as of the integral of $f$ w.r.t.~an unnormalized Gaussian density with mean in $x_\star$ and covariance $ (nH_\star)^{-1}$.

If we consider now probability measures $\mu_n$ as in \eqref{equ:post} but with $\Phi_n \equiv \Phi$ where $\Phi$ satisfies the assumptions of Theorem \ref{theo:laplace_method}, and if we suppose that $\mu_0$ possesses a continuous Lebesgue density $\pi_0\colon \bbR\to[0,\infty)$ with {$\pi_0(x_\star) > 0$}, then Theorem \ref{theo:laplace_method} and Remark \ref{rem:laplace_method} will imply for continuous and integrable $f\colon \bbR^d\to\bbR$ that
\begin{align*}
	\lim_{n\to \infty}
	\int_{\bbR^d} f(x) \mu_n(\d x)
	& =
	\lim_{n\to \infty}
	\frac{\int_{\bbR^d} f(x) \pi_0(x) \exp(-n \Phi(x))\ \d x}{\int_{\bbR^d} \pi_0(x) \exp(-n \Phi(x))\ \d x}\\
	& =
	\lim_{n\to \infty}
	\frac{c_0(f\pi_0)\, n^{-d/2}}{c_0(\pi_0)\, n^{-d/2}}
	=
	\frac{c_0(f\pi_0)}{c_0(\pi_0)}
	= f(x_\star).
\end{align*}
The same reasoning applies to the expectation of $f$ w.r.t.~a Gaussian measure $\mc N(x_\star, (nH_\star)^{-1})$ with unnormalized density $\exp(- \frac n2 \|x-x_\star\|^2_{H_\star})$.
Thus, we obtain the weak convergence of $\mu_n$ to $\mc N(x_\star, (nH_\star)^{-1})$, i.e., for any continuous and bounded $f\colon \bbR^d\to\bbR$ we have
\begin{equation}\label{equ:Laplace_weak}
	\lim_{n\to \infty} \left|\int_{\bbR^d} f(x) \ \mu_n(\d x) - \int_{\bbR^d} f(x) \ \mc N_{x_\star, (nH_\star)^{-1}}(\d x)\right|
	=
	0,
\end{equation}
where $\mc N_{x,C}$ is short for $\mc N(x,C)$.
In fact, for twice continously differentiable $f\colon \bbR^d \to \bbR$ we get by means of Theorem \ref{theo:laplace_method} the rate
\begin{equation}\label{equ:Laplace_weak_C2}
	\left|\int_{\bbR^d} f(x) \ \mu_n(\d x) - \int_{\bbR^d} f(x) \ \mc N_{x_\star, (nH_\star)^{-1}}(\d x)\right|
	\in
	\mc O(n^{-1}).
\end{equation}
Note that due to normalization we do not need to assume $\Phi(x_\star) = 0$ here.
Hence, this weak convergence suggests to use $\mc N_{x_\star, (nH_\star)^{-1}}$ as a Gaussian approximation to $\mu_n$.
In the next subsection we derive similar Gaussian approximation for the general case $\Phi_n \not\equiv \Phi$, {whereas subsection \ref{sec:conv} includes convergence results of the Laplace approximation in terms of the Hellinger distance.

\paragraph{Bayesian inference}
We present some context for the form of equation \eqref{equ:post} in the following. Integrals of the form \eqref{equ:post} arise naturally in the Bayesian setting for inverse problems with large amount of observational data or informative data. Note that the mathematical results for the Laplace approximation given in section \ref{sec:Laplace} are derived in a much more general setting and are not restricted to integrals w.r.t. the posterior in the Bayesian inverse framework. We refer to \cite{KaipioSomersalo2005,DashtiStuart2017} and the references therein for a detailed introduction to Bayesian inverse problems.

Consider a \rev{continuous forward response operator $\mathcal G\colon \bbR^d \to \mathbb R^K$}
 mapping the unknown parameters $x\in \bbR^d$ to the data space $\mathbb R^K$, where $K\in\mathbb N$ denotes the number of observations. 
We investigate the inverse problem of recovering unknown parameters $x\in \bbR^d$ from noisy observation $y\in \mathbb{R}^K$ given by
\begin{equation*}
y = \mathcal G(x)+\eta\,,
\end{equation*}
where $\eta \sim\mathcal N(0,\Gamma)$ is a Gaussian random variable with mean zero and covariance matrix $\Gamma$, 
which models the noise in the observations and in the model.

The Bayesian approach for this inverse problem of inferring $x$ from $y$ (which is ill-posed without further assumptions) works as follows: For fixed $y\in \mathbb{R}^K$ 
we introduce the least-squares functional (or negative loglikelihood in the language of statistics) $\Phi(\cdot;y):\mathbb R^d\to \mathbb R$ by
\begin{equation*}
\Phi(x ; y)=\frac{1}{2}|y-\mathcal G(x) |_{\Gamma}^2\,.
\end{equation*}
with $|\cdot |_\Gamma\coloneqq |\Gamma^{-\frac12}\cdot |$ denoting the weighted Euclidean norm in $\mathbb R^K$. 
The unknown parameter $x$ is modeled as a $\bbR^d$-valued random variable with prior distribution $\mu_0$ (independent of the observational noise $\eta$), which regularizes the problem and makes it well-posed by application of Bayes' theorem:
The pair $(x,y)$ is a jointly varying random variable on $\mathbb R^d \times \mathbb R^K$ and hence the solution to the Bayesian inverse problem is the 
conditional or posterior distribution $\mu$ of $x$ given the data $y$ 
where the law $\mu$ is given by 
\begin{equation*}
\mu(\d x)=\frac{1}{Z}\exp(-\Phi(x ; y))\mu_0(\d x)
\end{equation*}
with the normalization constant $Z \coloneqq \int_{\mathbb R^d}\exp(-\Phi(x;y))\mu_0(\d x)$.
If we assume a decaying noise-level by introducing a scaled noise covariance $\Gamma_n = \frac 1n \Gamma$, the resulting noise model $\eta_n \sim N(0,\Gamma_n)$ yields an $n$-dependent log-likelihood term which results in posterior measures $\mu_n$ of the form \eqref{equ:post} with $\Phi_n(x) = \Phi(x ; y)$.
Similarly, an increasing number $n\in\bbN$ of data $y_1, \ldots, y_n \in \bbR^k$ resulting from $n$ observations of $\mathcal G(x)$ with independent noises $\eta_1,\ldots,\eta_n\sim N(0,\Gamma)$ yields posterior measures $\mu_n$ as in \eqref{equ:post} with $\Phi_n(x) = \frac 1n \sum_{j=1}^n \Phi(x ; y_j)$.

\subsection{The Laplace Approximation}
Throughout the paper, we assume that the \rev{prior} measure $\mu_0$ is absolutely continuous w.r.t.~Lebesgue measure with density $\pi_0\colon\bbR^d\to[0,\infty)$, i.e.,
\begin{equation}\label{equ:nonGaussian_prior}
	\mu_0(\d x)= \pi_0(x) \d x, 
	\quad
	\text{ and we set }
	\quad
	\S_0 \coloneqq{\{x\in\bbR^d\colon \pi_0(x) > 0\} = \supp \mu_0}.
\end{equation}
Hence, also the measures $\mu_n$ in \eqref{equ:post} are absolutely continuous w.r.t.~Lebesgue measure, i.e., 
\begin{equation}\label{equ:post_2}
	\mu_n(\d x) = \frac 1{Z_n}\ \1_{\S_0}(x)\ \exp\left(-n I_n(x)\right) \d x
\end{equation}
where $I_n\colon \S_0 \to \bbR$ is given by
\begin{equation}\label{equ:I_n}
	I_n(x) \coloneqq \Phi_n(x) - \frac 1n\log \pi_0(x).
\end{equation}
In order to define the Laplace approximation of $\mu_n$ we need the following basic assumption.

\begin{assumption}\label{assum:LA_0}
There holds $\Phi_n, \pi_0 \in C^2(\S_0,\bbR)$\rev{, i.e., the mappings $\pi_0,\Phi_n\colon \S_0\to\bbR$ are twice continuously differentiable}.
Furthermore, $I_n$ has a unique minimizer $x_n \in \S_0$ satisfying
\[
	I_n(x_n) = 0,
	\quad
	\nabla I_n(x_n) = 0,
	\quad
	\nabla^2 I_n(x_n) > 0,
\]
where the latter denotes positive definiteness.
\end{assumption}

\begin{remark}\label{rem:In_zero}
Assuming that $\min_{x\in\S_0} I_n(x) = 0$ is just a particular (but helpful) normalization and in general not restrictive: If $\min_{x\in\S_0} I_n(x) = c >-\infty$, 
then we can simply set
\[
	\hat \Phi_n(x) \coloneqq \Phi_n(x) - c,
	\qquad
	\hat I_n(x) \coloneqq \hat \Phi_n(x) - \log \pi_0(x)
\]
for which we obtain
\[
	\mu_n(\d x) = \frac 1{\hat Z_n} \exp\left(-n \hat \Phi_n(x)\right) \mu_0(\d x),
	\quad 
	\hat Z_n = \int_{\bbR^d} \exp\left(-n \hat \Phi_n(x)\right) \mu_0(\d x),
\]
and $\min_{x\in\S_0} \hat I_n(x) = \min_x \hat \Phi_n(x) - \frac 1n \log \pi_0(x) = 0$.
\end{remark}

Given Assumption \ref{assum:LA_0} we define the \emph{Laplace approximation} of $\mu_n$ as the following Gaussian measure
\begin{equation}\label{equ:Laplace}
	\LAn \coloneqq \mc N (x_n, n^{-1} C_n),
	\qquad
	C^{-1}_n \coloneqq \nabla^2 I_n(x_n).
\end{equation}
Thus, we have
\begin{equation}\label{equ:Laplace2}
	\LAn(\d x)
	=
	\frac 1{\tilde Z_n}\exp\left(- \frac n2 \|x-x_n\|^2_{C_n^{-1}}\right) \ \d x,
	\qquad
	\tilde Z_n
	\coloneqq n^{-d/2} \sqrt{\det(2\pi C_n)},
\end{equation}
where 
we can view
\begin{equation}\label{equ:Laplace3}
	\tilde I_n(x)
	\coloneqq 
	\frac 12 \|x-x_n\|^2_{C_n^{-1}}
	= \underbrace{I_n(x_n)}_{=0} + \underbrace{\nabla I_n(x_n)^\top}_{=0} (x-x_n) + \frac 12 \|x-x_n\|^2_{\nabla^2 I_n(x_n)}
\end{equation}
as the second-order Taylor approximation $\tilde I_n = T_2 I_n(x_n)$ of $I_n$ at $x_n$.
This point of view is crucial for analyzing the approximation
\begin{equation*}
	\mu_n \approx \LAn,
	\qquad
	\frac 1{Z_n}
	\exp(-n I_n(x))
	\approx
	\frac1{\tilde Z_n}
	\exp\left( - n \tilde I_n(x)\right).
\end{equation*}

\paragraph{Notation and recurring equations.}
Before we continue, we collect recurring important definitions and where they can be found in Table \ref{tab:notation} and provide the following important equations cheat sheet
\begin{align*}
\mu_0(\mathrm d x) &= \pi_0(x) \mathrm d x&&\text{ (relative to Lebesgue measure)} \\
    \mu_n(\mathrm d x) &= Z_n^{-1}\exp(-n\Phi_n(x))\mu_0(\mathrm d x) &&\text{ (relative to $\mu_0$)}\\
        &= Z_n^{-1}\exp(-nI_n(x))\1_{S_0}(x)\mathrm d x &&\text{ (relative to Lebesgue measure)} \\
    \mathcal L_{\mu_n}(\mathrm d x) &= \tilde Z_n^{-1} \exp(-n T_2\Phi_n(x; x_n))\mu_0(\mathrm d x) &&\text{ (relative to $\mu_0$)}\\
        &= \tilde Z_n^{-1} \exp(-\frac{n}{2}\|x-x_n\|_{\nabla^2 I_n(x_n)}^2)\mathrm d x &&\text{ (relative to Lebesgue measure)} 
\end{align*}

\begin{table}[t]
\caption{Frequently used notation.}
\label{tab:notation}

\begin{tabular}{lll}
\hline\noalign{\smallskip}
Symbol &  Meaning & Reference\\
\noalign{\smallskip}\hline\noalign{\smallskip}
$\mu_0$ &  prior probability measure & \eqref{equ:nonGaussian_prior}\\
$\pi_0$ & Lebesgue density of the prior: $\mu_0(\mathrm d x) = \pi_0(x)\mathrm  d x$ 
& \eqref{equ:nonGaussian_prior}\\ 
$Z_0 = 1$ & hypothetically: Normalization constant of $\pi_0$&\\
\hline
$\mu_n$ &  posterior probability measure& \eqref{equ:post}, \eqref{equ:post_2}\\
$\Phi_n$ & (scaled) negative loglikelihood of posterior & \eqref{equ:post}\\ 
$I_n$ & (scaled) negative logdensity of posterior &  \eqref{equ:I_n}\\ 
$\pi_n(x)$ & \emph{unnormalized} Lebesgue density of the posterior & Lemma \ref{lem:conv_LA}\\
$Z_n$ & normalization constant of $\pi_n$ & \eqref{equ:post}\\
\hline
$\mathcal L_{\mu_n}$ & Laplace approximation of $\mu_n$ & \eqref{equ:Laplace2}\\
$\tilde I_n = T_2 I_n(x_n)$ & (scaled) negative logdensity of $\LAn$ &\eqref{equ:Laplace3} \\
$\tilde\pi_n$ &  \emph{unnormalized} Lebesgue density of $\LAn$ & Lemma \ref{lem:conv_LA}\\
$\tilde Z_n$ & normalization constant of $\tilde \pi_n$ & \eqref{equ:Laplace2}\\
\noalign{\smallskip}\hline
\end{tabular}
\end{table}

\subsection{Convergence in Hellinger Distance}
\label{sec:conv}
By a modification of Theorem \ref{theo:laplace_method} for integrals w.r.t.~a weight $\e^{-n\Phi_n(x)}$ we may show a corresponding version of \eqref{equ:Laplace_weak_C2}, i.e., for sufficiently smooth $f\in L^1_{\mu_0}(\bbR)$
\begin{equation}\label{equ:Laplace_weak_C2_LA}
	\left|\int_{\bbR^d} f(x)\  \mu_n(\d x) - \int_{\bbR^d} f(x)\  \LAn(\d x)\right|
	\in \mc O(n^{-1}).
\end{equation}
However, in this section we study a stronger notion of convergence of $\LAn$ to $\mu_n$, namely, w.r.t.~the \emph{total variation distance} $d_\text{TV}$ and the \emph{Hellinger distance} $d_\mathrm{H}$.
Given two probability measures $\mu$, $\nu$ on $\bbR^d$ and another probability measure $\rho$ dominating $\mu$ and $\nu$ the total variation distance of $\mu$ and $\nu$ is given by
\[
	d_\text{TV}(\mu,\nu)
	\coloneqq 
	\sup_{A \in \mc B(\bbR^d)}
	\left|\mu(A) - \nu(A)\right|
	=
	\frac 12 \int_{\bbR^d} \left|\frac{\d \mu}{\d \rho}(x) - \frac{\d \tilde \nu}{\d \rho}(x)\right| \rho(\d x)
\]	
and their Hellinger distance by
\[
	d^2_\text{H}(\mu,\nu)
	\coloneqq 
	\int_{\bbR^d} \left|\sqrt{\frac{\d \mu}{\d \rho}(x)} - \sqrt{\frac{\d \nu}{\d \rho}(x)}\right|^2 \rho(\d x).
\]
It holds true that
\[
	\frac{d^2_\text{H}(\mu,\nu)}2
	\leq
	d_\text{TV}(\mu,\nu)
	\leq
	d_\mathrm{H}(\mu,\nu),
\]
see \cite[Equation (8)]{GibbsSu2002}.
Note that, $d_\text{TV}(\mu_n,\LAn)\to 0$ implies that $|\int f  \d \mu_n - \int f \d \LAn| \to 0$ for any bounded and continuous $f\colon \bbR^d\to\bbR$.
In order to establish our convergence result, we require almost the same assumptions as in Theorem \ref{theo:laplace_method}, but now uniformly w.r.t.~$n$:

\begin{assumption}\label{assum:LA_Conv}
There holds $\Phi_n, \pi_0 \in C^3(\S_0,\bbR)$ for all $n$ and
\begin{enumerate}
\item \label{assum:LA_Conv:conv}
there exist the limits
\begin{equation}\label{equ:x_star}
	x_\star \coloneqq \lim_{n\to \infty} x_n
	\qquad
	H_\star \coloneqq \lim_{n\to \infty} H_n,
	\qquad
	H_n \coloneqq \nabla^2\Phi_n(x_n)
\end{equation}
in $\bbR^d$ and $\bbR^{d \times d}$, respectively, with $H_\star$ being positive definite and $x_\star$ belonging to the interior of $\S_0$.

\item
For each $r > 0$ there exists an $n_r\in\bbN$, $\delta_r>0$ and $K_r < \infty$ such that
\[
	\delta_r \leq \inf_{x \notin B_r(x_n)\cap \S_0} I_n(x) \qquad \forall n\geq n_r
\]
as well as 
\[
	\max_{x\in B_r(0) \cap \S_0} \|\nabla^3 \log \pi_0(x)\| \leq K_r,
	\max_{x\in B_r(0)\cap \S_0} \|\nabla^3 \Phi_n(x)\| \leq K_r \quad \forall n\geq n_r.
\]
\rev{
\item
There exists a uniformly bounding function $q\colon \SP_0\to[0,\infty)$ with
\[
	\exp(-n I_n(x)) \leq q(x),
	\qquad 
	\forall x \in \SP_0\ \forall n\geq n_0
\]
for an $n_0\in\bbN$ such that $q^{1-\epsilon}$ is integrable, i.e., $\int_{\SP_0} q^{1-\epsilon}(x)\ \d x < \infty$, for an $\epsilon\in(0,1)$. 
}
\end{enumerate}
\end{assumption}

The only additional assumptions in comparison to the classical convergence theorem of the Laplace method are about the third derivatives of $\pi_0$ and $\Phi_n$ and the convergence of $x_n \to x_\star$.
We remark that \eqref{equ:x_star} implies
\[
	\lim_{n\to \infty} C^{-1}_n 
	=
	\lim_{n\to \infty} 
	 \nabla^2\left( \Phi_n(x_n) - \frac 1n \log \pi_0(x_n)\right)
	= H_\star
\]
and, thus, also $\lim_{n\to \infty} C_n 	= H^{-1}_\star$.

\rev{The uniform lower bound on $I_n$ outside a ball around $x_n$ as well as the integrable majorant of the unnormalized densities $\e^{-nI_n}\leq 1$ of $\mu_n$ can be understood as uniform versions of the first and second assumption of Theorem \ref{theo:laplace_method}. 
The third item of Assumption \ref{assum:LA_Conv} implies the uniform integrability of the $\e^{-nI_n}\leq 1$ and is obviously satisfied for bounded supports $\SP_0$.
However, in the unbounded case it seems to be crucial\footnote{We provide a counterexample of nonconcentrating measures $\mu_n$ in the case, when the third issue of Assumptions \ref{assum:LA_Conv} does not hold, in Appendix C. We also state in Appendix C.1 sufficient conditions on $\Phi_n$ and $\pi_0$ such that the third issue of Assumptions \ref{assum:LA_Conv} is satisfied.}
for an increasing concentration of the $\mu_n$.}

We start our analysis with the following helpful lemma.

\begin{lemma}\label{lem:conv_LA}
Let Assumption \ref{assum:LA_0} and \ref{assum:LA_Conv} be satisfied and let $\pi_n, \tilde \pi_n \colon \bbR^d\to [0,\infty)$ denote the unnormalized Lebesgue densities of $\mu_n$ and $\LAn$, respectively, given by
\[
	\pi_n(x) \coloneqq \begin{cases} \exp\left(- n \Phi_n(x) \right) \pi_0(x), & x \in \S_0,\\ 0, & \text{otherwise}, \end{cases}
\]
and
\[
	\tilde \pi_n(x) \coloneqq \exp\left( -\frac n2 \|x-x_n\|^2_{C^{-1}_n}\right),
	\qquad
	x\in\bbR^d.	
\]
Then, for any $p\in\bbN$
\[
		\int_{\bbR^d}
		\left| \left( \frac{\pi_n(x)}{\tilde \pi_n(x)}\right)^{1/p} - 1\right|^p \LAn(\d x) 
		\in \mc O(n^{-p/2}).
\]
\end{lemma}
\begin{proof}
We define the remainder term 
\begin{align*}
	R_n(x) & \coloneqq I_n(x) - \tilde I_n(x)
	= I_n(x) - \frac 12 \|x-x_n\|^2_{C^{-1}_n},
\end{align*}
i.e., for $x\in \S_0$ we have $\frac{\pi_n(x)}{\tilde \pi_n(x)} = \exp(-nR_n(x))$. 
Moreover, note that for $x\in\S_0^c$ there holds $\pi_n(x) = 0$.
Thus, we obtain
\begin{align*}
	\int_{\bbR^d} \left| \left( \frac{\pi_n(x)}{\tilde \pi_n(x)}\right)^{1/p} - 1\right|^p \LAn(\d x) 
	& = \int_{\S_0^c} 1^p\ \LAn(\d x) \\
	& \qquad + \int_{\S_0} \left| \e^{-nR_n(x)/p} - 1\right|^p\ \LAn(\d x)\\
	& = J_0(n) + J_1(n) + J_2(n)
\end{align*}
where we define for a given radius $r>0$
\begin{align*}
	J_0(n) & \coloneqq\LAn(\S_0^c),\\
	J_1(n) & \coloneqq\int_{B_r(x_n) \cap \S_0} \left|\e^{-nR_n(x)/p} -1 \right|^{p}\ \LAn(\d x),\\
	J_2(n) & \coloneqq \int_{B^c_r(x_n)\cap \S_0} \left|\e^{-n R_n(x)/p} -1 \right|^p\ \LAn(\d x).
\end{align*}
\rev{In Appendix \ref{sec:proof_lem} we prove that
\[
	J_0(n) \in \mc O(\e^{-c_r n}),
	\qquad
	J_1(n) \in \mc O(n^{-p/2}),
	\qquad
	J_2(n) \in \mc O(\e^{- n c_{r,\epsilon}} n^{d/2} ),
\]
for $c_r, c_{r,\epsilon} >0$,} which then yields the statement.
\end{proof}

Lemma \ref{lem:conv_LA} provides the basis for our main convergence theorem.

\begin{theorem}\label{theo:conv_H}
Let the assumptions of Lemma \ref{lem:conv_LA} be satisfied.
Then, there holds
\[
		d_\mathrm{H}(\mu_n, \LAn) \in \mc O(n^{-1/2}).
\]
\end{theorem}
\begin{proof}
We start with
\begin{align*}
	d^2_\text{H}(\mu_n, \LAn)
	& = \int_{\bbR^d} \left[\frac{\sqrt{\pi_n(x)}}{\sqrt{Z_n}} - \frac{\sqrt{\tilde \pi_n(x)}}{\sqrt{\tilde Z_n}} \right]^2\ \d x\\
	& \leq \frac 2{\tilde Z_n} \int_{\bbR^d} \left[\sqrt{\pi_n(x)} - \sqrt{\tilde \pi_n(x)} \right]^2 \d x + 
	2 \left(\frac{1}{\sqrt{Z_n}} - \frac{1}{\sqrt{\tilde Z_n}}\right)^2  Z_n\\
	& = 2 \int_{\bbR^d} \left[ \sqrt{\frac{\pi_n(x)}{\tilde \pi_n(x)}} - 1\right]^2\ \LAn(\d x) + 
	2 \left(\frac{1}{\sqrt{Z_n}} - \frac{1}{\sqrt{\tilde Z_n}}\right)^2  Z_n.
\end{align*}
For the first term there holds due to Lemma \ref{lem:conv_LA}
\begin{align*}
	\int_{\bbR^d} \left[ \sqrt{\frac{\pi_n(x)}{\tilde \pi_n(x)}} - 1\right]^2\ \LAn(\d x) 
	\in \mc O(n^{-1}).
\end{align*}
For the second term on the right-hand side we obtain
\begin{align*}
	2\left(\frac{1}{\sqrt{Z_n}} - \frac{1}{\sqrt{\tilde Z_n}}\right)^2 Z_n 
	& = \frac{2}{\tilde Z_n} \left(\sqrt{Z_n} - \sqrt{\tilde Z_n}\right)^2
	= \frac 2{\tilde Z_n} \left(\frac{Z_n - \tilde Z_n}{\sqrt{Z_n} + \sqrt{\tilde Z_n}}\right)^2\\
	&\leq 2\frac {\left|Z_n - \tilde Z_n\right|^2}{\tilde Z^2_n}.
\end{align*}
Furthermore, due to Lemma \ref{lem:conv_LA} there exists a $c<\infty$ such that
\begin{align*}
	|Z_n - \tilde Z_n| & \leq \int_{\bbR^d} \left| \pi_n(x) - \tilde \pi_n(x)\right|\ \d x 
	= \tilde Z_n \int_{\bbR^d} \left| \frac{\pi_n(x)}{\tilde \pi_n(x)} - 1 \right|\ \LAn(\d x)\\
	& \leq c n^{-1/2} \tilde Z_n.
\end{align*}
This yields
\[
	2\left(\frac{1}{\sqrt{Z_n}} - \frac{1}{\sqrt{\tilde Z_n}}\right)^2 Z_n
	\leq 
	2\frac {\left|Z_n - \tilde Z_n\right|^2}{\tilde Z^2_n}
	\leq
	2c^2 n^{-1}
	\in \mc O(n^{-1}),
\]
which concludes the proof.
\end{proof}

\paragraph{Convergence of other Gaussian approximations.}
Let us consider now a sequence of arbitrary Gaussian approximations $\tilde \mu_n = \mc N(a_n, \frac 1n B_n)$ to the measures $\mu_n$ in \eqref{equ:post}. 
Under which conditions on $a_n \in\bbR^d$ and $B_n \in \bbR^{d\times d}$ do we still obtain the convergence $d_\text{H}(\mu_n, \tilde \mu_n) \to 0$? 
Of course, $a_n\to x_\star$ seems to be necessary but how about the covariances $B_n$?
Due to the particular scaling of $1/n$ appearing in the covariance of $\LAn$, one might suppose that for example $\tilde \mu_n = \mc N(x_n, \frac 1n I_d)$ or $\tilde \mu_n = \mc N(x_n, \frac 1n B)$ with an arbitrary symmetric and positive definite (spd) $B \in \bbR^{d\times d}$ should converge to $\mu_n$ as $n\to\infty$.
However, since
\[
	\left| d_\text{H}(\mu_n, \LAn) - d_\text{H}(\LAn, \tilde \mu_n)\right|
	\leq  
	d_\text{H}(\mu_n, \tilde \mu_n)
	\leq
	d_\text{H}(\mu_n, \LAn) + d_\text{H}(\LAn, \tilde \mu_n)
\]
and $d_\text{H}(\mu_n, \LAn) \to 0$, we have
\begin{equation}\label{equ:Conv_mun_2}
	d_\text{H}(\mu_n, \tilde \mu_n) \to 0
	\quad
	\text{ iff }
	\quad 	
	d_\text{H}(\LAn, \tilde \mu_n) \to 0.
\end{equation}
The following result shows that, in general, $\tilde \mu_n = \mc N(x_n, \frac 1n I_d)$ or $\tilde \mu_n = \mc N(x_n, \frac 1n B)$ do not converge to $\mu_n$. 

\begin{theorem}\label{theo:conv_H_star}
Let the assumptions of Lemma \ref{lem:conv_LA} be satisfied.
\begin{enumerate}
\item
For $\tilde \mu_n \coloneqq \mc N(x_n, \frac 1n B_n)$, $n\in\bbN$, with spd $B_n$, we have that 
\begin{equation}\label{equ:Bn_cond}
		\lim_{n\to \infty} d_\mathrm{H}(\mu_n, \tilde \mu_n) = 0
		\quad
		\text{ iff }
		\quad
		\lim_{n\to \infty} \det\left(\frac 12 (H^{1/2}_\star B_n^{1/2} + H_\star^{-1/2} B_n^{-1/2})\right) = 1.	
\end{equation}
If so and if $\|C_n - B_n\| \in \mc O(n^{-1})$, then we even have $d_\mathrm{H}(\mu_n, \tilde \mu_n) \in \mc O(n^{-1/2})$.

\item
For $\tilde \mu_n \coloneqq \mc N(a_n, \frac 1n B_n)$, $n\in\bbN$, with $B_n$ satisfying \eqref{equ:Bn_cond} and $\|x_n - a_n\| \in \mc O(n^{-1})$, we have that $d_\mathrm{H}(\mu_n, {\tilde \mu_n}) \in \mc O(n^{-1/2})$.
\end{enumerate}
\end{theorem}

The proof is straightforward given the exact formula for the Hellinger distance of Gaussian measures and can be found in Appendix \ref{sec:conv_H_star_proof}. 
Thus, Theorem \ref{theo:conv_H_star} tells us that, in general, the Gaussian measures $\tilde \mu_n = \mc N(x_n, \frac 1n I_d)$ do not converge to $\mu_n$ as $n\to \infty$ whereas it is easily seen that $\tilde \mu_n = \mc N(x_n, \frac 1n H_\star)$, indeed, do converge.\\

\paragraph{Relation to the Bernstein--von Mises theorem in Bayesian inference.}
The Bernstein--von Mises (BvM) theorem is a classical result in Bayesian inference and asymptotic statistics in $\bbR^d$ stating the posterior consistency under mild assumptions \cite{VanDerVaart1998}.
Its extension to infinite-dimensional situations does not hold in general \cite{DiaconisFreedman1986,Freedman1999}, but can be shown under additional assumptions \cite{GhoshalEtAl2000,CastilloNickl2013,CastilloNickl2014,Nickl2017}.
In order to state the theorem we introduce the following setting: let $Y_i \sim \nu_{x_0}$, $i\in\bbN$, be i.i.d.~random variables on $\bbR^D$, $d\leq D$, following a distribution $\nu_{x_0}(\d y) = \exp(-\ell(y, x_0)) \boldsymbol 1_{\S_y}(y) \d y$ where $\S_y \subset \bbR^D$ and where $\ell\colon \S_y\times \bbR^d \to [- \ell_{\min}, \infty)$ represents the negative log-likelihood function for observing $y \in \S_y$ given a parameter value $x \in \bbR^d$.
Assuming a prior measure $\mu_0(\d x) = \pi_0(x) \boldsymbol 1_{\S_0}(x) \ \d x$ for the unknown parameter, the resulting posterior after $n$ observations $y_i$ of the independent $Y_i$, $i=1,\ldots,n$, is of the form \eqref{equ:post} with
\begin{equation}\label{equ:Phi_large_data}
	\Phi_n(x) = \Phi_n(x; y_1,\ldots, y_n) = \frac 1n \sum_{i=1}^n \ell(y_i,x).
\end{equation}
We will denote the corresponding posterior measure by $\mu_n^{y_1,\ldots,y_n}$ in order to highlight the dependence of the particular data $y_1,\ldots,y_n$.
The BvM theorem states now the convergence of this posterior to a sequence of Gaussian measures. 
This looks very similar to the statement of Theorem \ref{theo:conv_H}.
However, the difference lies in the Gaussian measures as well as the kind of convergence.
In its usual form the BvM theorem states under {similar assumptions as for Theorem \ref{theo:conv_H}} that there holds in the large data limit 
\begin{equation}\label{equ:BvM}
	d_\text{TV}\left(\mu^{Y_1,\ldots, Y_n}_n,\  \mc N(\hat x_n, n^{-1} \mc I^{-1}_{x_0}) \right) \xrightarrow[n\to \infty]{\bbP} 0
\end{equation}
where $\mu^{Y_1,\ldots, Y_n}_n$ is now a random measure depending on the $n$ independent random variables $Y_1,\ldots, Y_n$ and where the convergence in probability is taken w.r.t.~randomness of the $Y_i$.
Moreover, $\hat x_n = \hat x_n(Y_1,\ldots, Y_n)$ denotes an efficient estimator of the true parameter $x_0 \in \S_0$---e.g., the maximum-likelihood or MAP estimator---and $\mc I_{x_0}$ denotes the Fisher information at the true parameter $x_0$, i.e.,
\[
	\mc I_{x_0}
	=
	\ev{ \nabla^2_x \ell(Y_i, x_0)}
	=
	\int_{\bbR^D}
	\nabla^2_x \ell(Y, x_0)
	\
	\exp(- \ell(y, x_0) )\ 
	\d y.
\]
Now both, the BvM theorem and Theorem \ref{theo:conv_H}, state the convergence of the posterior to a concentrating Gaussian measure where the rate of concentration of the latter (or better: of its covariance) is of order $n^{-1}$.
Furthermore, also the rate of convergence in the BvM theorem can be shown to be of order $n^{-1/2}$ \cite{HippMichel1976}.
However, the main differences are: 
\begin{itemize}
\item The BvM states convergence in probability (w.r.t.~the randomness of the $Y_i$) and takes as basic covariance the inverse expected Hessian of the negative log likelihood at the data generating parameter value $x_0$. Working with this quantity requires the knowledge of the true value $x_0$ and the covariance operator is obtained by marginalizing over all possible data outcomes $Y$. This Gaussian measure is not a practical tool to be used but rather a limiting distribution of a powerful theoretical result reconciling Bayesian and classical statistical theory. For this reason, the Gaussian approximation in the statement of the BvM theorem can be thought of as being a \emph{``prior''} approximation (in the loosest meaning of the word).
Usually, a crucial requirement is that the problem is \emph{well-specified} meaning that $x_0$ is an interior point of the prior support $\S_0$---although there exist results for misspecified models, see \cite{KleijnVanDerVaart2012}. Here, a BvM theorem is proven without the assumption that $x_0$ belongs to the interior of $\S_0$.
However, in this case the basic covariance is not the Fisher information but the Hessian of the mapping $x \mapsto d_\text{KL}(\nu_0 || \nu_x)$ evaluated at its unique minimizer where $d_\text{KL}(\nu_0 || \nu_x)$ denotes the Kullback-Leibler divergence of the data distribution $\nu_x$ given parameter $x \in \S_0$ w.r.t.~the true data distribution $\nu_0$.
 
\item
Theorem \ref{theo:conv_H} states the convergence for given realizations $y_i$ and takes the Hessian of the negative log posterior density evaluated at the current MAP estimate $x_n$ and the current data $y_1,\ldots,y_n$. This means that we do not need to know the true parameter value $x_0$ and we employ the actual data realization at hand rather than averaging over all outcomes. Hence, we argue that the Laplace approximation (as stated in this context) provides a \emph{``posterior''} {approximation converging to the Bayesian posterior as $n\to \infty$}.

Also, we require that the limit $x_\star = \lim_{n\to\infty} x_n$ is an interior point of the prior support $\S_0$.

\rev{\item From a numerical point of view, the Laplace approximation requires the computation of the MAP estimate and the corresponding Hessian at the MAP, whereas the BvM theorem employs the Fisher information, i.e.~requires an expectation w.r.t.~the observable data. 
Thus, the Laplace approximation is based on fixed and finite data in contrast to the BvM.}
\end{itemize}

The following example illustrates the difference between the two Gaussian measures: Let $x_0\in \mathbb R$ be an unknown parameter. {
Consider $n$ measurements $y_k \in \bbR$, $k=1,\ldots,n$, where $y_k$ is a realization of
\[
	Y_k = x_0^3 + \epsilon_k
\]
with $\epsilon_k\sim N(0, \sigma^2)$ i.i.d.}.
For the Bayesian inference we assume a prior $N(0,\tau^2)$ on $x$. Then the Bayesian posterior is of the form $\mu_n(\d x) \propto \exp(-nI_n(x))$ where
\[
	I_n(x) = \frac{x^2}{n\cdot 2\tau^2} +  \underbrace{\frac{1}{n\cdot 2\sigma^2}\sum_{k=1}^n ({y_k} - x^3)^2}_{=\Phi_n(x)}. 
\]
The MAP estimator $x_n$ is the Laplace approximation's mean and can be computed numerically as a minimizer of $I_n(x)$.
{It can be shown that $x_n$ converges to $x_\star = x_0$ for almost surely all realizations $y_k $ of $Y_k$ due to the strong law of large numbers.}
Now we take the Hessian ({w.r.t.~$x$}) of $I_n$,
\[
	\nabla^2I_n(x) = \frac{1}{n\cdot\tau^2} + \frac{15}{\sigma^2}\cdot x^4 - 6x\cdot \frac{1}{n\cdot \sigma^2}\sum_{k=1}^n y_k
\]
 and evaluate it in $x_n$ to obtain the covariance of the Laplace approximation, and, thus,
\[
	\LAn = \mathcal N\left(x_n, \frac{1}{\frac{1}{n\cdot\tau^2} + \frac{15}{\sigma^2}\cdot x_n^4 - 6x_n\cdot \frac{1}{n\cdot \sigma^2}\sum_{k=1}^n {y_k}}\right).  
\]
On the other hand we compute the Gaussian {BvM} approximation: 
{The} Fisher information is given as (recall that $\Phi$ is the loglikelihood term as defined above)
\[
	{\mathbb E}^{x_0} [\nabla_x^2 \Phi(x_0)] 
	= \mathbb E \left[ \frac{15}{\sigma^2}\cdot x_0^4 - 6x_0\cdot \frac{1}{n\cdot \sigma^2}\sum_{k=1}^n Y_k \right] 
	= \frac{15}{\sigma^2}\cdot x_0^4  - 6x_0\cdot \frac{1}{\sigma^2} x_0^3 
	= \frac{9}{\sigma^2}x_0^4
\]
and hence we get the Gaussian approximation
\[ 
	\mu_{\text{BVM}} = \mathcal N\left(x_n, \frac{\sigma^2}{9\cdot x_0^4}\right).
\]
Now we clearly see the difference between the two measures and how they will be asymptotically identical, since $x_n\to x_\star = x_0$ due to consistency, $\frac{1}{n}\sum_{k=1}^n y_k$ converging a.s.~to $x_0^3$ due to the strong law of large numbers, and with the prior-dependent part vanishing for $n\to 0$.

\begin{remark}
Having raised the issue whether the BvM approximation $\mc N(\hat x_n, n^{-1} \mc I^{-1}_{x_0})$ or the Laplace one $\LAn$ is closer to a given posterior $\mu_n$, one can of course ask for the best Gaussian approximation of $\mu_n$ w.r.t.~a certain distance or divergence.
Thus, we mention \cite{LuEtAl2017,PinskiEtAl2013} where such  a best approximation w.r.t.~the Kullback-Leibler divergence is considered.
The authors also treat the case of best Gaussian mixture approximations for multimodal distributions and state a BvM like convergence result for the large data (and small noise) limit.
However, the computation of such a best approximation can become costly whereas the Laplace approximation can be obtained rather cheaply.
\end{remark}

\subsection{The case of singular Hessians}  
The assumption, that the Hessians $H_n = \nabla^2 \Phi_n(x_n)$ as well as their limit $H_\star$ are positive definite, is quite restrictive.
For example, for Bayesian inference with more unknown parameters than observational information, this assumption is not satisfied.
Hence, we discuss in this subsection the convergence of the Laplace approximation in case of singular Hessians $H_n$ and $H_\star$.
\rev{Nonetheless, we assume throughout the section that Assumption \ref{assum:LA_0} is satisfied. This yields that the Laplace approximation $\LAn$ is well-defined.}
This means in particular that we suppose a regularizing effect of the log prior density $\log \pi_0$ on the minimization of $I_n(x) = \Phi_n(x) - \frac 1n \log \pi_0(x)$.

We first discuss necessary conditions for the convergence of the Laplace approximation and subsequently state a positive result for Gaussian \rev{prior} measures $\mu_0$.

\paragraph{Necessary conditions.}
Let us consider the simple case of $\Phi_n\equiv \Phi$, i.e., the probability measures $\mu_n$ are given by
\[
	\mu_n(\d x) \propto \exp\left(-n \Phi(x)\right)\, \mu_0(\d x),
\]
where we assume now that $\Phi \colon \S_0 \to [c, \infty)$ with $c>-\infty$.
Intuitively, $\mu_n$ should converge weakly to the Dirac measure $\delta_{\mc M_\Phi}$ on the \rev{set}
\[
	\mc M_\Phi \coloneqq \argmin_{x\in\S_0} \Phi(x).
\]
On the other hand, the associated Laplace approximations $\LAn$ will converge weakly to the Dirac measure $\delta_{\mc M_{\mc L}}$ \rev{in the affine subspace}
\[
	\mc M_{\mc L} \coloneqq \{x \in \bbR^d\colon (x-x_\star)^\top H_\star (x-x_\star) = 0\}.
\]
Hence, it is necessary for the convergence $\LAn \to \mu_n$ in total variation or Hellinger distance that $\mc M_\Phi = \mc M_{\mc L}$, i.e., that the \rev{set} of minimizers of $\Phi$ is linear. 
In order to ensure the latter, we state the following.

\begin{assumption}\label{assum:LA_Conv_linear_subspace}
Let $\mc X\subseteq \bbR^d$ be a linear subspace such that for a projection $\P_{\mc X}$ onto $\mc X$ there holds

\[
	 \Phi_n \equiv \Phi_n \circ \P_{\mc X} \qquad \text{ on } \S_0 \text{ for each } n\in\bbN
\]
and let the restriction $\Phi_n \colon \mc X \to \bbR$ possess a unique and nondegenerate global minimum for each $n\in\bbN$.
\end{assumption}
For the case $\Phi_n = \Phi$ this assumption implies, that
\[
	\mc M_{\Phi} = \argmin_{x\in\S_0} \Phi(x) = x_\star + \mc X^c
\]
where $\mc X^c$ denotes a complementary subspace to $\mc X$, i.e., $\mc X \oplus \mc X^c = \bbR^d$ and $x_\star\in \mc X$ {the unique minimizer of $\Phi$ over $\mc X$}.
Besides that, Assumption \ref{assum:LA_Conv_linear_subspace} also yields that $x^\top H_n x = 0$ iff $x \in \mc X^c$.
Hence, this also holds for the limit $H_\star = \lim_{n\to\infty} H_n$ and we obtain 
\[
	\mc M_{\mc L} = x_{\star} + \mc X^c = \mc M_{\Phi}.
\]
Moreover, since Assumption \ref{assum:LA_Conv_linear_subspace} yields
\[
	\mu_n(\d x) \propto \exp \left( - n \Phi_n(x_{\mc X})\right) \mu_0(\d x_{\mc X} \d x_{c}),
\]
where $x_{\mc X} \coloneqq \P_{\mc X}x$ and $x_{c} \coloneqq \P_{\mc X^c}x = x - x_{\mc X}$, the marginal of $\mu_n$ coincides with the marginal of $\mu_0$ on $\mc X^c$.
Hence, the Laplace approximation can only converge to $\mu_n$ \rev{in total variation or Hellinger distance} if this marginal is Gaussian.
We, therefore, consider the special case of Gaussian \rev{prior} measures $\mu_0$. 

\begin{remark}
\rev{Please note that, despite this to some extent negative result for the Laplace approximation for singular Hessians, the preconditioning of sampling and quadrature methods via the Laplace approximation may still lead to efficient algorithms in the small noise setting. 
The analysis of Laplace approximation-based sampling methods, as introduced in the next section, in the underdetermined case will be subject to future work.}
\end{remark}

\paragraph{Convergence for Gaussian \rev{prior} $\mu_0$.}

A useful feature of Gaussian \rev{prior} measures $\mu_0$ is that the Laplace approximation possesses a convenient representation via its density w.r.t.~$\mu_0$. 

\begin{proposition}[{cf. \cite[Proposition 1]{Wacker2017}}] \label{propo:Laplace_Gaussian_Prior}
Let Assumption \ref{assum:LA_0} be satisfied and $\mu_0$ be Gaussian. 
Then there holds
\begin{equation}\label{eq:LA_RadonNikodym}
	\frac{\d \LAn}{\d \mu_0}(x)
	\propto
	\exp(- n T_2\Phi_n(x; x_n)),
	\qquad
	x \in \bbR^d,
\end{equation}
where $T_2\Phi_n(\cdot; x_n)$ denotes the Taylor polynomial of order 2 of $\Phi_n$ at the point $x_n\in\bbR^d$.
\end{proposition}
In fact, the representation \eqref{eq:LA_RadonNikodym} does only hold for \rev{prior} measures $\mu_0$ with Lebesgue density $\pi_0\colon \bbR^d \to \bbR^{d\times d}$ satisfying $\nabla^3 \log \pi_0 \equiv 0$.

\begin{corollary}\label{cor:conv_H_sing}
Let Assumption \ref{assum:LA_0} be satisfied and $\mu_0$ be Gaussian. 
Further, let Assumption \ref{assum:LA_Conv_linear_subspace} hold true and assume that the restriction $\Phi_n \colon \mc X \to \bbR$ and the marginal density $\pi_0$ on $\mc X$ satisfy Assumption \ref{assum:LA_Conv} on $\mc X$.
Then the approximation result of Theorem \ref{theo:conv_H} holds.
\end{corollary}
\begin{proof}
By using Proposition \ref{propo:Laplace_Gaussian_Prior}, we can express the Hellinger distance $d_\text{H}(\mu_n,\LAn)$ as follows
\begin{align*}
	d^2_\text{H}(\mu_n,\LAn)  
	& = \int_{\bbR^{d}} \left(\sqrt{\frac{\d \mu_n}{\d \mu_0}(x)} - \sqrt{\frac{\d \LAn}{\d \mu_0}(x)} \right)^2 \mu_0(\d x)\\
	& = \int_{\bbR^{d}} \left(\sqrt{\frac{\exp(-n \Phi_n(x))}{Z_n}} - \sqrt{\frac{\exp(-n T_2\Phi_n(x;x_n))}{\tilde Z_n}} \right)^2 \mu_0(\d x).
\end{align*}
We use now the decomposition $\bbR^d = \mc X \oplus \mc X^c$ with $x  \coloneqq x_{\mc X} + x_{c}$ for $x\in\bbR^d$ with $x_{\mc X} \in \mc X$ and $x_c \in \mc X^c$.
We note, that due to Assumption \ref{assum:LA_Conv_linear_subspace}, we have that
\[
	T_2\Phi_n(x;x_n) = T_2\Phi_n(x_{\mc X};x_n), \qquad x \in \bbR^d.
\]
We then obtain by disintegration and denoting $\tilde \Phi_n(x) \coloneqq T_2\Phi_n(x;x_n) = \tilde \Phi_n(x_{\mc X})$
\begin{align*}
	d^2_\text{H}(\mu_n,\LAn)  
	& = \int_{\bbR^{d}} \left(\sqrt{\frac{\e^{-n \Phi_n(x_{\mc X})}}{Z_n}} - \sqrt{\frac{\e^{-n \tilde \Phi_n(x_{\mc X})}}{\tilde Z_n}} \right)^2 \mu_0(\d x_{\mc X} \d x_c)\\
	& = \int_{\mc X} \int_{\mc X^c} \left(\sqrt{\frac{\e^{-n \Phi_n(x_{\mc X})}}{Z_n}} - \sqrt{\frac{\e^{-n \tilde \Phi_n(x_{\mc X})}}{\tilde Z_n}} \right)^2 \mu_0(\d x_{c} | x_{\mc X}) \ \mu_0(\d x_{\mc X})\\
	& = \int_{\mc X} \left(\sqrt{\frac{\e^{-n \Phi_n(x_{\mc X})}}{Z_n}} - \sqrt{\frac{\e^{-n \tilde \Phi_n(x_{\mc X})}}{\tilde Z_n}} \right)^2 \mu_0(\d x_{\mc X}),
\end{align*}
where $\mu_0(\d x_{\mc X})$ denotes the marginal of $\mu_0$ on $\mc X$. 
Since $\Phi_n$ and $I_n(x_{\mc X}) = \Phi_n(x_{\mc X}) - \frac 1n \log \pi_0(x_{\mc X})$, where $\pi_0(x_{\mc X})$ denotes the Lebesgue density of the marginal $\mu_0(\d x_{\mc X})$, satisfy the assumptions of Theorem \ref{theo:conv_H} on $\S_0 \cap \mc X = \mc X$, the statement follows.
\end{proof}

We provide some illustrative examples for the theoretical results stated in this subsection. 

\begin{example}[Divergence of the Laplace approximation in the singular case]\label{exam:2D_1}
We assume a Gaussian prior $\mu_0 = N(0, I_2)$ on $\bbR^2$ and $\Phi(x) = \|y - F(x)\|^2$ where
\begin{equation}\label{eq:ex_2D_2}
	y = 0,
	\qquad
	F(x) = x_2-x_1^2,
	\qquad
	x = (x_1,x_2) \in \bbR^2.
\end{equation}
We plot the Lebesgue densities of the resulting $\mu_n$ and $\LAn$ for $n=128$ in the left and middle panel of Figure \ref{fig:exam4_Phi_error}. 
The red line in both plots indicate the different \rev{sets}
\[
	\mc M_\Phi = \{x \in \bbR^2\colon x_2 = x_1^2\},
	\qquad
	\mc M_{\mc L} = \{x \in \bbR^2\colon x_2 = 0\},
\]
around which $\mu_n$ and $\LAn$, respectively, concentrate as $n\to \infty$.
As $\mc M_\Phi  \neq \mc M_{\mc L}$, we observe no convergence of the Laplace approximation as $n\to \infty$, see the right panel of Figure \ref{fig:exam4_Phi_error}.
Here, the Hellinger distance is computed numerically by applying a tensorized trapezoidal rule on a suffieciently large subdomain of $\bbR^2$.

\begin{figure}[htb]
\includegraphics[width = 0.32\textwidth]{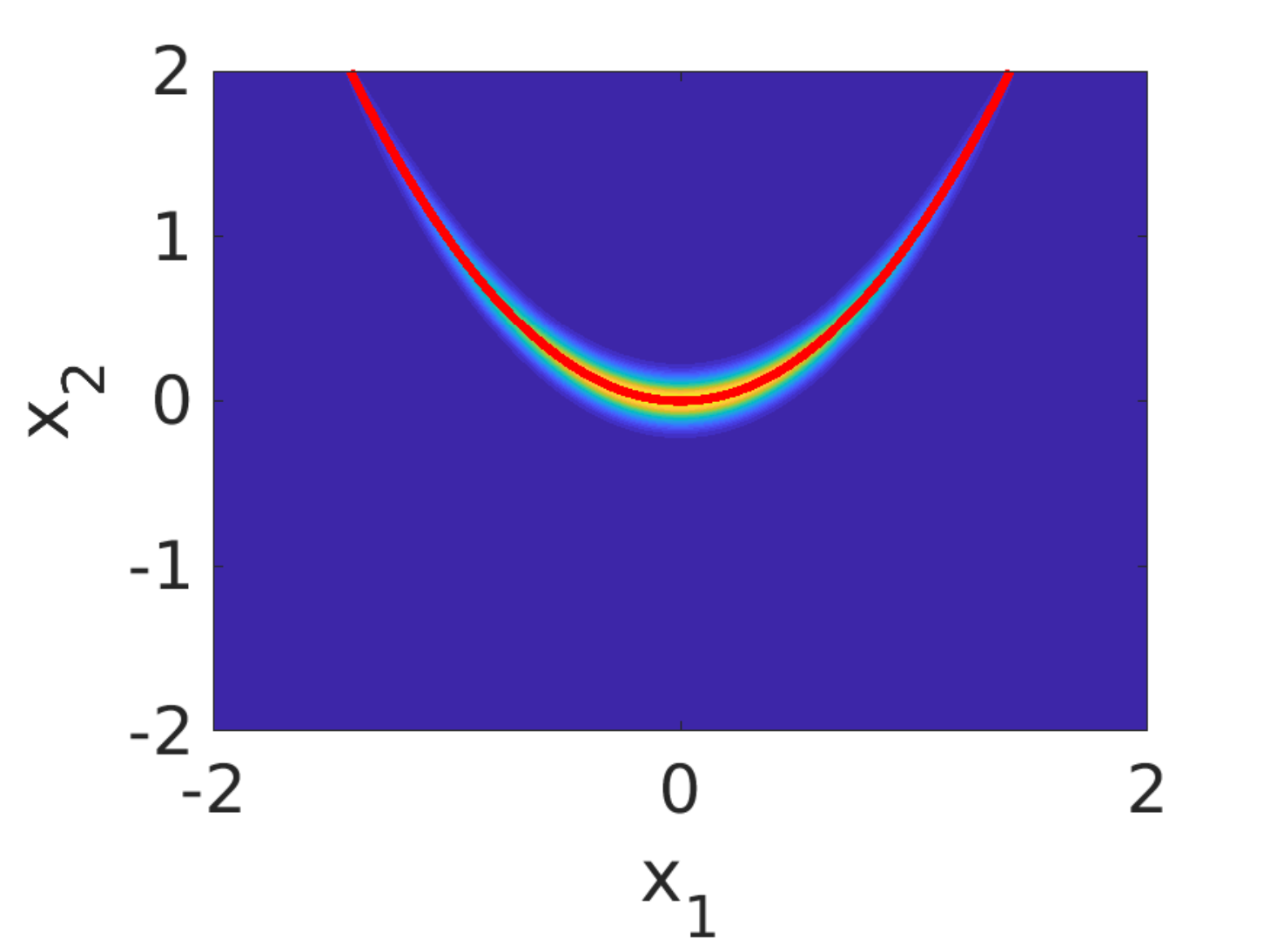}
\hfill
\includegraphics[width = 0.32\textwidth]{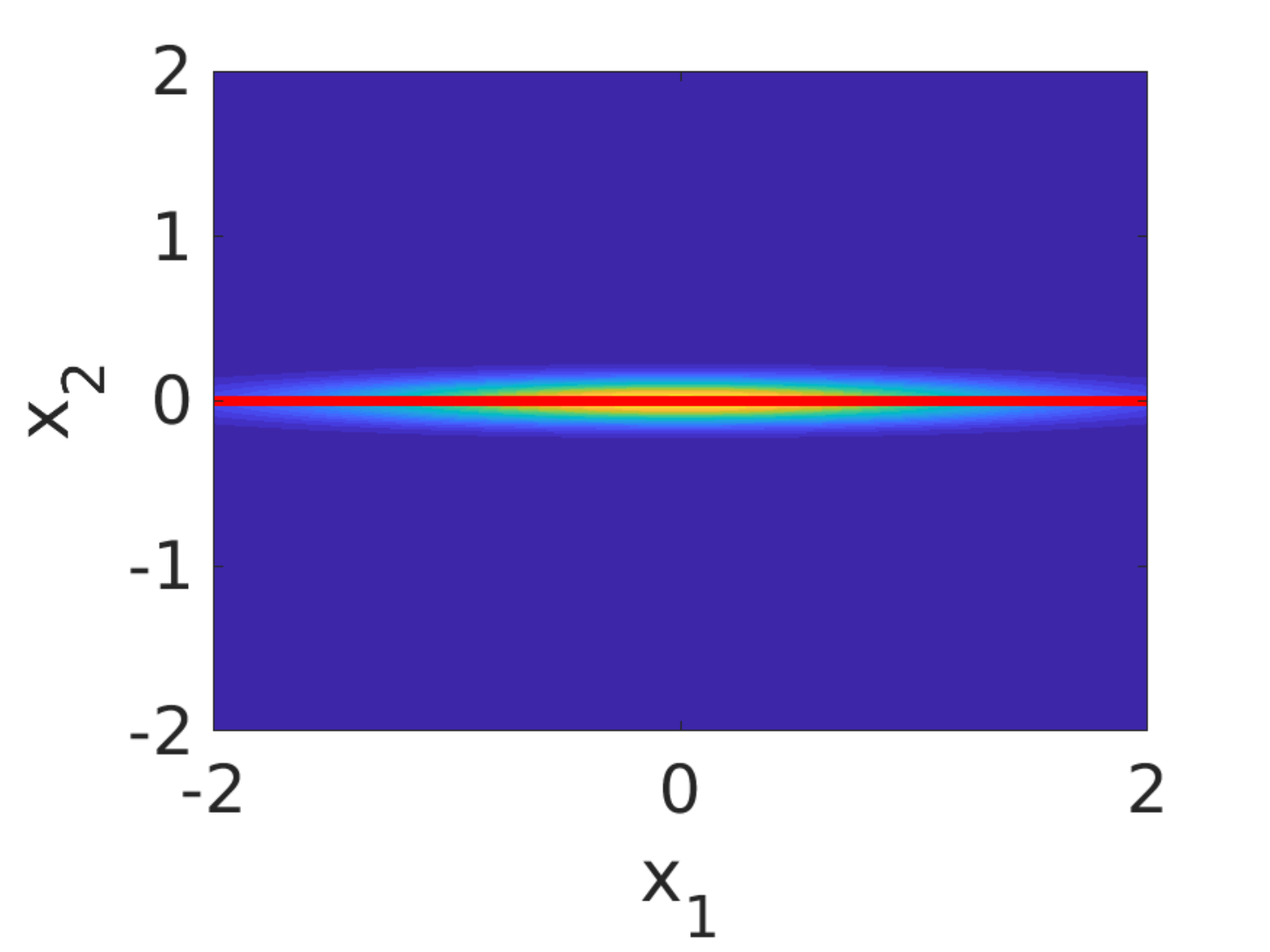}
\hfill
\includegraphics[width = 0.32\textwidth]{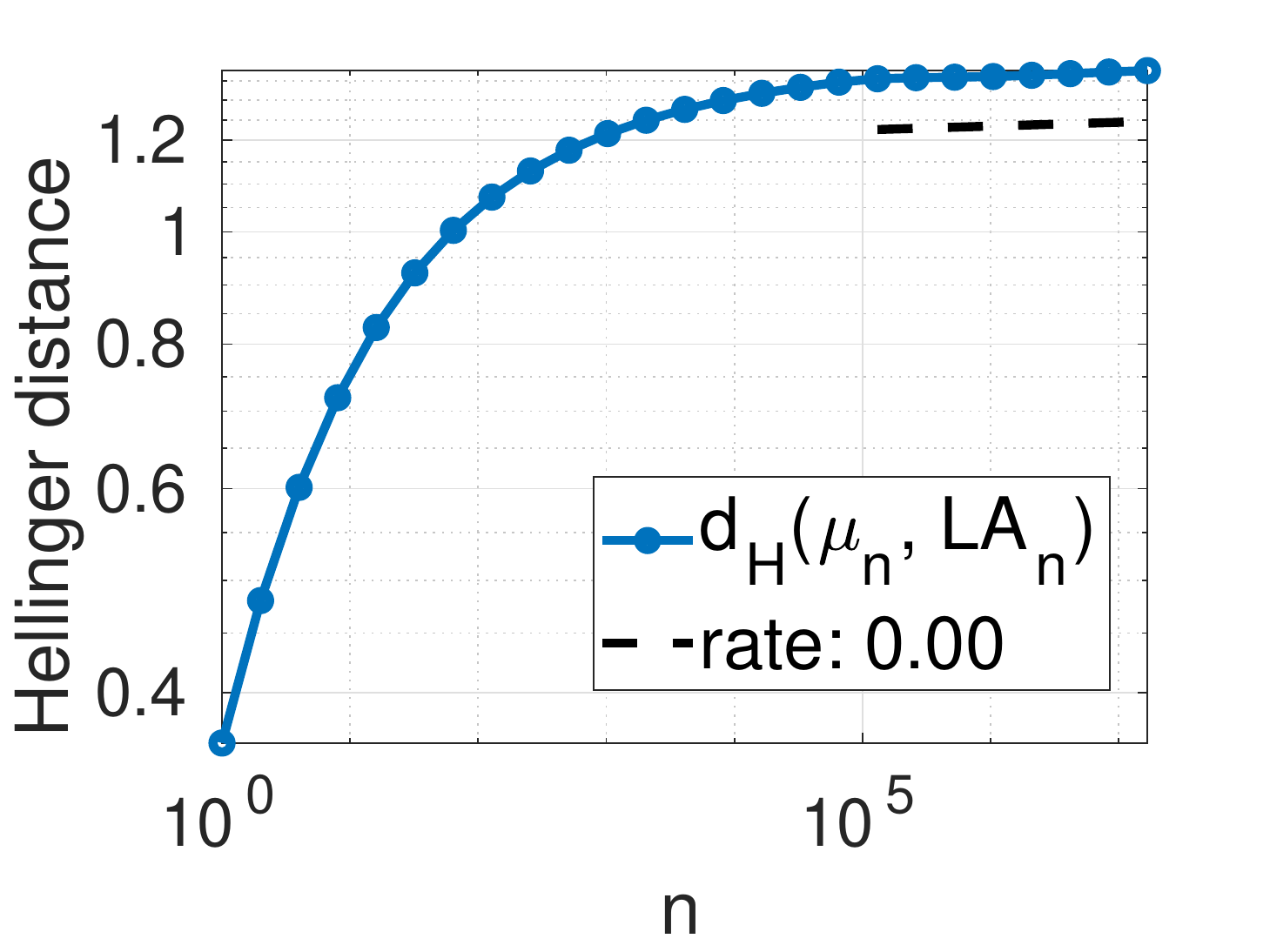}

\caption{Plots of the Lebesgue densities of $\mu_n$ (left) and $\LAn$ (middle) for $n=128$ as well as the Hellinger distance between $\mu_n$ and $\LAn$ for Example \ref{exam:2D_1}. 
The red line in the left and middle panel represents the \rev{set} $\mc M_\Phi$ and $\mc M_{\mc L}$ around which $\mu_n$ and $\LAn$, respectively, concentrate as $n\to \infty$.
}
\label{fig:exam4_Phi_error}
\end{figure}
\end{example}

\begin{example}[Convergence of the Laplace approximation in the singular case in the setting of Corollary \ref{cor:conv_H_sing}]\label{exam:2D_2}
Again, we suppose a Gaussain prior $\mu_0 = N(0, I_2)$ and $\Phi$ in the form of {$\Phi(x) = \|y - F(x)\|^2$} with 
\begin{equation}\label{eq:ex_2D_3}
	y = \begin{pmatrix} \frac \pi2\\ 0.5\end{pmatrix},
	\qquad
	F(x) = \begin{pmatrix} \exp((x_2-x_1)/5)\\ \sin(x_2-x_1) \end{pmatrix},
	\qquad
	x = (x_1,x_2) \in \bbR^2.
\end{equation}
Thus, the invariant subspace is $\mc X^c = \{x \in\bbR^2 \colon x_1 = x_2\}$.
In the left and middle panel of Figure \ref{fig:exam5_Phi_error} we present the Lebesgue densities of $\mu_n$ and its Laplace approximation $\LAn$ for $n=25$ and by the red line the \rev{sets} $\mc M_\Phi = \mc M_{\mc L} = x_\star + \mc X^c$. 
We observe the convergence guaranteed by Corollary \ref{cor:conv_H_sing} in the right panel of Figure \ref{fig:exam5_Phi_error} where we can also notice a preasymptotic phase with a shortly increasing Hellinger distance.
Such a preasmyptotic phase is to be expected due to $d_\text{H}(\mu_n, \LAn) \in \mc O(n^{-1/2}) + \mc O(\e^{- n \delta_r} n^{d/2})$ as shown in the proof of Theorem \ref{theo:conv_H}.

\begin{figure}[htb]
\includegraphics[width = 0.32\textwidth]{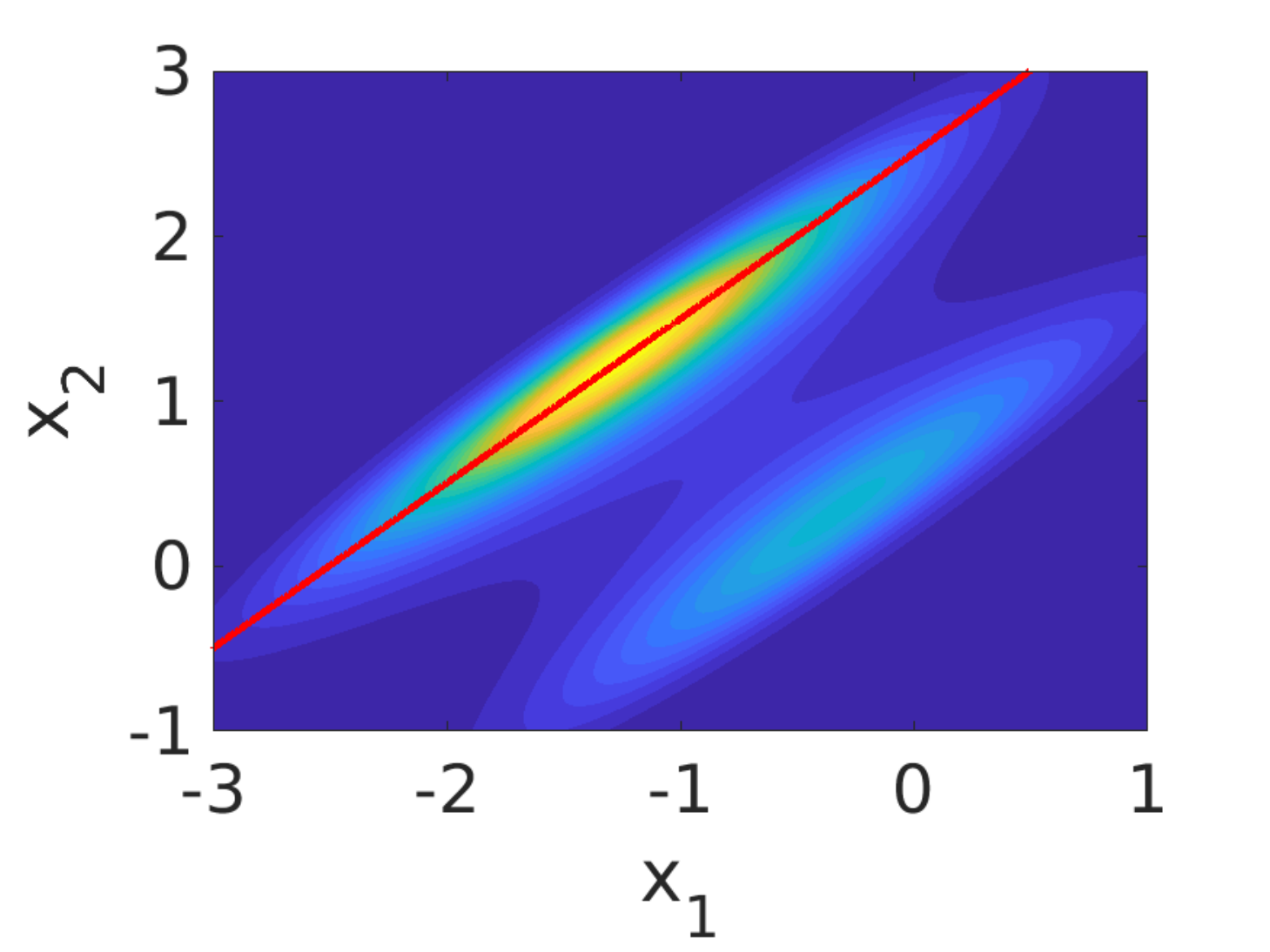}
\hfill
\includegraphics[width = 0.32\textwidth]{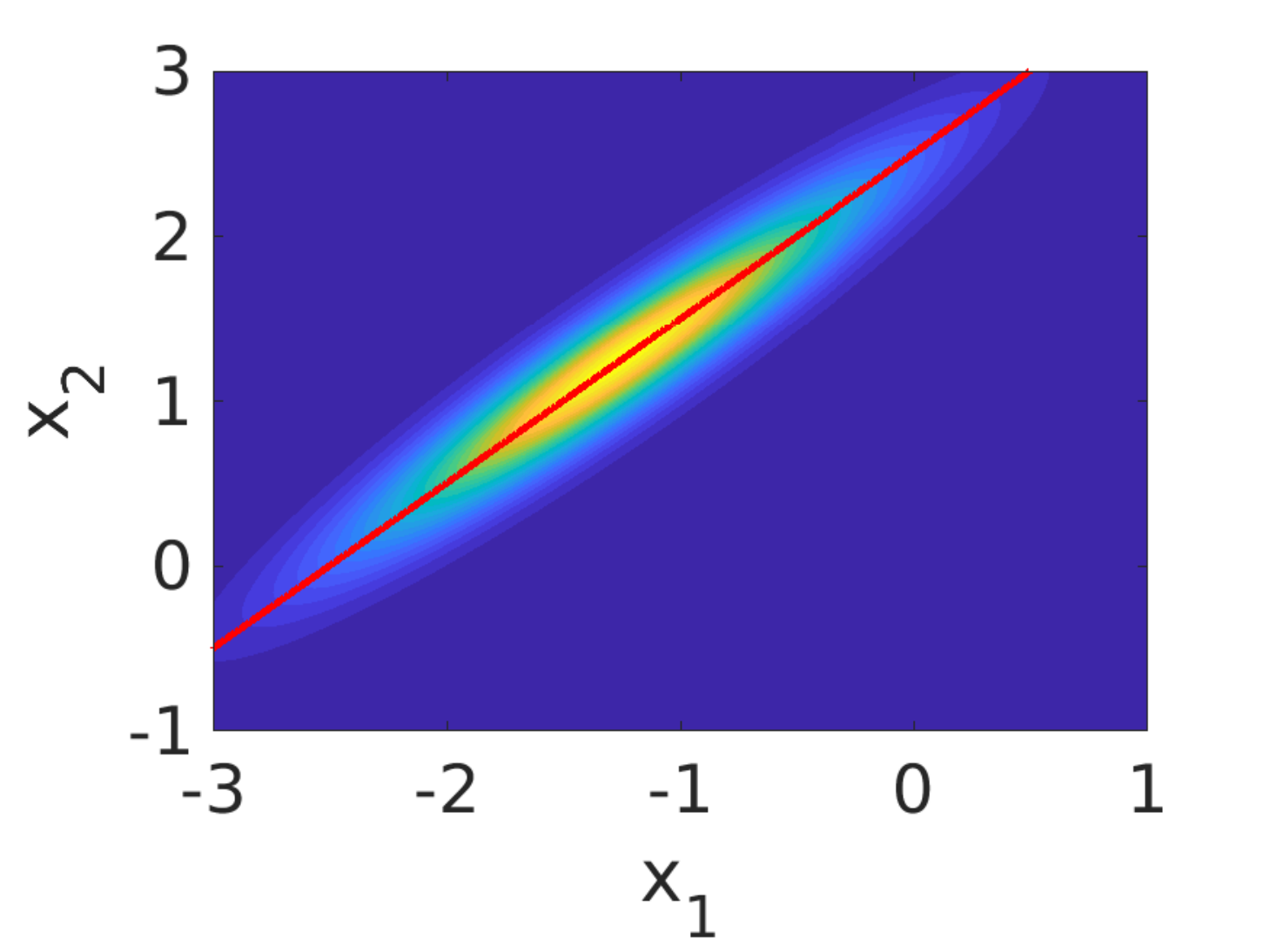}
\hfill
\includegraphics[width = 0.32\textwidth]{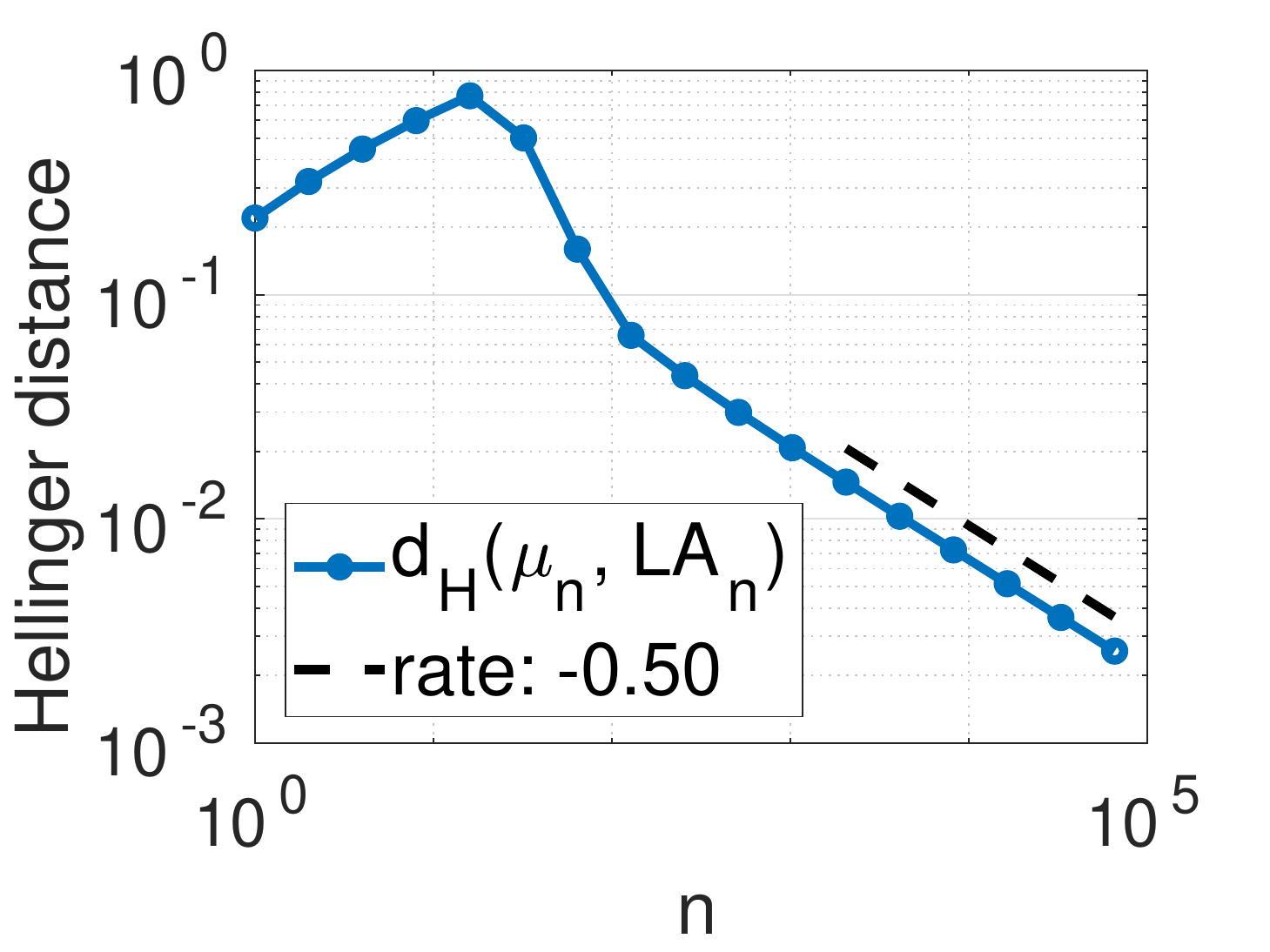}

\caption{Same as in Figure \ref{fig:exam4_Phi_error} but for Example \ref{exam:2D_2}.}
\label{fig:exam5_Phi_error}
\end{figure}
\end{example}

\section{{Robustness of Laplace-based Monte Carlo Methods}}
\label{sec:numerics}
In practice, we are often interested in expectations or integrals of quantities of interest $f\colon \bbR^d\to \bbR$ w.r.t.~$\mu_n$ such as
\[
	\int_{\bbR^d} f(x) \ \mu_n(\d x).
\]
For example, in Bayesian statistics the posterior mean ($f(x) = x$) or posterior probabilities ($f(x) = \mathbf 1_A(x)$, $A \in \mc B(\bbR^d)$) are desirable quantities.
Since $\mu_n$ is seldom given in explicit form, numerical integration must be applied for approximating such integrals.
To this end, since the prior measure $\mu_0$ is typically a well-known measure for which efficient numerical quadrature methods are available, the integral w.r.t.~$\mu_n$ is rewritten as two integrals w.r.t.~$\mu_0$
\begin{equation}\label{eq:ZZ}
	\int_{\bbR^d} f(x) \ \mu_n(\d x)
	=
	\frac{\int_{\bbR^d} f(x) \ \exp(-n\Phi_n(x))\ \mu_0(\d x)}{\int_{\bbR^d} \exp(-n\Phi_n(x))\ \mu_0(\d x)}\eqqcolon\frac{Z'_n}{Z_n}.
\end{equation}
If then a quadrature rule such as $\int_{\bbR^d} g(x) \ \mu_0(\d x) \approx \frac 1N \sum_{i=1}^N w_i\ g(x_i)$ is used, we end up with an approximation
\[
	\int_{\bbR^d} f(x) \ \mu_n(\d x)
	\approx
	\frac {\sum_{i=1}^N w_i\ f(x_i) \exp(-n\Phi_n(x_i))}{\sum_{i=1}^N w_i\ \exp(-n\Phi_n(x_i))}.
\]
This might be a good approximation for \rev{small} $n \in \bbN$.
However, as soon as $n\to \infty$ the likelihood term $\exp(-n\Phi_n(x_i))$ will deteriorate and this will be reflected by a deteriorating efficiency of the quadrature scheme---not in terms of the convergence rate w.r.t.~$N$, but w.r.t.~the constant in the error estimate, {as we will display later in examples}.

If the Gaussian Laplace approximation $\LAn$ of $\mu_n$ is used {as the \rev{prior} measure for numerical integration} instead of $\mu_0$, we get the following approximation
\[
	\int_{\bbR^d} f(x) \ \mu_n(\d x)
	\approx
	\frac {\sum_{i=1}^N w_i\ f(x_i) \frac{\pi_n(x_i)}{\tilde \pi_n(x_i)}}{\sum_{i=1}^N w_i\  \frac{\pi_n(x_i)}{\tilde \pi_n(x_i)}},
\]
where $\pi_n$ and $\tilde \pi_n$ denote the unnormalized Lebesgue density of $\mu_n$ and $\LAn$, respectively.
This time, we can not only apply well-known quadrature and sampling rules for Gaussian measures, but moreover, we also know due to Lemma \ref{lem:conv_LA}, that the ratio $\frac{\pi_n(x)}{\tilde \pi_n(x)}$ converges in mean w.r.t.~$\LAn$ to $1$.
Hence, we do not expect a deteriorating efficiency of the numerical integration as $n\to \infty$.
On the contrary, as we subsequently discuss for several numerical integration methods, their efficiency for a finite number of samples $N\in\bbN$ will even improve as $n\to\infty$ if they are based on the Laplace approximation $\LAn$.

For the sake of simplicity, we consider the simple case of $\Phi_n \equiv \Phi + \text{const}$ in the following presentation---nonetheless, the presented results can be extended to the general case given appropriate modifications of the assumptions.
Thus, we consider probability measures $\mu_n$ of the form
\begin{equation}\label{equ:mu_simple}
	\mu_n(\d x)
	\propto
	\e^{-n \Phi(x)} \ \mu_0(\d x)
\end{equation}
where we assume that $\Phi$ satisfies the assumptions of Theorem \ref{theo:laplace_method}.
However, when dealing with the Laplace approximation of $\mu_n$ and, particularly, with the ratios of the corresponding normalizing constants, it is helpful to use the following representation
\begin{equation}\label{equ:mu_simple_2}
	\mu_n(\d x)
	=
	\frac 1{Z_n} \e^{-n \Phi_n(x)} \mu_0(\d x),
	\qquad
	\Phi_n(x) \coloneqq \Phi(x) - \iota_n,
\end{equation}
where $\iota_n \coloneqq \min_{x \in \S_0} \Phi(x) - \frac 1n \log \pi_0(x)$ and $	Z_n = \e^{n \iota_n} \int_{\bbR^d} \e^{-n\Phi(x)} \pi_0(x)\ \d x$.
By this construction the resulting $I_n(x) \coloneqq \Phi_n(x) - \frac 1n \log \pi_0(x)$ satisfies $I_n(x_n) = 0$ as required in Assumption \ref{assum:LA_0} for the construction of the Laplace approximation $\LAn$.
\rev{Note, that for $\Phi_n = \Phi - \iota_n$ the Assumption \ref{assum:LA_0} and \ref{assum:LA_Conv} imply the assumptions of Theorem 1 for $f = \pi_0$ and $p=0$.}

\paragraph{Preliminaries.} Before we start analyzing numerical methods based on the Laplace approximation as their reference measure, we take a closer look at the details of the asymptotic expansion for integrals provided in Theorem \ref{theo:laplace_method} and their implications for expectations w.r.t.~$\mu_n$ given in \eqref{equ:mu_simple_2}.
\begin{enumerate}
\item
\textbf{The coefficients:} The proof of Theorem \ref{theo:laplace_method} in \cite[Section IX.5]{Wong2001} provides explicit expressions\footnote{{There is a typo in \cite[Section IX.5]{Wong2001} stating that the sum in \eqref{equ:Laplace_coef} is taken over all $\valpha \in \bbN_ 0^d$ with $|\valpha| = k$.}} for the coefficients $c_k \in \bbR$ in the asymptotic expansion
\[
	\rev{
	\int_D f(x) \exp(-n \Phi(x)) \d x
	=
	\e^{-n \Phi(x_\star)} 
	n^{-d/2}
	\left(
	\sum_{k=0}^p c_k(f) n^{- k}
	+ \mc O\left(n^{-p-1}\right)
	\right),
	}
\]
namely---\rev{given that $f \in C^{2p+2}(D, \mathbb R)$ and $\Phi \in C^{2p+3}(D, \mathbb R)$}---that
\begin{equation} \label{equ:Laplace_coef}
	c_k(f)
	=
	\sum_{\valpha \in \bbN_ 0^d \colon |\valpha| = 2k}
	\frac{\kappa_{\valpha}}{\valpha !} D^{\valpha} F(0)
\end{equation}
where for $\valpha = (\alpha_1,\ldots,\alpha_d)$ we have $|\valpha| = \alpha_1+\cdots+\alpha_d$, $\valpha ! = \alpha_1 ! \cdots \alpha_d!$, $D^{\valpha} = D^{\alpha_1}_{x_1} \cdots D^{\alpha_d}_{x_d}$ and
\[
	F(x) \coloneqq f(h(x))\ \det(\nabla h(x))
\]
with $h\colon \Omega \to U(x_\star)$ being a diffeomorphism between $0 \in \Omega \subset \bbR^d$ and a particular neighborhood $U(x_\star)$ of $x_\star$ mapping $h(0) = x_\star$ and such that $\det(\nabla h(0)) = 1$.
The diffeomorphism $h$ is specified by the well-known Morse's Lemma and depends only on $\Phi$.
\rev{In particular, if $\Phi \in C^{2p+3}(D, \mathbb R)$, then  $h\in C^{2p+1}(\Omega, U(x_\star))$.
For the constants $\kappa_{\valpha} = \kappa_{\alpha_1}\cdots \kappa_{\alpha_d} \in\bbR$ we have $\kappa_{\alpha_i} = 0$ if $\alpha_i$ is odd and $\kappa_{\alpha_i} = (2/\lambda_i)^{(\alpha_i +1)/2} \Gamma((\alpha_i+1)/2)$ otherwise with $\lambda_i>0$ denoting the $i$th eigenvalue of $H_\star = \nabla^2 \Phi(x_\star)$. 
Hence, we get 
\begin{equation} \label{equ:Laplace_coef_2}
	c_k(f)
	=
	\sum_{\valpha \in \bbN_ 0^d \colon |\valpha| = k}
	\frac{\kappa_{2\valpha}}{(2\valpha)!} D^{2\valpha} F(0).
\end{equation}
}

\item
\textbf{The normalization constant of $\mu_n$:} Theorem \ref{theo:laplace_method} implies that 
\rev{if $\pi_0\in C^2(\bbR^d; \bbR)$ and $\Phi\in C^3(\bbR^d, \bbR)$, then}
\[
	\rev{
	\int_{\bbR^d} \pi_0(x) \ \exp(-n\Phi(x)) \ \d x 
	= \e^{-n\Phi(x_\star)} n^{-d/2} \left(\rev{\frac{(2\pi)^{d/2}\, \pi_0(x_\star)}{\sqrt{\det(H_\star)}}} 
	+ \mc O(n^{-1})\right).
	}
\]
Hence, we obtain for the normalizing constant $Z_n$ in \eqref{equ:mu_simple_2} that
\begin{equation}\label{equ:Zn_simple_asymp}
	Z_n 
	=
	\e^{n (\iota_n -\Phi(x_\star))} n^{-d/2} \left(\rev{\frac{(2\pi)^{d/2}\, \pi_0(x_\star)}{\sqrt{\det(H_\star)}}} + \mc O(n^{-1})\right).
\end{equation}
If we compare this to the normalizing constant $\tilde Z_n = n^{-d/2} \sqrt{\det(2\pi C_n)}$ of its Laplace approximation we get
\[
	\frac{Z_n}{\tilde Z_n} 
	= 
	\e^{n(\iota_n-\Phi(x_\star))} 
	\rev{\frac{\frac{\pi_0(x_\star)}{\sqrt{\det(H_\star)}} + \mc O(n^{-1})}{\sqrt{\det(C_n)}}}.
\]
We now show that 
\begin{equation}\label{equ:Ratio_Zn_asymp}
	\frac{Z_n}{\tilde Z_n} \rev{= 1 + \mc O(n^{-1})}.
\end{equation}
First, we get due to $C_n \to H_\star^{-1}$ that $\rev{\sqrt{\det(C_n)}\to \frac 1{\sqrt{\det(H_\star)}}}$ as $n\to \infty$.
Moreover, 
\[
	\e^{n(\iota_n-\Phi(x_\star))} 
	=
	\frac{\exp(n (\Phi(x_n) - \Phi(x_\star)))}{\pi_0(x_n)}.
\]
Since $x_n \to x_\star$ continuity implies $\pi_0(x_n) \to \pi_0(x_\star)$ as $n\to \infty$.
Besides that, the strong convexity of $\Phi$ in a neighborhood of $x_\star$---due to $\nabla^2 \Phi(X_\star) >0$ and $\Phi\in C^3(\bbR^d,\bbR)$---implies that for a $c>0$
\[
	\Phi(x_n) - \Phi(x_\star) \leq \frac 1{2c} \|\nabla\Phi(x_n)\|^2,
\]
also known as Polyak--\L ojasiewicz condition.
Because of
\[
	\nabla \Phi(x_n) = \frac 1n \nabla \log \pi_0(x_n),
\]
since $\nabla I_n(x_n) = 0$, we have that $|\Phi(x_n) - \Phi(x_\star)| \in \mc O(n^{-2})$, and hence,
\[
	\lim_{n\to\infty} \e^{n(\iota_n-\Phi(x_\star))} = 1/\pi_0(x_\star).
\]
This yields \eqref{equ:Ratio_Zn_asymp}.

\item
\textbf{The expectation w.r.t.~$\mu_n$:}
The expectation of a $f \in L^1_{\mu_0}(\bbR)$ w.r.t.~$\mu_n$ is given by
\[
	\evalt{\mu_n}{f} = \frac{ \int_{\S_0} f(x) \pi_0(x) \exp(-n\Phi(x))\ \d x}{ \int_{\S_0} \pi_0(x) \exp(-n\Phi(x))\ \d x}.
\]
If $f, \pi_0 \in C^2(\bbR^d, \bbR)$ and and $\Phi\in C^3(\bbR^d, \bbR)$, then we can apply the asymptotic expansion above to both integrals and obtain
\begin{align} \label{equ:Exp_mu}
	\evalt{\mu_n}{f} 
	& = \frac{\e^{-n \Phi(x_\star)} n^{-d/2}\ ( c_0(f\pi_0) + \mc O(n^{-1}))}{\e^{-n \Phi(x_\star)} n^{-d/2}\ ( c_0(\pi_0) + \mc O(n^{-1}))}
	= f(x_\star) + \mc O(n^{-1}).
\end{align}
If $f, \pi_0 \in C^4(\bbR^d; \bbR)$ and $\Phi\in C^5(\bbR^d, \bbR)$, then we can make this more precise by using the next explicit terms in the asymptotic expansions of both integrals, apply the rule for the division of (asymptotic) expansions (cf. \cite[Section 1.8]{Olver1997}) and obtain $\evalt{\mu_n}{f} = f(x_\star) + \tilde c_1(f,\pi_0) n^{-1} + \mc O(n^{-2})$ where $\tilde c_1(f,\pi_0) = \frac{1}{c_0(\pi_0)} c_1(f\pi_0) - \frac{c_1(\pi_0)}{c^2_0(\pi_0)}c_0(f\pi_0)$.

\item
\textbf{The variance w.r.t.~$\mu_n$:}
The variance of a $f \in L^2_{\mu_0}(\bbR)$ w.r.t.~$\mu_n$ is given by
\begin{align*}
	\Var_{\mu_n}(f)
	& = \evalt{\mu_n}{f^2} - \evalt{\mu_n}{f}^2.
\end{align*}
If $f, \pi_0 \in C^2(\bbR^d;\bbR)$ and $\Phi\in C^3(\bbR^d, \bbR)$, then we can exploit the result for the expectation w.r.t.~$\mu_n$ from above and obtain 
\begin{align} \label{equ:Var_mu}
	\Var_{\mu_n}(f)
	&
	=
	f^2(x_\star) + \mc O(n^{-1})
	- \left(f(x_\star) + \mc O(n^{-1})\right)^2
	\in
	\mc O(n^{-1}).
\end{align}
If $f,\pi_0\in C^4(\bbR^d, \bbR)$ and $\Phi\in C^5(\bbR^d, \bbR)$, then a straightforward calculation---see Appendix \ref{sec:Var_mu_n}--- using the explicit formulas for $c_1(f^2\pi_0)$ and $c_1(f\pi_0)$ as well as  $\nabla h(0) = I$ yields
\begin{align} \label{equ:Var_mu_2}
	\Var_{\mu_n}(f)
	=
	n^{-1} \|\nabla f(x_\star)\|^2_{H_\star^{-1}} + \mc O(n^{-2}).
\end{align}
Hence, the variance $\Var_{\mu_n}(f)$ decays like $n^{-1}$ provided that $\nabla f(x_\star) \neq 0$---otherwise it decays (at least) like $n^{-2}$.
\end{enumerate}

\begin{remark}\label{rem:Conv_xn}
As already exploited above, the assumptions of Theorem \ref{theo:laplace_method} imply that $\Phi$ is strongly convex in a neighborhood of $x_\star = \lim_{n\to \infty} x_n$, where $x_n = \argmin_{x\in\S_0} \Phi(x) - \frac 1n \log \pi_0(x)$.
This yields $|\Phi(x_n) - \Phi(x_\star)| \in \mc O(n^{-2})$, and thus
\begin{equation}\label{equ:Conv_xn}
    \|x_n - x_\star\| \in \mc O(n^{-1}).
\end{equation}
\end{remark}

\subsection{Importance Sampling}\label{sec:LAIS}
Importance sampling is a variant of Monte Carlo integration where an integral w.r.t.~$\mu$ is rewritten as an integral w.r.t.~a dominating \emph{importance distribution} $\mu \ll \nu$, i.e.,
\[
	\int_{\bbR^d} f(x) \ \mu(\d x) = \int_{\bbR^d} f(x) \ \frac{\d\mu}{\d\nu}(x) \ \nu(\d x).
\]
The integral appearing on the righthand side is then approximated by Monte Carlo integration w.r.t.~$\nu$: given $N$ independent draws $x_i$, $i=1,\ldots,N$, according to $\nu$ we estimate
\[
	\int_{\bbR^d} f(x) \ \mu(\d x)
	\approx
	\frac 1N \sum_{i=1}^N w(x_i) f(x_i),
	\qquad
	w(x_i) \coloneqq \frac{\d\mu}{\d\nu}(x_i).
\]  
Often the density or \emph{importance weight function} $w = \frac{\d\mu}{\d\nu}\colon \bbR^d \to [0,\infty)$ is only known up to a normalizing constant $\tilde w \propto \frac{\d\mu}{\d\nu}$. 
In this case, we can use \emph{self-normalized importance sampling}
\[
	\int_{\bbR^d} f(x) \ \mu(\d x)
	\approx
	\frac {\sum_{i=1}^N \tilde w(x_i) \ f(x_i)}{\sum_{i=1}^N \tilde w(x_i)}
	\eqqcolon
	\mathrm{IS}^{(N)}_{\mu,\nu}(f).
\]  
As for Monte Carlo, there holds a strong law of large numbers (SLLN) for self-normalized importance sampling, i.e.,
\[
	\frac {\sum_{i=1}^N \tilde w(X_i) \ f(X_i)}{\sum_{i=1}^N \tilde w(X_i)}
	\xrightarrow[N\to\infty]{\text{a.s.}}
	\evalt{\mu}{f},
\]
where $X_i\sim \nu$ are i.i.d.~, which follows from the ususal SLLN and the continuous mapping theorem.
Moreover, by the classical central limit theorem (CLT) and Slutsky's theorem also a similar statement holds for self-normalized importance sampling: given that
\[
	\sigma^2_{\mu,\nu}(f)
	\coloneqq
	\evalt{\nu}{\left(\frac{\d \mu}{\d \nu}\right)^2 (f-\evalt{\mu}{f})^2}
	< \infty 
\]
we have
\[
	\sqrt{N} \left( 	\frac {\sum_{i=1}^N \tilde w(X_i) \ f(X_i)}{\sum_{i=1}^N \tilde w(X_i)} - \evalt{\mu}{f}\right)
	\xrightarrow[N\to\infty]{\mc D}
	\mc N(0, 	\sigma^2_{\mu,\nu}(f)).
\]
Thus, the asymptotic variance $\sigma^2_{\mu,\nu}(f)$ serves as a measure of efficiency for self-normalized importance sampling.
To ensure a finite $\sigma^2_{\mu,\nu}(f)$ for many functions of interest $f$, e.g., bounded $f$, the importance distribution $\nu$ has to have heavier tails than $\mu$ such that the ratio $\frac{\d\mu}{\d\nu}$ belongs to $L^2_\nu(\bbR)$, see also \cite[Section 3.3]{RobertCasella1999}.
Moreover, if we even have $\frac{\d\mu}{\d\nu} \in L^\infty_\nu(\bbR)$ we can bound
\begin{equation}\label{equ:var_ratio_IS}
	\sigma^2_{\mu,\nu}(f)
	\leq
	\left\|\frac{\d\mu}{\d\nu}\right\|_{L^\infty_\nu}\
	\evalt{\mu}{(f-\evalt{\mu}{f})^2}
	\qquad
	\Leftrightarrow
	\qquad
	\frac{\sigma^2_{\mu,\nu}(f)}{\Var_\mu(f)} 
	\leq
	\left\|\frac{\d\mu}{\d\nu}\right\|_{L^\infty_\nu},
\end{equation}
i.e., the ratio between the asymptotic variance of importance sampling w.r.t.~$\nu$ and plain Monte Carlo w.r.t.~$\mu$ can be bounded by the $L^\infty_\nu$- or supremum norm of the importance weight $\frac{\d\mu}{\d\nu}$.

For the measures $\mu_n$ a natural importance distribution (called $\nu$ above) which allows for direct sampling are the \rev{prior} measure $\mu_0$ and the Gaussian Laplace approximation $\LAn$.
We study the behaviour of the resulting asymptotic variances $\sigma^2_{\mu_n,\mu_0}(f)$ and $\sigma^2_{\mu_n,\LAn}(f)$ in the following.

\paragraph{Prior importance sampling.}
First, we consider $\mu_0$ as importance distribution. 
For this choice the importance weight function $w_n \coloneqq \frac{\d \mu_n}{\d\mu_0}$ is given by
\[
	w_n(x) = \frac 1{Z_n} \exp(-n \Phi_n(x)),
	\qquad
	x \in \S_0,
\]
with $\Phi_n(x) = \Phi(x)-\iota_n$, see \eqref{equ:mu_simple_2}. 
Concerning the bound in \eqref{equ:var_ratio_IS} we immediately obtain for sufficiently smooth $\pi_0$ and $\Phi$ by \eqref{equ:Zn_simple_asymp}, assuming w.l.o.g.~$\min_x \Phi(x) = \Phi(x_\star)=0$, that
\[
	\|w_n\|_{L^\infty}
	=
	Z^{-1}_n
	\e^{n\iota_n}
	=
	\tilde c n^{d/2},
	\qquad
	\tilde c > 0,
\]
explodes as $n\to\infty$.

Of course, that is just the deterioration of an upper bound, but in fact we can prove the following rather negative result \rev{where we use the notation $g(n) \sim h(n)$ for the asymptotic equivalence of functions of $n$, i.e., $g(n) \sim h(n)$ iff $\lim_{n\to\infty} \frac{g(n)}{h(n)} =1$.}

\begin{lemma}\label{lem:IS_prior}
\rev{
Given $\mu_n$ as in \eqref{equ:mu_simple_2} with $\Phi$ satisfying the assumptions of Theorem \ref{theo:laplace_method} for $p=1$ and $\pi_0 \in C^4(\bbR^d,\bbR)$ with $\pi_0(x_\star) \neq 0$, we have for any $f\in C^4(\bbR^d, \bbR)\cap L^1_{\mu_0}(\bbR)$ with $\nabla f(x_\star) \neq 0$ that}
\[
	\sigma^2_{\mu_n,\mu_0}(f)  
	\sim \rev{\tilde c_f}n^{d/2 -1},
	\qquad
	\rev{\tilde c_f > 0},
\]
which yields \rev{$\frac{\sigma^2_{\mu_n,\mu_0}(f)}{\Var_{\mu_n}(f)} \sim \tilde c_f n^{d/2}$ for another $\tilde c_f > 0$}.
\end{lemma}
\begin{proof}
W.l.o.g.~we may assume that $f(x_\star) = 0$, since $\sigma^2_{\mu_n,\mu_0}(f) = \sigma^2_{\mu_n,\mu_0}(f-c)$ for any $c\in\bbR$.
Moreover, for simplicity we assume w.l.o.g.~that $\Phi(x_\star) = 0$.
We study 
\begin{align*}
	\sigma^2_{\mu_n,\mu_0}(f)
	& =
	\frac 1{Z_n^2}
	\int_{\S_0}
	\e^{- 2n\Phi_n(x)}
	\ (f(x) - \evalt{\mu_n}{f})^2
	\ \mu_0(\d x)\\
	& =
	\frac 1{\e^{-2n\iota_n} Z_n^2}
	\int_{\S_0}
	\e^{- 2n\Phi(x)}
	\ (f(x) - \evalt{\mu_n}{f})^2
	\ \mu_0(\d x)
\end{align*}
by analyzing the growth of the numerator and denominator w.r.t.~$n$.
Due to the preliminaries presented above we know that \rev{$\e^{-2n\iota_n}Z^2_n = c_0^2 n^{-d} + \mc O(n^{-d-1})$} with $c_0 = \rev{(2\pi)^{d/2}\, \pi_0(x_\star)/\sqrt{\det(H_\star)}} >0$. 
Concerning the numerator we start with decomposing
\[
	\int_{\S_0}
	\e^{- 2n\Phi(x)}
	\ (f(x) - \evalt{\mu_n}{f})^2
	\ \mu_0(\d x)
	=
	J_1(n) - 2 J_2(n) + J_3(n)
\]
where this time
\begin{align*}
	J_1(n) & \coloneqq \int_{\S_0} f^2(x)	\e^{- 2n\Phi(x)} \ \mu_0(\d x),\\
	J_2(n) & \coloneqq \evalt{\mu_n}{f} \int_{\S_0} f(x) \e^{- 2n\Phi(x)} \ \mu_0(\d x),\\
	J_3(n) & \coloneqq \evalt{\mu_n}{f}^2 \int_{\S_0} \e^{- 2n\Phi(x)} \ \mu_0(\d x).
\end{align*}
We derive now asymptotic expansions of these terms based on Theorem \ref{theo:laplace_method}.
It is easy to see that the assumptions of Theorem \ref{theo:laplace_method} are also fulfilled when considering integrals w.r.t.~$\e^{- 2n\Phi}$.
We start with $J_1$ and obtain due to $f(x_\star) = 0$ that
\begin{align*}
	J_1(n) 
	& = \int_{\S_0} f^2(x) \e^{- 2n\Phi(x)} \ \mu_0(\d x)
	=
	c'_{1}(f^2\pi_0) n^{-d/2-1} + \mc O(n^{-d/2-2})
\end{align*}
where $c'_{1}(f^2\pi_0) \in \bbR$ is the same as $c_1(f^2\pi_0)$ in \eqref{equ:Laplace_coef} but for $2\Phi$ instead of $\Phi$.

Next, we consider $J_2$ and recall that due to $f(x_\star) = 0$ we have $\evalt{\mu_n}{f} \in \mc O(n^{-1})$, see \eqref{equ:Exp_mu}. 
Furthermore, $f(x_\star) = 0$ also implies $\int_{\S_0} f(x) \pi_0(x)\, \e^{- 2n\Phi(x)}\ \d x \in \mc O(n^{-1-d/2})$, see Theorem \ref{theo:laplace_method}.
Thus, we have
\begin{align*}
	J_2(n) 
	= \evalt{\mu_n}{f} \int_{\S_0} f(x) \e^{- 2n\Phi(x)} \ \mu_0(\d x)
	& \in \mc O\left(n^{-d/2-2} \right).
\end{align*}
Finally, we take a look at $J_3$. 
By Theorem \ref{theo:laplace_method} we have $\int_{\S_0} \exp(- 2n\Phi(x)) \ \mu_0(\d x) \in \mc O(n^{-d/2})$ and, hence, obtain
\begin{align*}
	J_3(n) = \evalt{\mu_n}{f}^2 \int_{\S_0} \exp(- 2n\Phi(x)) \ \mu_0(\d x) 
	\, \in \mc O(n^{-2-d/2}).
\end{align*}
Hence, $J_1$ has the dominating power w.r.t.~$n$ and we have that
\begin{align*}
	\int_{\S_0}
	\e^{- 2n\Phi(x)}
	\ (f(x) - \evalt{\mu_n}{f})^2
	\ \mu_0(\d x)
	& \sim c'_{1}(f^2\pi_0) n^{-d/2-1}.
\end{align*}
At this point, we remark that due to the assumption $\nabla f(x_\star) \neq 0$ we have $c'_{1}(f^2\pi_0) \neq 0$: we know by \eqref{equ:Laplace_coef_2} that  $c'_{1}(f^2\pi_0) =\frac 12 \sum_{j=1}^d \kappa_{2\ve_j} D^{2\ve_j} F(0)$ where $F(x) = \pi_0(h'(x)) f^2(h'(x)) \det(\nabla h'(x))$ and $h'$ denotes the diffeomorphism for $2\Phi$ appearing in Morse's lemma and mapping $0$ to $x_\star$;
applying the product formula and using $f(x_\star)=0$ as well as $\det(\nabla h'(0)) = 1$ we get that $D^{2\ve_j} F(x_\star) = \pi_0(x_\star) D^{2\ve_j} ( f^2(h'(x_\star)) )$; similarly, we get using $f(x_\star)=0$ that $D^{2\ve_j} ( f^2(h'(x_\star))) = 2 | \ve_j^\top \nabla h'(0) \nabla f(x_\star)|^2$; since $h'$ is a diffeomorphishm $\nabla h'(0)$ is regular and, thus, $c'_{1}(f^2\pi_0) \neq 0$.
The statement follows now by
\[
	\sigma^2_{\mu_n,\mu_0}(f)
	=
	\frac{c'_{1}(f^2\pi_0) n^{-d/2-1} + \mc O(n^{-d/2-2})}{c_0^2 n^{-d} + \mc O(n^{-d-1})}
	\sim \frac{c'_{1}(f^2\pi_0)}{c_0^2} \ n^{d/2-1}
\]
and by recalling that $\Var_{\mu_n}(f) \sim c n^{-1}$ because of $\nabla f(x_\star) \neq 0$, see \eqref{equ:Var_mu_2}.
\end{proof}
Thus, Lemma \ref{lem:IS_prior} tells us that the asymptotic variance of importance sampling for $\mu_n$ with the \rev{prior} $\mu_0$ as importance distribution grows like $n^{d/2-1}$ as $n\to \infty$ for a large class of integrands. 
Hence, its efficiency deteriorates like $n^{d/2-1}$ for $d\geq 3$ as the target measures $\mu_n$ become more concentrated.

\paragraph{Laplace-based importance sampling.}
We now consider the Laplace approximation $\LAn$ as importance distribution which yields the following importance weight function
\begin{equation}\label{equ:LAIS_weight}
	w_n(x) 
	\coloneqq \frac{\d\mu_n}{\d\LAn}(x)
	=
	\frac {\tilde Z_n}{Z_n} \exp(-n R_n(x)) \textbf 1_{\S_0}(x),
	\qquad
	x \in \bbR^d,
\end{equation}
with $R_n(x) = I_n(x) - \tilde I_n(x) = I_n(x_n) - \frac 12 \|x-x_n\|^2_{C_n^{-1}}$ for $x\in\S_0$.
In order to ensure $w_n \in L^2_{\LAn}(\bbR)$ we need that
\[
	\evalt{\LAn}{\exp(-2n R_n)}
	= 
	\frac{1}{\tilde Z_n} \int_{\S_0} \exp(- n[2I_n(x) - \tilde I_n(x)])\ \d x
	< \infty.
\]
Despite pathological counterexamples a sound requirement for $w_n \in L^2_{\LAn}(\bbR)$ is that
\[
	{\lim_{\|x\|\to\infty}  2I_n(x) - \tilde I_n(x) = +\infty,}
\]
for example by assuming that there exist  $\delta, c_1 > 0$, $c_0 > 0$, and $n_0 \in \bbN$ such that 
\begin{equation}\label{equ:cond_In_IS}
	{I_n(x) \geq c_1 \|x\|^{2+\delta} + c_0,}
	\qquad
	\forall x \in\S_0\ \forall n\geq n_0.
\end{equation}
If the Lebesgue density $\pi_0$ of $\mu_0$ is bounded, then \eqref{equ:cond_In_IS} is equivalent to the existence of $n_0$ and a $\tilde c_0$ such that
\[
	{\Phi_n(x) \geq c_1 \|x\|^{2+\delta} + \tilde c_0,}
	\qquad
	\forall x \in\S_0\ \forall n\geq n_0.
\]
Unfortunately, condition \eqref{equ:cond_In_IS} is not enough to ensure a well-behaved asymptotic variance $\sigma^2_{\mu_n,\LAn}(f)$ as $n\to \infty$, since
\[
	\|w_n\|_{L^\infty}
	=
	\frac{\tilde Z_n}{Z_n}
	\exp(- n \min_{x\in\S_0} R_n(x)).
\]
Although, we know due to \eqref{equ:Ratio_Zn_asymp} that $\frac{\tilde Z_n}{Z_n} \to 1$ as $n\to \infty$, the supremum norm of the importance weight $w_n$ of Laplace-based importance sampling will explode exponentially with $n$ if $\min_x R_n(x) < 0$.
This can be sharpened to proving that even the asymptotic variance of Laplace-based importance sampling w.r.t.~$\mu_n$ as in \eqref{equ:mu_simple_2} deteriorates exponentially as $n\to\infty$ for many functions $f\colon \bbR^d\to \bbR$ if  
\[
	\exists x \in \S_0 \colon \Phi(x) < \frac 12 \Phi(x_\star) + \frac 14 \|x - x_\star\|^2_{H^{-1}_\star}
\]
by means of Theorem \ref{theo:laplace_method} applied to 
\[
	\int_{\bbR^d} (f(x) - \evalt{\mu_n}{f})^2 \ \exp(- n[2I_n(x) - \tilde I_n(x)]) \ \d x.
\]
This means, except when $\Phi$ is basically strongly convex, the asymptotic variance of Laplace-based important sampling can explode exponentially or not even exist as $n$ increases.
However, in the good case, so to speak, we obtain the following.

\begin{proposition}\label{propo:LAIS_conv_L2}
Consider the measures $\mu_n$ as in \eqref{equ:mu_simple_2} \rev{with $\Phi_n = \Phi - \iota_n$ and $\pi_0$ satisfying Assumption \ref{assum:LA_0} and \ref{assum:LA_Conv}.}
If there exist an $n_0\in\bbN$ such that for all $n\geq n_0$ we have
\begin{equation} \label{equ:In_strong_convex}
	I_n(x) \geq I_n(x_n) + \frac 12 (x-x_n)^\top \nabla^2 I_n(x_n) (x-x_n) \qquad \forall x \in \S_0,
\end{equation}
then for any $f \in L^2_{\mu_0}(f)$
\[
	\lim_{n\to\infty} 
	\frac{\sigma^2_{\mu_n,\LAn}(f)}{\Var_{\mu_n}(f)}
	= 1.
\]
\end{proposition} 
\begin{proof}
The assumption \eqref{equ:In_strong_convex} ensures that $R_n(x) = I_n(x) - \tilde I_n(x) \geq 0$ for each $x\in\S_0$.
Thus, 
\[
	\|w_n\|_{L^\infty}
	=
	\frac{\tilde Z_n}{Z_n}
\]
and the assertion follows by \eqref{equ:var_ratio_IS} and the fact that $\lim_{n\to\infty} \frac{\tilde Z_n}{Z_n} = 1$ due to \eqref{equ:Ratio_Zn_asymp}.
\end{proof}

Condition \eqref{equ:In_strong_convex} is for instance satisfied, if $I_n$ is strongly convex with a constant $\gamma \geq \lambda_{\min}(\nabla^2 I_n(x_n))$ where the latter denotes the smallest eigenvalue of the positive definite Hessian $\nabla^2 I_n(x_n)$. 
However, this assumption or even \eqref{equ:In_strong_convex} is quite restrictive and, probably, hardly fulfilled for many interesting applications.
Moreover, the success in practice of Laplace-based importance sampling is well-documented.
How come that despite a possible infinite asymptotic variance Laplace-based importance sampling performs that well?
In the following we refine our analysis and exploit the fact that the Laplace approximation concentrates around the minimizer of $I_n$.
Hence, with an increasing probability samples drawn from the Laplace approximation are in a small neighborhood of the minimizer.
Thus, if $I_n$ is, e.g., only locally strongly convex---which the assumptions of Theorem \ref{theo:conv_H} actually imply---then with a high probability the mean squared error might be small.
We clarify these arguments in the following and present a positive result for Laplace-based importance sampling under mild assumptions but for a weaker error criterion than the decay of the mean squared error.

First we state a concentration result for $N$ samples drawn from $\LAn$ which is an immediate consequence of Proposition \ref{propo:con_LA_Phi_n}.

\begin{proposition}\label{propo:LAIS_conc}
Let $N\in\bbN$ be arbitrary and let $X^{(n)}_i \sim \LAn$ be i.i.d. where $i = 1,\ldots,N$.
Then, for a sequence of radii  $r_n \geq r_0 n^{-q}>0$, $n\in\bbN$, with $q \in (0, 1/2)$ we have
\[
	\bbP\left( \max_{i = 1,\ldots, N} \|X^{(n)}_i - x_n\| \leq r_n \right) = 1 - \e^{- c_0 N n^{1-2q}} \xrightarrow[n\to\infty]{} 1.
\]
\end{proposition}

\begin{remark}\label{rem:mu_n_B_n}
In the following we require expectations w.r.t.~restrictions of the measures $\mu_n$ in \eqref{equ:mu_simple_2}  to shrinking balls $B_{r_n}(x_n)$.
To this end, we note that the statements of Theorem \ref{theo:laplace_method} also hold true for shrinking domains $D_n = B_{r_n}(x_\star)$ with $r_n = r_0n^{-q}$ as long as $q <1/2$.
This can be seen from the proof of Theorem \ref{theo:laplace_method} in \cite[Section IX.5]{Wong2001}.
In particular, all coefficients in the asymptotic expansion for $\int_{D_n} f(x) \exp(-n \Phi(x)) \d x$ with sufficiently smooth $f$ are the same as for $\int_{D} f(x) \exp(-n \Phi(x)) \d x$ and the difference between both integrals decays for increasing $n$ like $\exp(-c n^\epsilon)$ for an $\epsilon>0$ and $c>0$.
Concerning the balls $B_{r_n}(x_n)$ with decaying radii $r_n = r_0 n^{-q}$, $q\in[0,1/2)$, we have due to $\|x_n - x_\star\| \in \mc O(n^{-1})$---see Remark \ref{rem:Conv_xn}---that $B_{r_n/2}(x_\star) \subset B_{r_n}(x_n) \subset B_{2r_n}(x_\star)$ for sufficiently large $n$.
Thus, the facts for $\mu_n$ as in \eqref{equ:mu_simple_2} stated in the preliminaries before Section \ref{sec:LAIS} do also apply to the restrictions of $\mu_n$ to $B_{r_n}(x_n)$ with $r_n = r_0 n^{-q}$, $q\in[0,1/2)$. 
In particular, the difference between $\evalt{\mu_n}{f}$ and $\evalt{\mu_n}{f\ | \ B_{r_n}(x_n)}$ decays faster than any negative power of $n$ as $n\to \infty$.
\end{remark}

The next result shows that the mean absolute error of the Laplace-based importance sampling behaves like $n^{-(3q-1)}$ conditional on all $N$ samples belonging to shrinking balls $B_{r_n}(x_n)$ with $r_n = r_0 n^{-q}$, $q \in (1/3, 1/2)$. 

\begin{lemma}\label{lem:LAIS_conv}
Consider the measures $\mu_n$ in \eqref{equ:mu_simple_2} and suppose they satisfy the assumptions of Theorem \ref{theo:conv_H}.
Then, for \rev{any $f \in  C^2(\bbR^d, \bbR) \cap L^2_{\mu_0}(\bbR)$} there holds for the error
\[
	e_{n,N}(f) 
	\coloneqq
	\left| \mathrm{IS}^{(N)}_{\mu_n, \LAn}(f) - \evalt{\mu_n}{f}\right|,
\]
of the Laplace-based importance sampling with $N\in\bbN$ samples that
\[
	\ev{ e_{n,N}(f)  \ \big| \ X^{(n)}_1, \ldots, X^{(n)}_N \in B_{r_n}(x_n)} 
	\in \mc O(n^{-(3q-1)}),
\]
where $r_n = r_0 n^{-q}$ with $q \in (1/3,1/2)$.
\end{lemma}
\begin{proof}
We start with
\begin{align*}
	e_{n,N}(f)
	& \coloneqq
	\left| \mathrm{IS}^{(N)}_{\mu_n, \LAn}(f) - \evalt{\mu_n}{f}\right|\\
	& \leq
	\left| \mathrm{IS}^{(N)}_{\mu_n, \LAn}(f) - \evalt{\mu_n}{f\ | \ B_{r_n}(x_n) }\right| +  \left| \evalt{\mu_n}{f} - \evalt{\mu_n}{f\ | \ B_{r_n}(x_n) }\right|.
\end{align*}
The second term decays subexponentially w.r.t.~$n$, see Remark \ref{rem:mu_n_B_n}.
Hence, it remains to prove that 
\[
	\ev{ \left. \left| \mathrm{IS}^{(N)}_{\mu_n, \LAn}(f) - \evalt{\mu_n}{f\ | \ B_{r_n}(x_n) }\right| \ \right|  \ X^{(n)}_1, \ldots, X^{(n)}_N \in B_{r_n}(x_n) } 
	\in \mc O(n^{- (3q-1)}).
\]
To this end, we write the self-normalizing Laplace-based importance sampling estimator as
\begin{align*}
	\mathrm{IS}^{(N)}_{\mu_n, \LAn}(f)
	& =
	\frac{\frac 1N \sum_{i=1}^N \tilde w_n (X^{(n)}_i) f(X^{(n)}_i) }{\frac 1N \sum_{i=1}^N \tilde w_n (X^{(n)}_i)}
	=
	H_{n,N} \ S_{n,N}
\end{align*}
where we define
\[
	H_{n,N} \coloneqq \frac{Z_n}{\tilde Z_n} \ \frac1{\frac 1N \sum_{i=1}^N \tilde w_n (X^{(n)}_i)},
	\qquad
	S_{n,N} = \frac 1N \sum_{i=1}^N w_n (X^{(n)}_i) f(X^{(n)}_i),
\]
and recall that $w_n$ is as in \eqref{equ:LAIS_weight} and $\tilde w_n(x) = \exp(-n R_n(x))$. 
Notice that
\[
	\ev{S_{n,N}} = \evalt{\mu_n}{f},
	\quad
	\ev{S_{n,N}\ | \ X^{(n)}_1, \ldots, X^{(n)}_N \in B_{r_n}(x_n) } = \evalt{\mu_n}{f \ | \ B_{r_n}(x_n)}.
\]
Let us denote the event that $X^{(n)}_1, \ldots, X^{(n)}_N \in B_{r_n}(x_n)$ by $A_{n,N}$ for brevity.
Then,
\begin{align*}
	& \ev{ \left. \left| \mathrm{IS}^{(N)}_{\mu_n, \LAn}(f) - \evalt{\mu_n}{f\ | \ B_{r_n}(x_n) }\right| \ \right|  \ A_{n,N} }\\
	= & 
	\quad \ev{ \left. \left| H_{n,N} S_{n,N} - \ev{S_{n,N} \ | \ A_{n,N} } \right| \ \right| \ A_{n,N}}\\
	\leq & \quad \ev{ \left. \left|S_{n,N} - \ev{S_{n,N} \ | \ A_{n,N} } \right| \ \right| \ A_{n,N}} + \ev{ \left. \left|\left(H_{n,N} -1\right) S_{n,N} \right| \ \right| \ A_{n,N}}
\end{align*}
The first term in the last line can be bounded by the conditional variance of $S_{n,N}$ given $X^{(n)}_1, \ldots, X^{(n)}_N \in B_{r_n}(x_n)$, i.e., by Jensen's inequality we obtain
\begin{align*}
	\ev{ \left. \left|S_{n,N} - \ev{S_{n,N} \ | \ A_{n,N} } \right| \ \right| \ A_{n,N}}^2
	& \leq \Var( \left. S_{n,N} \ \right| \  A_{n,N} )\\
	& = \frac 1N \Var_{\mu_n}\left( \left. f \ \right| \  B_{r_n}(x_n) \right)
	\in \mc O(n^{-1}),
\end{align*}
see Remark \ref{rem:mu_n_B_n} and the preliminaries before Subsection \ref{sec:LAIS}.
Thus,
\begin{align*}
	\ev{ \left. \left|S_{n,N} - \ev{S_{n,N} \ | \ A_{n,N} } \right| \ \right| \ A_{n,N}}
	\in \mc O(n^{-1/2})
	\rev{\, \subset \mc O(n^{- (3q-1)})}
\end{align*}
and it remains to study if $\ev{ \left. \left|\left(H_{n,N} -1\right) S_{n,N} \right| \ \right| \ A_{n,N}} \in \mc O(n^{- (3q-1)})$.
Given that $X^{(n)}_1, \ldots, X^{(n)}_N \in B_{r_n}(x_n)$ we can bound the values of the random variable $H_{n,N}$ for sufficiently large $n$: first, we have \rev{$Z_n / \tilde Z_n = 1 + \mc O(n^{-1})$, see \eqref{equ:Ratio_Zn_asymp},} and second
\[
	\exp\left(- n \max_{|x-x_n| \leq r_n} |R_n(x)|  \right) \leq \frac1{\frac 1N \sum_{i=1}^N \tilde w_n (X^{(n)}_i)} \leq \exp\left( n \max_{|x-x_n| \leq r_n} |R_n(x)|  \right).
\]
Since $|R_n(x)| \leq c_3 \|x-x_n\|^3$ for $|x-x_n| \leq r_n$ due to the local boundedness of the third derivative of $I_n$ and $r_n = r_0n^{-q}$, we have that
\[
	\exp\left(-c n^{1-3q} \right) \leq \frac1{\frac 1N \sum_{i=1}^N \tilde w_n (X^{(n)}_i)} \leq \exp\left(c n^{1-3q} \right),
\]
where $c>0$. 
Thus, there exist $\alpha_n \leq 1 \leq \beta_n$ with \rev{$\alpha_n = \e^{- c n^{1-3q} } (1 + \mc O(n^{-1}))$ and $\beta_n \sim \e^{cn^{1-3q}} (1 + \mc O(n^{-1}))$} such that
\[	
	\bbP\left( \alpha_n \leq H_{n,N} \leq \beta_n \ | \  X^{(n)}_1, \ldots, X^{(n)}_N \in B_{r_n}(x_n) \right) = 1.
\]
Since \rev{$\e^{\pm c n^{1-3q} } (1 +  \mc O(n^{-1})) = 1 \pm cn^{1-3q} + \mc O(n^{-1})$} we get that for sufficiently large $n$ there exists a $\tilde c >0$ such that
\[	
	\bbP\left( \left|  H_{n,N} -1 \right| \leq c n^{1-3q} + \tilde c n^{-1} \ | \  X^{(n)}_1, \ldots, X^{(n)}_N \in B_{r_n}(x_n)\right)
	= 
	1.
\]
Hence,
\begin{align*}
	\ev{ \left. \left|\left(H_{n,N} -1\right) S_{n,N} \right| \ \right| \ A_{n,N}}
	& \leq \left(c n^{1-3q} + \tilde c n^{-1} \right) \ev{\left. |S_{n,N}| \ \right| \ A_{n,N} }\\
	& \in	\mc O(n^{- (3q-1)} ),
\end{align*}
since $\ev{ |S_{n,N}| \ | \ A_{n,N} } \leq \evalt{\mu_n} {|f| \ | \ B_{r_n(x_n)}}$ is uniformly bounded w.r.t.~$n$.
This concludes the proof.
\end{proof}

We now present our main result for the Laplace-based importance sampling which states that the corresponding error decays in probability to zero as $n\to \infty$ and the order of decay is arbitrary close to $n^{-1/2}$.

\begin{theorem}\label{theo:conv_LAIS_stoch}
Let the assumptions of Lemma \ref{lem:LAIS_conv} be satisfied.
Then, for \rev{any $f \in  C^2(\bbR^d, \bbR) \cap L^2_{\mu_0}(\bbR)$ and each} sample size $N\in\bbN$ the error $e_{n,N}(f)$ of Laplace-based importance sampling satisfies
\[
	n^{\delta} e_{n,N}(f) \xrightarrow[n\to \infty]{\bbP} 0,
	\qquad
	\delta \in [0,1/2).
\]
\end{theorem} 
\begin{proof}
Let $0\leq \delta < 1/2$ and $\epsilon > 0$ be arbitrary.
We need to show that
\[
	\lim_{n\to\infty} \bbP\left( n^{\delta} e_{n,N}(f) > \epsilon \right) = 0.
\]
Again, let us denote the event that $X^{(n)}_1, \ldots, X^{(n)}_N \in B_{r_n}(x_n)$ by $A_{n,N}$ for brevity.
By Proposition \ref{propo:LAIS_conc} we obtain for radii  $r_n = r_0 n^{-q}$ with $q \in (1/3, 1/2)$ that
\begin{align*}
	\bbP\left( n^{\delta} \ e_{n,N}(f) \leq \epsilon \right)
	& \geq \bbP\left( n^{\delta} e_{n,N}(f) \leq \epsilon \text{ and } X_1,\ldots, X_N \in B_{r_n}(x_n)\right)\\
	& = \bbP\left( n^{\delta}e_{n,N}(f) \leq \epsilon \ | \ A_{n,N} \right) \ \bbP(A_{n,N}) \\
	& \geq \bbP\left( n^{\delta}e_{n,N}(f) \leq \epsilon \ | \ A_{n,N} \right) \  \left(1 - C_N \e^{- c_0 N  n^{1-2q}}\right).
\end{align*}
The second term on the righthand side in the last line obviously tends to $1$ exponentially as $n\to \infty$.
Thus, it remains to prove that 
\[
	\lim_{n\to\infty} \bbP\left( n^{\delta} e_{n,N}(f) \leq \epsilon \ | \ X_1,\ldots, X_N \in B_{r_n}(x_n) \right) = 1
\]
To this end, we apply a conditional Markov inequality for the positive random variable $e_{n,N}(f)$, i.e.,
\[
	\bbP\left( n^{\delta} e_{n,N}(f) > \epsilon \ | \ A_{n,N}\right)
	\leq
	 \frac {n^{\delta}}{\epsilon}  \ev{e_{n,N}(f) \ | \ A_{n,N}}
	\in
	\mc O\left(n^{\delta - \min(3q-1, 1/2)} \right)
\]
where we used Lemma \ref{lem:LAIS_conv}.
Choosing $q \in (1/3, 1/2)$ such that $q > \frac{1+\delta}3 \in [1/3, 1/2) $ yields the statement.
\end{proof}

\subsection{Quasi-Monte Carlo Integration}
We now want to approximate integrals as in \eqref{eq:ZZ} w.r.t.~measures $\mu_n(\d x) \propto \exp(-n\Phi(x)) \mu_0(\d x)$ as in \eqref{equ:mu_simple_2} by Quasi-Monte Carlo methods. 
These will be used to estimate the ratio $Z'_n/Z_n$ by separately approximating the two integrals $Z'_n$ and $Z_n$ in \eqref{eq:ZZ}. 
The preconditioning strategy using the Laplace approximation will be explained exemplarily for Gaussian and uniform priors, two popular choices for Bayesian inverse problems. 

We start the discussion by first focusing on a uniform prior distribution $\mu_0=\mathcal U([-\frac 12, \frac 12]^d)$.
The integrals $Z'_n$ and $Z_n$ are then
\begin{equation}\label{equ:Zn_Theta}
	Z'_n=\int_{[-\frac 12, \frac 12]^d} f(x)\Theta_n(x)\mathrm dx, 
	\qquad 
	Z_n=\int_{[-\frac 12, \frac 12]^d} \Theta_n(x)\mathrm dx,
\end{equation}
where we set $\Theta_n(x) \coloneqq\exp(-n\Phi(x))$ for brevity.

We consider Quasi-Monte Carlo integration based on \emph{shifted Lattice rules}: an $N$-point Lattice rule in the cube $[-\frac 12, \frac 12]^d$ is based on points 
\begin{equation}\label{equ:lattice_points}
	x_i=\fraco\Big(\frac{iz}{N}+\Delta \Big) -\frac 12\,, \quad i=1,\ldots,N\,,
\end{equation}
where $z \in \{1,\ldots , N-1\}^d$ denotes the so-called generating vector, $\Delta$ is a uniformly distributed random shift on $[-\frac 12, \frac 12]^d$ and $\fraco$ denotes the fractional part (component-wise).
These randomly shifted points provide unbiased estimators
\[
	Z'_{n,QMC} \coloneqq \frac{1}{N}\sum_{i=1}^N f(x_i)\Theta(x_i),
	\qquad 
	Z_{n,QMC} \coloneqq \frac{1}{N}\sum_{i=1}^N \Theta(x_i)
\]
of the two integrals $Z'_n$ and $Z_n$ in \eqref{equ:Zn_Theta}. 
Under the assumption that the quantity of interest $f\colon \bbR^d\to \bbR$ is linear and bounded, we can focus in the following on the estimation of the normalization constant $Z_n$, the results can be then straightforwardly generalized to the estimation of $Z'_n$.
For the estimator $Z_{n,QMC}$ we have the following well-known error bound.

\begin{theorem}\cite[Thm. 5.10]{KuoActa}\label{thm:QMC}
Let
$\gamma= \{\gamma_{\bm{\nu}}\}_{{\bm{\nu}}\subset \{1,\ldots,d\}}$
denote POD (product and order dependent) weights of the form $\gamma_{\bm{\nu}}=\alpha_{|{\bm{\nu}}|}\prod_{j\in {\bm{\nu}}}\beta_j$ specified by two sequences $\alpha_0=\alpha_1=1, \alpha_2,\ldots\ge 0$ and $\beta_1\ge \beta_2\ge \ldots >0$ for ${\bm{\nu}}\subset \{1,\ldots,d\}$ and $|{\bm{\nu}}|=\# {\bm{\nu}}$. Then, a randomly shifted Lattice rule with $N=2^m, m\in \mathbb N$, can be constructed via a component-by-component algorithm with POD weights at costs of $\mathcal O(dN\log N+d^2 N)$ operations, such that \rev{for sufficiently smooth $\Theta\colon [-\frac 12, \frac 12]^d \to [0,\infty)$}
\begin{equation}
\mathbb E_{\Delta}[(Z_n-Z_{n,QMC})^2]^{1/2}\le \left(2\sum_{\emptyset\neq {\bm{\nu}}\subset\{1,\ldots,d\}}\gamma_{\bm{\nu}}^\kappa\left(\frac{2\zeta(2\kappa)}{(2\pi^2)^\kappa}\right)^{|{\bm{\nu}}|} \right)^{\frac{1}{2\kappa}} \|\Theta_n\|_\gamma \ \  N^{- \frac{1}{2\kappa}}
\end{equation}
for $\kappa\in(1/2,1]$ with
\[
	\|\Theta_n\|_\gamma^2 \coloneqq \sum_{{\bm{\nu}}\subset\{1,\ldots,d\}} \frac{1}{\gamma_{\bm{\nu}}}\int_{[-\frac 12, \frac 12]^{|{\bm{\nu}}|}}\left(\int_{[-\frac 12, \frac 12]^{d-|{\bm{\nu}}|}}\frac{\partial^{|{\bm{\nu}}|}\Theta_n}{\partial x_{\bm{\nu}}}(x)\mathrm d x_{{1:d}\setminus {\bm{\nu}}}\right)^2\mathrm d x_{\bm{\nu}}
\]
and $\zeta(a)\coloneqq\sum_{k=1}^\infty k^{-a}$.
\end{theorem}

The norm $\| \Theta_n\|_\gamma$ in the convergence analysis depends on $n$, 
in particular, it can grow {polynomially} w.r.t. the concentration level $n$ of the measures $\mu_n$ 
as we state in the next result.

\begin{lemma}\label{lem:Prior_QMC_norm}
\rev{Let $\Phi\colon \bbR^d\to[0,\infty)$ satisfy the assumptions of Theorem \ref{theo:laplace_method} for $p=2d$}. 
Then, for the norm $\| \Theta_n\|_\gamma$ in the error bound in Theorem \ref{thm:QMC} there holds
\[
	 \lim_{n\to\infty} n^{-d/4} \|\Theta_n\|_\gamma
	>
	0.
\]
\end{lemma}
The proof of Lemma \ref{lem:Prior_QMC_norm} is rather technical and can be found in Appendix \ref{sec:Prior_QMC_norm}.
We remark that Lemma \ref{lem:Prior_QMC_norm} just tells us that the root mean squared error estimate for QMC integration based on the prior measure explodes like $n^{d/4}$.
This does in general not indicate that the error itself explodes; in fact the QMC integration error for the normalization constant is bounded by $1$ in our setting.
Nonetheless, Lemma \ref{lem:Prior_QMC_norm} indicates that a naive Quasi-Monte Carlo integration based on the uniform prior $\mu_0$ is not suitable for highly concentrated target or posterior measures $\mu_n$.
We subsequently propose and study a Quasi-Monte Carlo integration based on the Laplace approximation $\LAn$.

\paragraph{Laplace-based Quasi-Monte Carlo.}
To stabilize the numerical integration for concentrated $\mu_n$, we propose a preconditioning based on the Laplace approximation, i.e., an affine rescaling according to the mean and covariance of $\LAn$. 
In the uniform case, the functionals $I_n$ are independent of $n$.
The computation of the Laplace approximation requires therefore only one optimization to solve for $x_n = x_\star = \argmin_{x\in [-\frac 12, \frac 12]^d} \Phi(x)$. 
In particular, the Laplace approximation of $\mu_n$ is given by $\LAn = \mathcal N(x_\star, \frac1n H_{\star}^{-1})$ where $H_{\star}$ denotes the positive definite Hessian $\nabla^2 \Phi(x_\star)$.
Hence, $H_\star$ allows for an orthogonal diagonalization $H_\star = QDQ^\top$ with orthogonal matrix $Q\in\mathbb R^{d\times d}$ and diagonal matrix $D=\diag(\lambda_1,\ldots\lambda_d)\in\mathbb R^{d\times d}$, $\lambda_1\ge \ldots\ge \lambda_d>0$. 

We now use this decomposition in order to construct an affine transformation which reverses the increasing concentration of $\mu_n$ and yields a QMC approach robust w.r.t.~$n$.
This transformation is given by
\[
	g_n(x)
	\coloneqq
	x_\star + \sqrt{\frac{2|\ln\tau|}{n}} QD^{-\frac12}x,
	\qquad
	x \in [-\frac 12, \frac 12]^d,
\]
where $\tau \in (0,1)$ is a truncation parameter. 
The idea of the transformation $g_n$ is to zoom into the parameter domain and thus, to counter the concentration effect. 
The domain will then be truncated to $G_n \coloneqq g_n([-\frac 12, \frac 12]^d) \subset [-\frac 12, \frac 12]^d$ and we consider
\begin{equation}\label{eq:transqoi}
	\hat Z_n 
	\coloneqq 
	\int_{G_n} \Theta_n(x)\ \d x
	=
	C_{\text{trans}, n}  \int_{[-\frac 12, \frac 12]^d} \Theta_n(g_n(x))\ \d x.
\end{equation}
The determinant of the Jacobian of the transformation $g_n$ is given by $ \det(\nabla g_n(x)) \equiv C_{\text{trans}, n} = \left(\frac{2|\ln\tau|}{n}\right)^{\frac d2} \sqrt{\det (H_\star)} \sim c_\tau n^{-d/2}$.
We will now explain how the parameter $\tau$ effects the truncation error. 
For given $\tau\in(0,1)$, the Laplace approximation is used to determine the truncation effect: 
\begin{align*}
	\int_{G_n} \ \LAn(\mathrm dx)
	& 
	=
	\frac{C_{\text{trans}, n}}{\tilde Z_n}  \int_{[-\frac 12, \frac 12]^d} \exp\left( -\frac n2 \|g_n(x) - x_\star\|^2_{H_\star} \right) \ \d x\\
	& =
	\left(\frac{|\ln \tau|}{\pi}\right)^{d/2} \int_{[-\frac 12, \frac 12]^d} \exp\left( - |\ln \tau| x^2 \right) \ \d x\\
	& =
	\left(\frac{|\ln \tau|}{\pi}\right)^{d/2} \left(\frac{\sqrt \pi \mathrm{erf}(0.5 \sqrt{|\ln \tau}|)}{\sqrt{|\ln \tau}|)}\right)^d
	= \mathrm{erf}(0.5 \sqrt{|\ln \tau}|)^d.
\end{align*}
Thus, since  due to the concentration effect of the Laplace approximation we have $\int_{\S_0} \ \LAn(\mathrm dx) \to 1$ exponentially with $n$, we get
\begin{align*}
	\int_{\S_0\setminus G_n} \ \LAn(\mathrm dx)
	\leq 1 - \mathrm{erf}(0.5 \sqrt{|\ln \tau}|)^d,
\end{align*}
thus, the truncation error $\int_{\S_0\setminus G_n} \ \LAn(\mathrm dx)$ becomes arbitrarly small for sufficiently small $\tau\ll 1$, since $\mathrm{erf}(t) \to 1$ as $t\to1$.
If we apply now QMC integration using shifted Lattice rule in order to compute the integral over $[-\frac 12, \frac 12]^d$ on the righthand side of \eqref{eq:transqoi}, we obtain the following estimator for $\hat Z_n$ in \eqref{eq:transqoi}:
\[
		\hat Z_{n,QMC} \coloneqq \frac{C_{\text{trans}, n}}{N}\sum_{i=1}^N \Theta(g_n(x_i))
\]
with $x_i$ as in \eqref{equ:lattice_points}.
Concerning the norm $\|\Theta_n\circ g_n \|_\gamma $ appearing in the error bound for $|\hat Z_n - \hat Z_{n,QMC}|$ we have now the following result.

\begin{lemma}\label{lem:Laplace_QMC_norm}
\rev{Let $\Phi\colon \bbR^d\to[0,\infty)$ satisfy the assumptions of Theorem \ref{theo:laplace_method} for $p=2d$.}
Then, for the norm $\|\Theta_n\circ g_n\|_\gamma$ with $g_n$ as above there holds 
\[
	\|\Theta_n\circ g_n \|_\gamma \in \mc O(1)
	\qquad
	\text{ as } n\to\infty.
\]
\end{lemma}
Again, the proof is rather technical and can be found in Appendix \ref{sec:Laplace_QMC_norm}.
This proposition yields now our main result.
\begin{corollary}\label{cor:QMCprec}
Given the assumptions of Lemma \ref{lem:Laplace_QMC_norm}, a randomly shifted lattice rule with $N=2^m, m\in \mathbb N$, can be constructed via a component-by-component algorithm with product and order dependent weights at costs of $\mathcal O(dN\log N+d^2 N)$ operations, such that for $\kappa\in(1/2,1]$
\begin{equation}
\mathbb E_{\Delta}[(Z_n-\hat Z_{n,QMC})^2]^{1/2} \le n^{-\frac{d}{2}} \left( c_1  h(\tau)+ c_2  n^{-\frac{1}{2}}+ c_3 N^{- \frac{1}{2\kappa}}\right) 
\end{equation}
with constants $c_1,c_2, c_3>0$ independent of $n$ and $h(\tau) = 1 - \mathrm{erf}(0.5 \sqrt{|\ln \tau}|)^d$. 
\end{corollary}
\begin{proof}
The triangle inequality leads to a separate estimation of the domain truncation error of the integral w.r.t. the posterior and the QMC approximation error, i.e.
 \begin{eqnarray*}
\mathbb E_{\Delta}[(Z_n-\hat Z_{n,QMC})^2]^{1/2}&\le& |Z_n - \hat Z_n| + \mathbb E_{\Delta}[(\hat Z_n-\hat Z_{n,QMC})^2]^{1/2}\,.
\end{eqnarray*}
The second term on the right hand side corresponds to the QMC approximation error. Thus, Theorem \ref{thm:QMC} and Lemma \ref{lem:Laplace_QMC_norm} imply
 \begin{eqnarray*}
\mathbb E_{\Delta}[(\hat Z_n-\hat Z_{n,QMC})^2]^{1/2}\le c_3 n^{-\frac d2} N^{- \frac{1}{2\kappa}} \,, \quad \kappa\in(1/2,1]\,,
\end{eqnarray*}
where the term $n^{-\frac d2}$ is due to $C_{\text{trans}, n} \sim c_\lambda n^{-\frac d2}$.
The domain truncation error can be estimated similar to the proof of Lemma \ref{lem:conv_LA}: 
\begin{eqnarray*}
	|Z_n -\hat Z_n|&=&|\int_{\S_0\setminus G_n} \Theta_n(x)\ \d x|\\
	&=&\big|\int_{\S_0 \setminus G_n} \Theta_n(x) \ \d x -\tilde{Z_n} \int_{\S_0 \setminus G_n}  \ \LAn(\mathrm dx)  +\tilde{Z}_n\int_{\S_0 \setminus G_n}  \ \LAn(\mathrm dx)  \Big|\\
&=&\big | \tilde{Z}_n  \int_{\S_0 \setminus G_n}e^{-n \Phi(x)} e^{n \tilde \Phi(x)} \LAn(\d x)-\tilde{Z}_n \int_{\S_0 \setminus G_n}   \ \LAn(\d x) +\tilde{Z}_n\int_{\S_0 \setminus G_n} \ \LAn(\d x)\Big|\\
&\le&\tilde{Z}_n\int_{\S_0 \setminus G_n} \big| e^{-n \tilde R_n(x)}  -1 \big| \LAn(\mathrm dx) +\tilde{Z}_n\big|\int_{\S_0 \setminus G_n} \ \LAn(\mathrm dx)\Big|\\
&\le&\tilde Z_n\int_{\S_0 } \big | e^{-n \tilde R(x)}  -1 \big | \LAn (\mathrm dx) + \tilde{Z}_n \left(1 - \mathrm{erf}(0.5 \sqrt{|\ln \tau}|)^d\right)
\end{eqnarray*}
where $\tilde{Z}_n=n^{-\frac{d}{2}}\sqrt{\det(2\pi H_\star^{-1})}$. The result follows by the proof of Lemma \ref{lem:conv_LA}.
\end{proof}

\begin{remark}
In the case of Gaussian priors, the transformation simplifies to $w=g_n(x)=x_\star + n^{\frac12}QD^{-\frac12}x$ due to the unboundedness of the prior support. However, to show an analogous result to Corollary \ref{cor:QMCprec}, uniform bounds w.r. to $n$ on the norm of the mixed first order derivatives of the preconditioned posterior density $\Theta_n(g_n(T^{-1}x))$ in a weighted Sobolev space, where $T^{-1}$ denotes the inverse cumulative distribution function of the normal distribution, need to be proven. See \cite{Kuo2016} for more details on the weighted space setting in the Gaussian case. Then, a similar result to Corollary \ref{cor:QMCprec} follows straightforwardly from \cite[Thm 5.2]{Kuo2016}. The numerical experiments shown in subsection \ref{sec:NE} suggest that we can expect a noise robust behavior of Laplace-based QMC methods also in the Gaussian case. This will be subject to future work. 
\end{remark}

\begin{remark}
Note that the QMC analysis in Theorem \ref{thm:QMC} can be extended to an infinite dimensional setting, cp. \cite{Kuo2016} and the references therein for more details. This opens up the interesting possibility to generalize the above results to the infinite dimensional setting and to develop methods with convergence independent of the number of parameters and independent of the measurement noise. Furthermore, higher order QMC methods can be used for cases with smooth integrands, cp. \cite{SchwabQMC,Gantner,DGLS17}, leading to higher convergence rates than the first order methods discussed here. In the uniform setting, it has been shown in \cite{SchillingsScaling} that the assumptions on the first order derivatives (and also higher order derivatives) of the transformed integrand are satisfied for Bayesian inverse problems related to a class of parametric operator equations, i.e., the proposed approach leads to a robust performance w.r.t. the size of the measurement noise for integrating w.r.t.~posterior measure resulting from this class of forward problems. The theoretical analysis of this setting will be subject to future work.    
\end{remark}

\begin{remark}[Numerical Quadrature]
Higher regularity of the integrand allows to use higher order methods such as sparse quadrature and higher order QMC methods, leading to faster convergence rates. In the infinite dimensional Bayesian setting with uniform priors, we refer to \cite{SchillingsSchwab2013,Schillings2014} for more details on sparse quadrature for smooth integrands. In the case of uniform priors, the methodology introduced above can be used to bound the quadrature error for the preconditioned integrand by the truncation error and the sparse grid error. 
\end{remark}

\subsection{Examples}\label{sec:NE}
In this subsection we present two examples illustrating our previous theoretical results for importance sampling and quasi-Monte Carlo integration based on the \rev{prior} measure $\mu_0$ and the Laplace approximation $\LAn$ of the target measure $\mu_n$.
Both examples are Bayesian inference or inverse problems where the first one uses a toy forwad map and the second one is related to inference for a PDE model.

\subsubsection{Algebraic Bayesian inverse problem}\label{sec:Exam_Alg}
We consider inferring $x \in [-\frac 12, \frac 12]^d$ for $d=1,2,3,4$ based on a uniform prior $\mu_0 = \mathcal U([-\frac 12, \frac 12]^d)$ and a realisation of the noisy observable of $Y = F(X) + \varepsilon_n$ where $X\sim \mu_0$ and the noise $\varepsilon_n \sim N(0, n^{-1} \Gamma_d)$, $\Gamma_d = 0.1I_d$, are independent, and $F(x) = (F_1(x),\ldots,F_d(x))$ with 
\[
	F_1(x) = \exp(x_1/5), \quad
	F_2(x) = x_2 - x_1^2, \quad
	F_3(x) = x_3, \quad
	F_4(x) = 2x_4 + x_1^2,
\] 
for $x = (x_1,\ldots,x_d)$. 
The resulting posterior measure $\mu_n$ on $[-\frac 12, \frac 12]^d$ are of the form \eqref{equ:mu_simple_2} with
\[
	\Phi(x)
	=
	\frac12
	\|y -F(x)\|^2_{\Gamma_d^{-1}}.
\]
We used $y = G(0.25\cdot {\boldsymbol 1}_d)$ throughout where ${\boldsymbol 1}_d = (1,\ldots,1) \in \bbR^d$.
We then compute the posterior expectation of the quantity of interest $f(x) = x_1+\cdots+x_d$.
To this end, we employ importance sampling and quasi-Monte Carlo integration based on $\mu_0$ and the Laplace approximation $\LAn$ as outlined in the precious subsections.
We compare the output of these methods to a reference solution obtained by a brute-force tensor grid trapezoidal rule for integration.
In particular, we estimate the root mean squared error (RMSE) of the methods and how it evolves as $n$ increases. 

\emph{Results for importance sampling:} 
In order to be sufficiently close to the asymptotic limit, we use $N = 10^5$ samples for self-noramlized importance sampling.
We run $1,000$ independent simulations of the importance sampling integration and compute the resulting empirical RMSE.
In Figure \ref{fig:exam_alg_IS} we present the results for increasing $n$ and various $d$ for prior-based and Laplace-based importance sampling.
We obtain a good match to the theoretical results, i.e., the RMSE for choosing the prior measure as importance distribution behaves like $n^{d/4 - 1/2}$ in accordance to Lemma \ref{lem:IS_prior}.
Besides that the RMSE for choosing the Laplace approximation as importance distribution decays like $n^{1/2}$ after a preasymptotic phase.
This is relates to the statement of Theorem \ref {theo:conv_LAIS_stoch} where we have shown that the absolute error\footnote{We have also computed the empirical $L^1$-error which showed a similar behaviour as the RMSE.} decays in probability like $n^{1/2}$.
Note that the assumptions of Proposition \ref{propo:LAIS_conv_L2} are not satisfied for this example for all $d=1,2,3,4$.
\begin{figure}[htb]
\includegraphics[width = 0.49\textwidth]{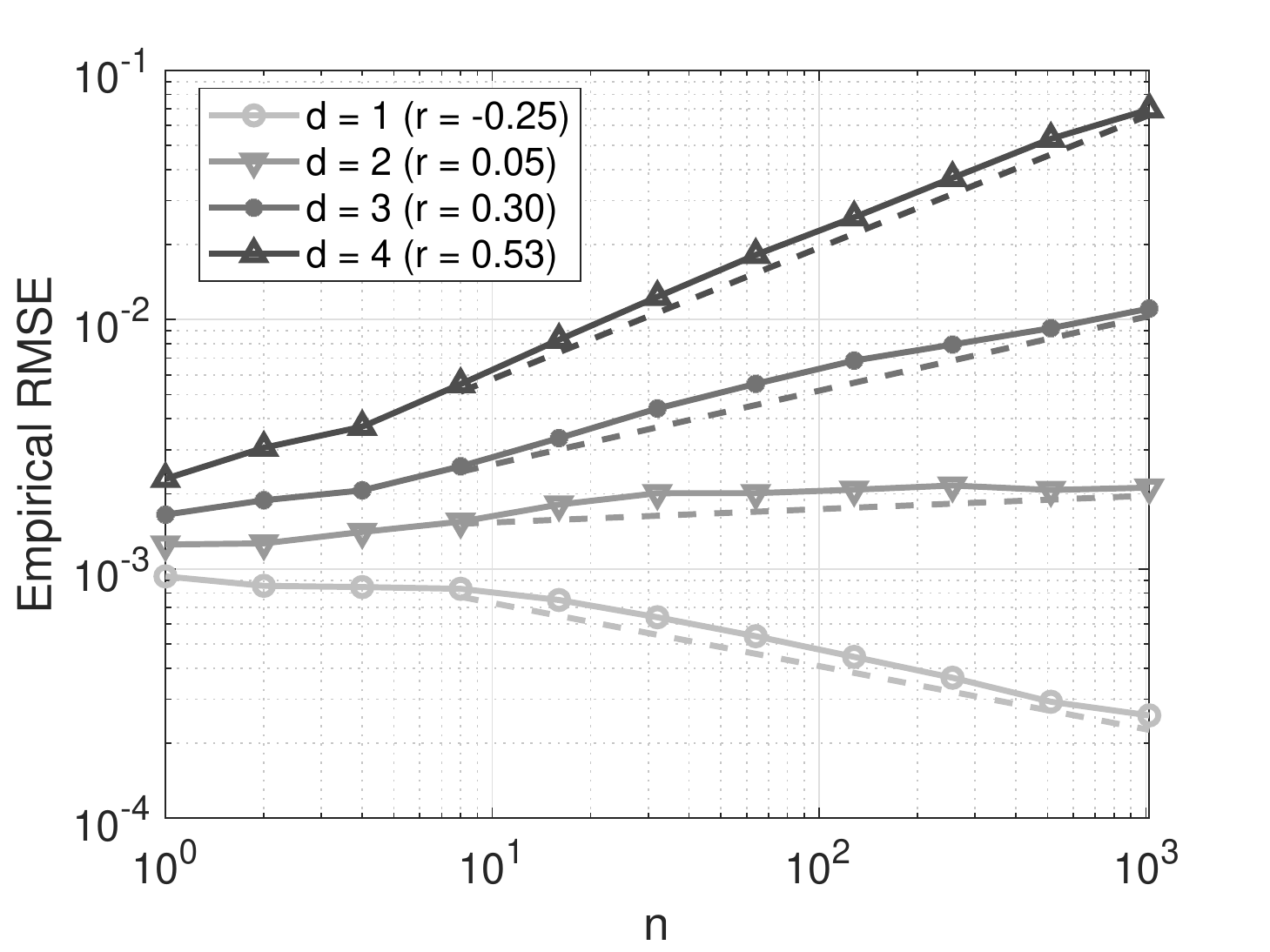}
\hfill
\includegraphics[width = 0.49\textwidth]{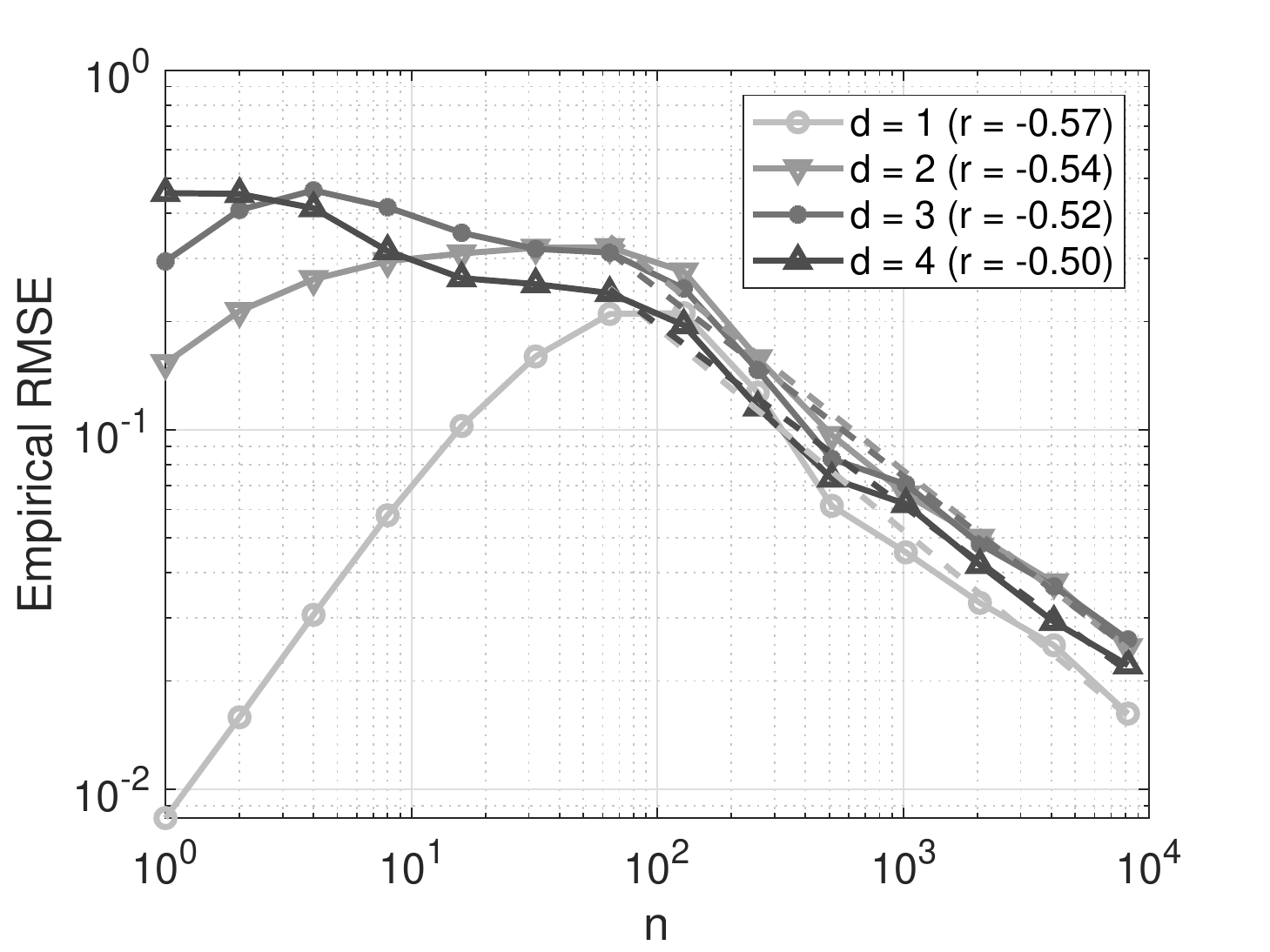}
\caption{Growth and decay of the empirical RMSE of prior-based (left) and Laplace-based (right) importance sampling for the example in Section \ref{sec:Exam_Alg} for decaying noise level $n^{-1}$ and various dimensions $d$.}
\label{fig:exam_alg_IS}
\end{figure}

\emph{Results for quasi-Monte Carlo:} 
We use $N = 2^{10}$ quasi-Monte Carlo points for prior- and Laplace-based QMC.
For the Laplace-case we employ a truncation parameter of $\tau = 10^{-6}$ and discard all transformed points outside of the domain $[-\frac 12, \frac 12]^d$.
Again, we run $1,000$ random shift simulations for both QMC methods and estimate the empirical RMSE.
However, for QMC we report the relative RMSE, since, for example, the decay of the normalization constant $Z_n\in\mc O(n^{-d/2})$ dominates the growth of the absolute error of prior QMC integration for the normalization constant.
In Figure \ref{fig:exam_alg_QMC_prior} and \ref{fig:exam_alg_QMC_LA} we display the resulting relative RMSE for the quantity related integral $Z_n'$, the normalization constant $Z_n$, i.e.,
\[
	Z'_n	 = \int_{[-\frac 12, \frac 12]^d} f(x) \ \exp(-n \Phi(x))\ \mu_0(\d x),
	\;
	Z_n	 = \int_{[-\frac 12, \frac 12]^d} 1 \ \exp(-n \Phi(x))\ \mu_0(\d x),
\]
and the resulting ratio $\frac {Z'_n}{Z_n}$ for increasing $n$ and various $d$ for prior-based and Laplace-based QMC.
For prior-based QMC we observe for dimensions $d \geq 2$ a algebraic growth of the relative error.
In the previous subsection we have proven that the corresponding classical error bound will explode which does not necessary imply that the error itself explodes---as we can see for $d=1$. 
However, this simple example shows that also the error will often grow algebraically with increasing $n$.
For the Laplace-based QMC we observe on the other hand in Figure \ref{fig:exam_alg_QMC_LA} a decay of the relative empirical RMSE.
By Corollary \ref{cor:QMCprec} we can expect that the relative errors stay bounded as $n\to \infty$.
This provides motivation for a further investigation.
In particular, we will analysize the QMC ratio estimator for $\frac{Z'_n}{Z_n}$ in a future work.
\begin{figure}[htb]
\includegraphics[width = \textwidth]{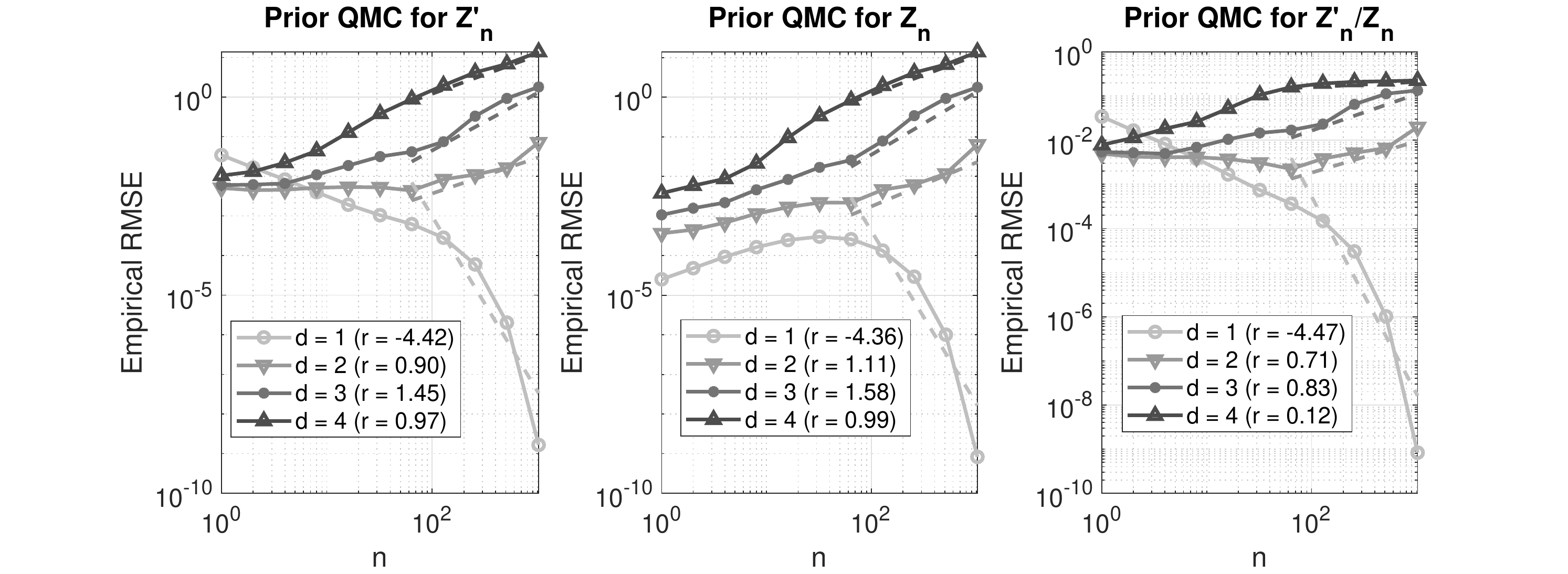}
\caption{Empirical relative RMSE of prior-based quasi-Monte Carlo for the example in Section \ref{sec:Exam_Alg} for decaying noise level $n^{-1}$ and various dimensions $d$.}
\label{fig:exam_alg_QMC_prior}
\end{figure}
\begin{figure}[htb]
\includegraphics[width = \textwidth]{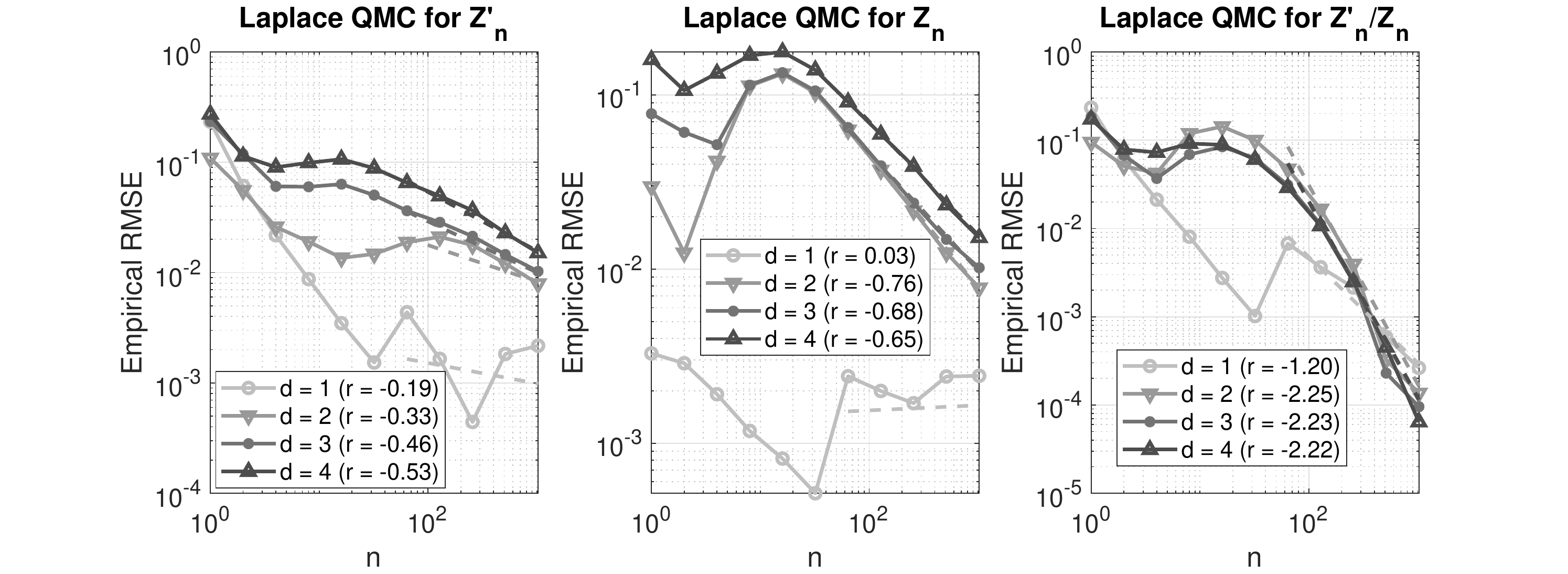}
\caption{Empirical relative RMSE of Laplace-based quasi-Monte Carlo for the example in Section \ref{sec:Exam_Alg} for decaying noise level $n^{-1}$ and various dimensions $d$.}
\label{fig:exam_alg_QMC_LA}
\end{figure}

\subsubsection{Bayesian inference for an elliptic PDE}
In the following we illustrate the preconditioning ideas from the previous section by 
Bayesian inference with differential equations.
To this end we consider the following model parametric elliptic problem
\begin{eqnarray}\label{eq:forwardex}
-\mbox{div}(\hat u_d \nabla q)=f \quad \mbox{in } D\coloneqq[0,1]\, , \ q=0  \quad \mbox{in } \partial D\, , 
\end{eqnarray}
with $f(t)=100\cdot t$, $t\in[0,1]$, and diffusion coefficient
\[
  \hat u_d(t) = \exp\left(\sum_{k=1}^dx_k \ \psi_k(t) \right)\, , \qquad d\in\{1,2,3\},
\]
where $\psi_k(t)= \frac{0.1}{k} \sin(k\pi t)$ and the $x_k \in \bbR$, $k=1,\ldots,d$, are to be inferred by noisy observations of the solution $q$ at certain points $t_j \in [0,1]$.
For $d=1,2$ these observations are taken at $t_1 = 0.25$ and $t_2 = 0.75$ and for $d=3$ they are taken at $t_j \in \{0.125,0.25,0.375,0.6125,0.75,0.875\}$.
We suppose an additive Gaussian observational noise $\epsilon\sim\mathcal N(0, \Gamma_n)$ with noise covariance $\Gamma_n = n^{-1}\Gamma_{obs}$ and $\Gamma_{obs}\in\bbR^{2\times 2}$ or $\Gamma_{obs}\in\bbR^{7\times 7}$, respectively, specified later on.
In the following we place a uniform and a Gaussian prior $\mu_0$ on $\bbR^d$ and would like to integrate w.r.t.~the resulting posterior $\mu_n$ on $\bbR^d$ which is of the form \eqref{equ:mu_simple_2} with
\[
	\Phi(x)
	=
	\frac12
	\|y -F(x)\|^2_{\Gamma^{-1}_{obs}}
\]
where $F\colon \bbR^d \to \bbR^2$ for $d=1,2$, and $F\colon \bbR^d \to \bbR^7$ for $d=3$, respectively, denotes the mapping from the coefficients $x\coloneqq (x_k)_{k=1}^d$ to the observations of the solution $q$ of the elliptic problem above and the vector $y \in \bbR^2$ or $y\in\bbR^7$, respectively, denotes the observational data resulting from $y = F(x) + \epsilon$ with $\epsilon$ as above.
Our goal is then to compute the posterior expectation (i.e., w.r.t.~$\mu_n$) of the following quantity of interest $f\colon \bbR^d\to \bbR$: $f(u)$ is the value of the solution $q$ of the elliptic problem at $t=0.5$. 

\paragraph{Uniform prior.}
We place a uniform prior $\mu_0 = \mathcal U([-\frac 12, \frac 12]^d)$ for $d=1,2$ or $3$ and choose $\Gamma_{obs} = 0.01 I_2$ for $d=1,2$ and $\Gamma_{obs} = 0.01 I_7$ for $d=3$. 
We display the resulting posteriors $\mu_n$ for $d=2$ in Figure \ref{fig:unipost} illustrating the concentration effect of the posterior for various values of the noise scaling $n$ and the resulting transformed posterior with $\Phi\circ g_n$ based on Laplace approximation.
The truncation parameter is set to $\tau=10^{-6}$.
{We observe the almost quadratic behavior of the preconditioned posterior, as expected from the theoretical results.}
\begin{figure}[htb]
\includegraphics[width = 0.98\textwidth]{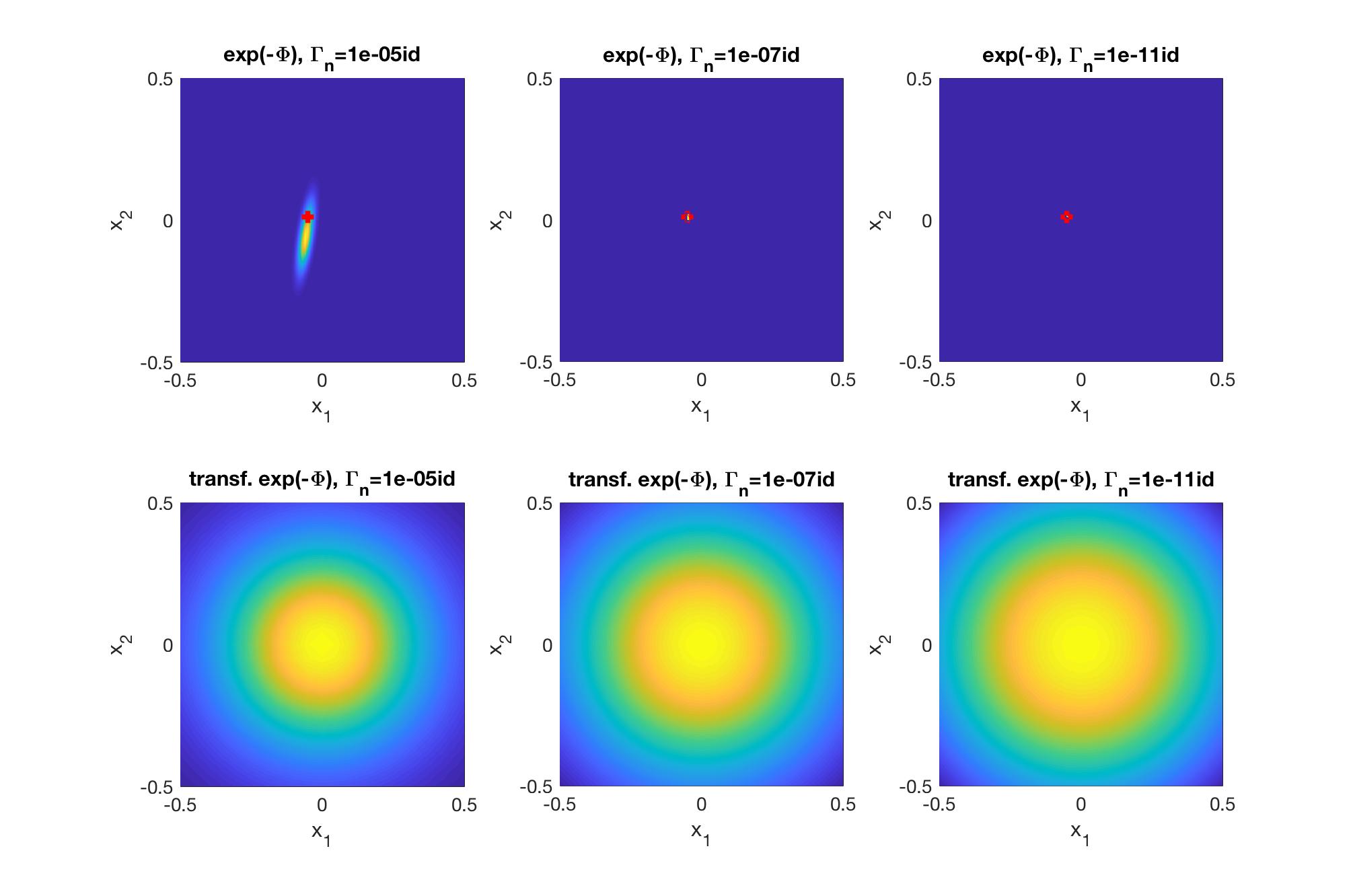}
\caption{The first row shows the posterior distribution for various values of the noise scaling, the second row shows the corresponding preconditioned posteriors based on Laplace approximation, 2d test case, uniform prior distribution, $\tau=10^{-6}$.}
\label{fig:unipost}
\end{figure}

We are now interested in the numerical performance of the Importance Sampling and QMC method based on the prior distribution compared to the performance of the preconditioned versions based on Laplace approximation. 
The QMC estimators are constructed by an off-the-shelf generating vector (order-2 randomly shifted weighted Sobolev space), which can be downloaded from \url{https://people.cs.kuleuven.be/~dirk.nuyens/qmc-generators/} (exod2\_base2\_m20\_CKN.txt). 
The reference solution used to estimate the error is based on a (tensor grid) trapezoidal rule with $10^6$ points in 1D, $4\cdot 10^6$ points in 2D in the original domain, i.e., the truncation error is quantified and in the transformed domain in 3D with $10^6$ points. 
Figure \ref{fig:uni} illustrates the robust behavior of the preconditioning strategy w.r.t. the scaling $1/n$ of the observational noise. 
Though we know from the theory that in the low dimensional case (1D, 2D), the importance sampling method based on the prior is expected to perform robust, we encounter numerical instabilities due to the finite number of samples used for the experiments. 
The importance sampling results are based on $10^6$ sampling points, the QMC results on $8192$ shifted lattice points with $2^6$ random shifts.
\begin{figure}[htb]
\centering
\includegraphics[width = 0.8\textwidth]{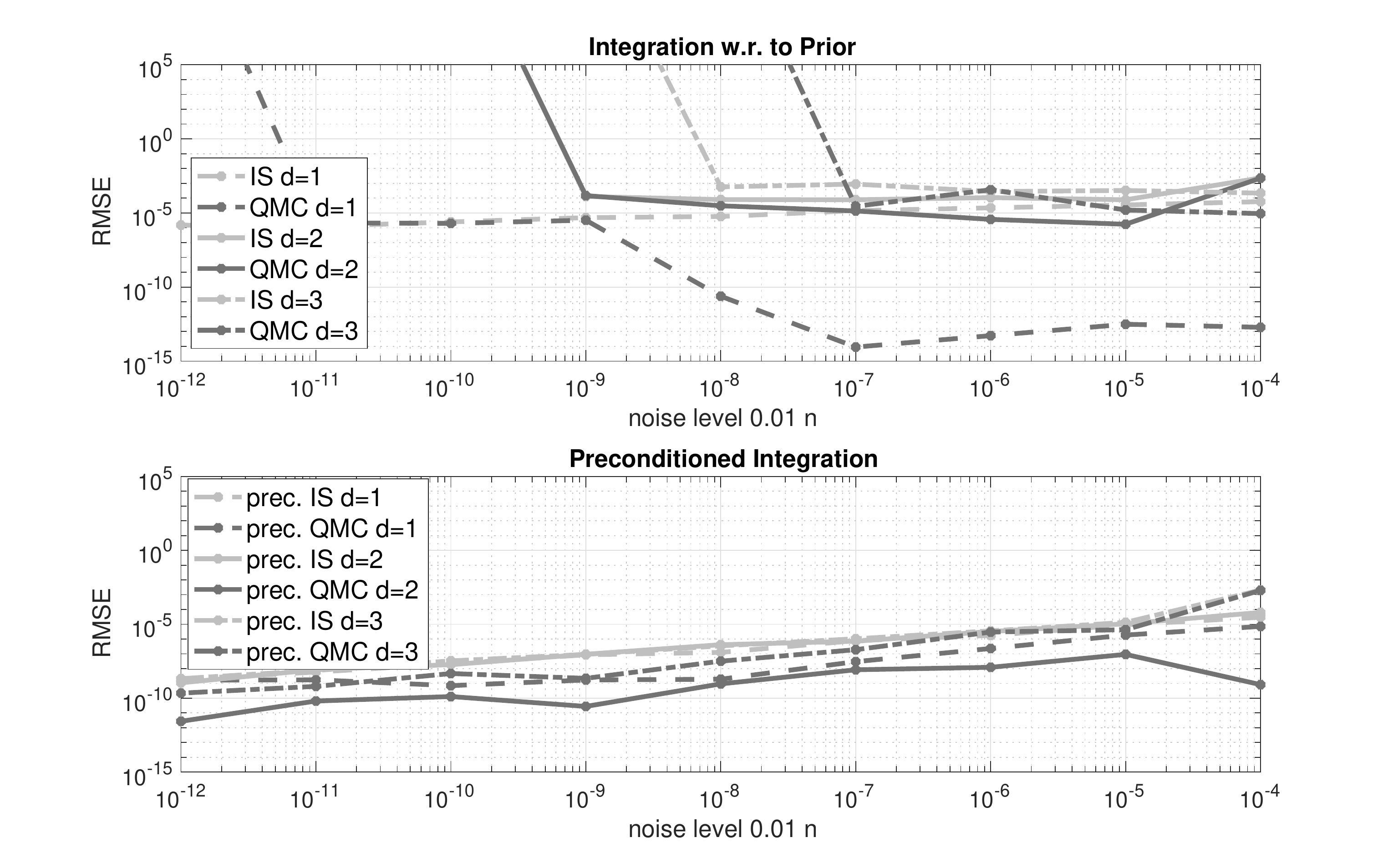}

\caption{The (estimated) root mean square error (RMSE) of the approximation of the quantity of interest for different noise levels ({$n=10^{2},\ldots,10^{10}$}) using the Importance Sampling strategy and QMC method for $d= 1,2,3$.}
\label{fig:uni}
\end{figure}

Figure \ref{fig:uniQMC} shows the RMSE of the normalization constant $Z_n$ using the preconditioned QMC approach with respect to the noise scaling $1/n$. 
We observe a numerical confirmation of the predicted dependence of the error w.r.t. the dimension (cp. Corollary \ref{cor:QMCprec}).

 \begin{figure}[htb]
\centering
\includegraphics[width = 0.8\textwidth]{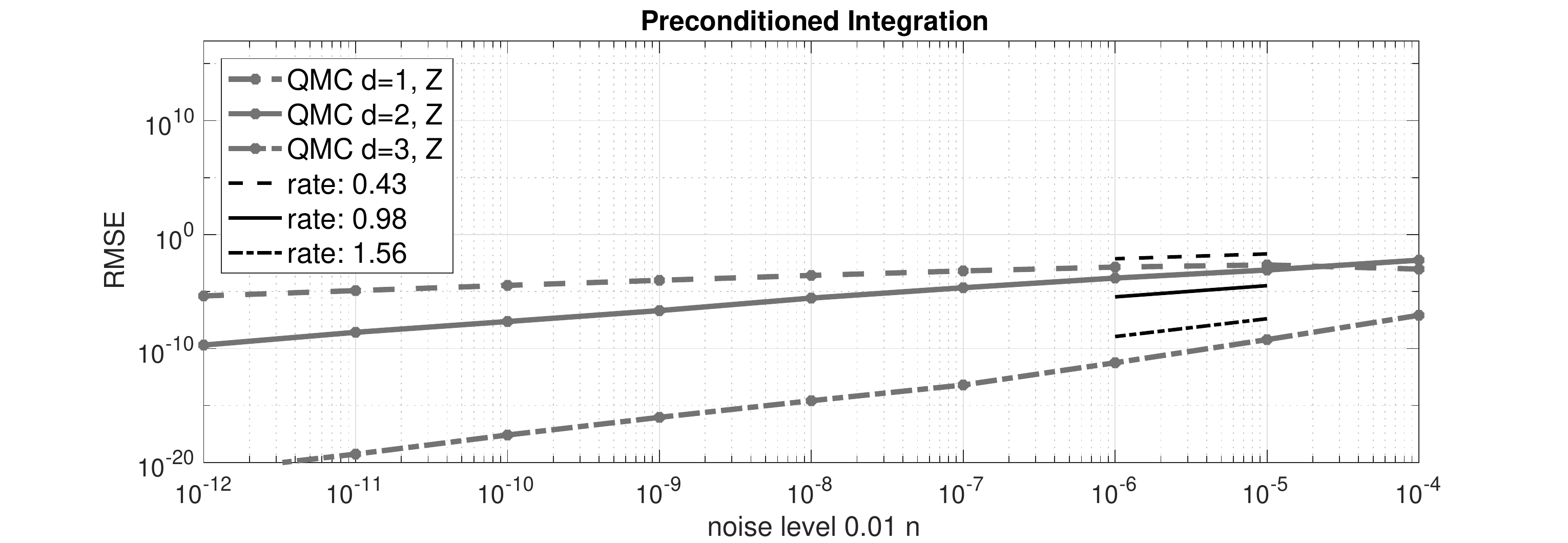}

\caption{The (estimated) root mean square error (RMSE) of the approximation of the normalization constant $Z$ for different noise levels ({$n=10^{2},\ldots,10^{10}$}) using the preconditioned QMC method for $d=1,2,3$.  }
\label{fig:uniQMC}
\end{figure}

We remark that the numerical experiments for the ratio suggest a behavior $n^{-1/2}$, i.e., a rate independent of the dimension $d$, of the RMSE for the preconditioned QMC approach, cp. Figure \ref{fig:uni}. This will be subject to future work.

\paragraph{Gaussian prior.}
{We choose as prior $\mu_0 = \mathcal N(0,I_2)$ for the coefficients $x=(x_1,x_2) \in \bbR^2$ for $\hat u_2$ in the elliptic problem above.}
For the noise covariance we set this time $\Gamma_{obs}=I_2$.
The performance of the prior based and preconditioned version of Importance Sampling is depicted in Figure \ref{fig:Gaussa}. Clearly, the Laplace approximation as a preconditioner improves the convergence behavior; we observe a robust behavior w.r.t. the noise level.

The convergence of the QMC approach is depicted in Fig \ref{fig:Gaussb}, showing a consistent behavior with the considerations in the previous section.

\begin{figure}[htb]
\centering
\includegraphics[width = 0.8\textwidth]{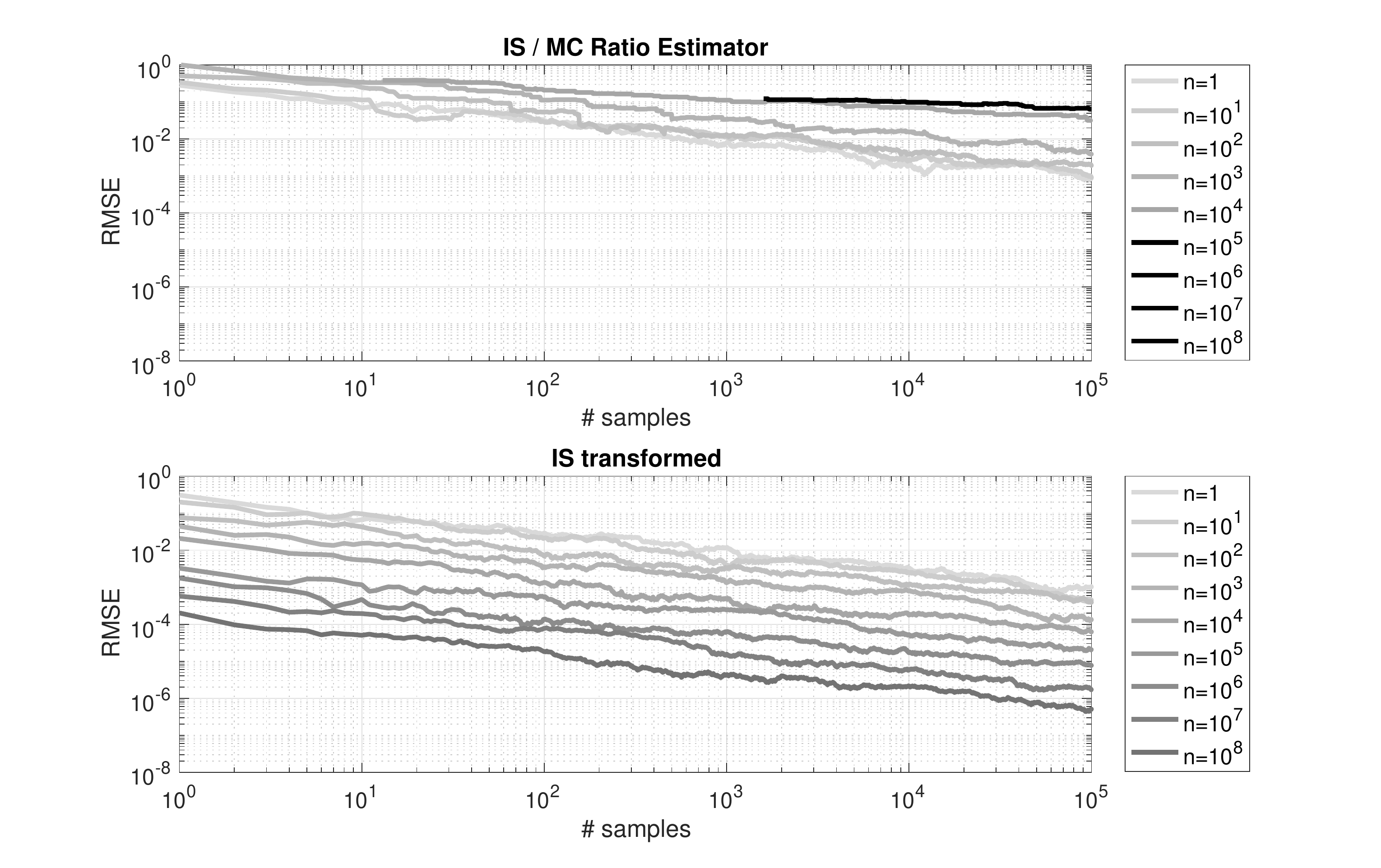}
\caption{The (estimated) root mean square error (RMSE) of the approximation of the quantity of interest for different noise levels ({$n=10^{0},\ldots,10^{8}$}) using the Importance Sampling strategy. The first row shows the result based on prior information (Gaussian prior), the second row the results using the Laplace approximation for preconditioning. The reference solution is based on a tensor grid Gauss--Hermite rule with $10^4$ points for the preconditioned integrand using the Laplace approximation.}
\label{fig:Gaussa}
\end{figure}
\begin{figure}[htb]
\centering
\includegraphics[width = 0.8\textwidth]{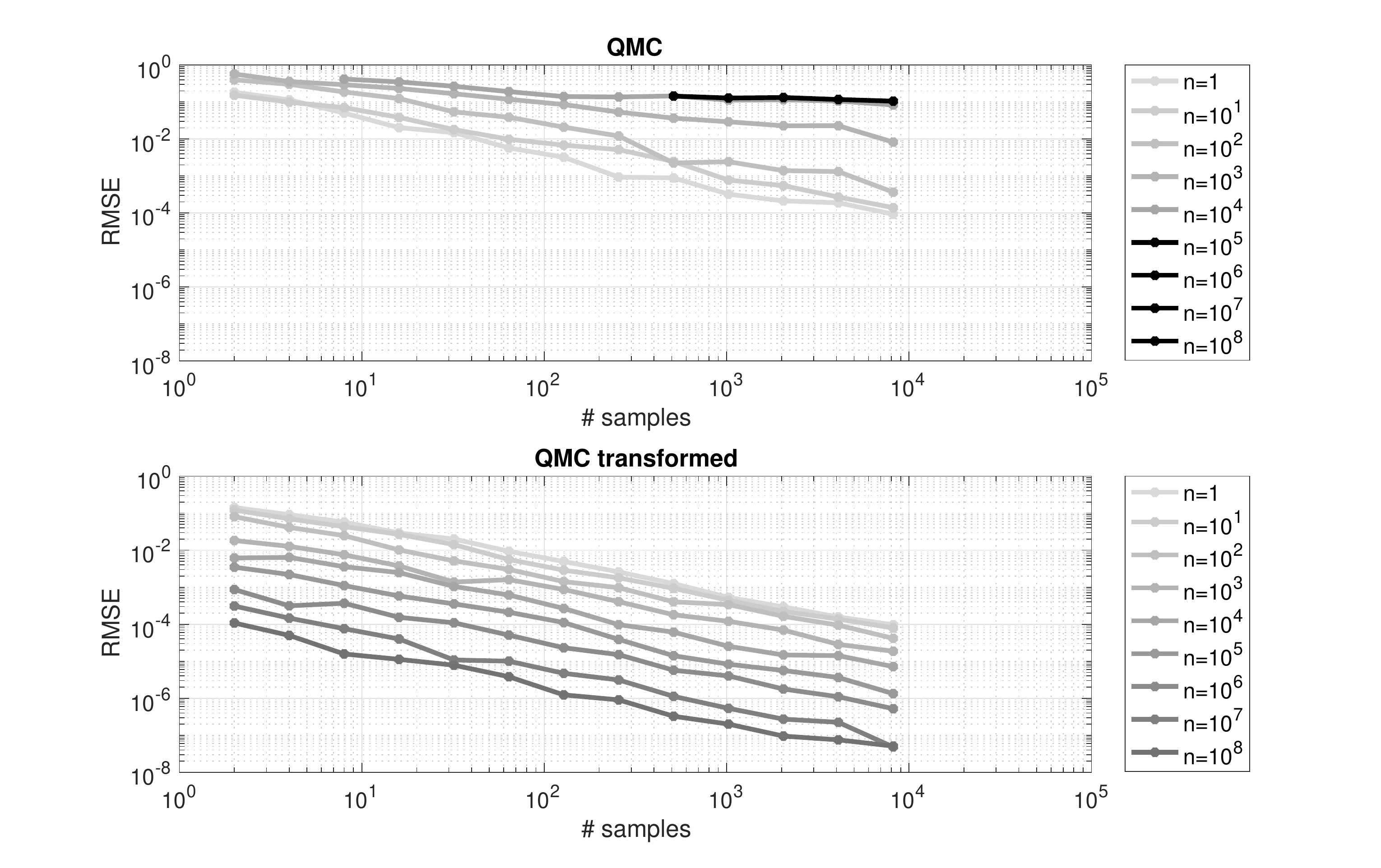}

\caption{The (estimated) root mean square error (RMSE) of the approximation of the quantity of interest for different noise levels ({$n=10^{0},\ldots,10^{8}$}) using the  QMC method (below). The first row shows the result based on prior information (Gaussian prior), the second row the results using the Laplace approximation for preconditioning. The reference solution is based on a tensor grid Gauss--Hermite rule with $10^4$ points for the preconditioned integrand using the Laplace approximation.}
\label{fig:Gaussb}
\end{figure}

\section{Conclusions and Outlook to Future Work}
This paper makes a number of contributions in the development of {numerical} methods for Bayesian inverse problems, which are robust w.r.t.~the size of the measurement noise {or the concentration of the posterior measure, respectively}. 
We analyzed the convergence of the Laplace approximation to the posterior distribution in Hellinger distance. 
This forms the basis for the design of variance robust methods. 
In particular, we proved that Laplace-based importance sampling behaves well in the small noise or large data size limit, respectively. 
For uniform priors, Laplace-based QMC methods have been developed with theoretically and numerically proven errors decaying {with the noise level or concentration of the measure, repestively.} 
Some future directions of this work include the development of noise robust Markov chain Monte Carlo methods and the combination of dimension independent and noise robust strategies. This will require the study of the Laplace approximation in infinite dimensions in a suitable setting. Finally, we could study in more details the error in the ratio estimator using Laplace-based QMC methods. The use of higher order QMC methods has been proven to be a promising direction for a broad class of Bayesian inverse problems and the design of noise robust versions is an interesting and potentially fruitful research direction.

\paragraph{Acknowledgements.}
The authors \rev{are grateful to} the Isaac Newton Institute for Mathematical Sciences for support and hospitality during the programme Uncertainty quantification for complex systems: theory and methodologies when work on this paper was undertaken. This work was supported by: EPSRC grant numbers EP/K032208/1 and EP/R014604/1.
\rev{Moreover, the authors would like to thank the anonymous referees and Remo Kretschmann for helpful comments which improved the paper significantly.} Furthermore, BS was supported by the DFG project 389483880 and by the DFG RTG1953 "Statistical Modeling of Complex
Systems and Processes".

\appendix

\section{Concentration of the Laplace Approximation}\label{sec:TIS}

Due to the well-known \emph{Borell-TIS inequality} for Gaussian measures on Banach spaces \cite[Chapter 3]{LedouxTalagrand2002} we obtain the following useful concentration result for the Laplace approximation.

\begin{proposition}
\label{propo:con_LA_Phi_n}
Let Assumption \ref{assum:LA_0} and Assumption \ref{assum:LA_Conv} be satisfied.
Then, for any $r>0$ there exists an $c_r >0$ such that
\[
	\LAn\left( B^c_r(x_n) \right)
	\in
	\mc O\left(\e^{-c_r n} \right),
\]
{where $B^c_r(x_n) = \{x\in\bbR^d\colon \|x-x_n\|>r\}$.}
\end{proposition}
\begin{proof}
Let $X_n \sim \mc N\left(0, n^{-1} C_n\right)$.
Then
\[
	\LAn\left( B^c_r(x_n) \right)
	= 
	\bbP\left( \|X_n\| > r \right).
\]
We now use the well-known concentration of Gaussian measures, namely,
\[
	\bbP\left( \left| \|X_n\| - \ev{\|X_n\|}\right| > r \right) \leq 2\exp\left( -\frac{r^2}{2\sigma_n^2}\right),
\]
where $\sigma_n^2 \coloneqq \sup_{\|x\| \leq 1} \ev{|x^\top X_n|^2}$, see \cite[Chapter 3]{LedouxTalagrand2002}.
There holds $\sigma_n^2 = \lambda_\text{max}(n^{-1}C_n) = \|n^{-1}C_n\|$ and we get
\begin{align*}
	\bbP\left( \|X_n\| > r + \ev{\|X_n\|} \right)
	& = 
	\bbP\left( \|X_n\| - \ev{\|X_n\|}> r \right)
	\leq \bbP\left( \left| \|X_n\| - \ev{\|X_n\|} \right| > r\right)\\
	& \leq 2\exp\left( - n \frac{r^2}{2\|C_n\|}\right).
\end{align*}
Due to Assumption \ref{assum:LA_Conv}, i.e., $C_n\to H^{-1}_\star > 0$, there exists a finite constant $0<c$ such that $\|C_n\|\geq c$ for all $n\in\bbN$.
Analogously, there exist a constant $K<\infty$ such that $\tr(C_n) \leq K$ for all $n$.
The latter implies
\[
	\ev{\|X_n\|}\leq\ev{\|X_n\|^2}^{{1/2}} = \tr(n^{-1}C_n)^{1/2} \leq n^{-1/2}\sqrt{K}.
\]
Hence, for an arbitrary $r$ let $n_0$ be such that $\ev{\|X_n\|}\leq r/2$ for all $n\geq n_0$, which yields
\begin{align*}
	\bbP\left( \|X_n\| > r \right)
	\leq 
	\bbP\left( \|X_n\| > \frac r2 + \ev{\|X_n\|} \right)
	\leq \exp\left( -n \frac{r^2}{8c}\right).	
\end{align*}
\end{proof}


\section{Proofs}\label{sec:appendix_proofs}
We collect the rather technical proofs in this appendix.

\subsection{Proof of Lemma \ref{lem:conv_LA}}\label{sec:proof_lem}
Recall that we want to bound
\begin{align*}
	J_0(n) & \coloneqq\LAn(\S_0^c),\\
	J_1(n) & \coloneqq\int_{B_r(x_n) \cap \S_0} \left|\e^{-nR_n(x)/p} -1 \right|^{p}\ \LAn(\d x),\\
	J_2(n) & \coloneqq \int_{B^c_r(x_n)\cap \S_0} \left|\e^{-n R_n(x)/p} -1 \right|^p\ \LAn(\d x)
\end{align*}
where $R_n(x) \coloneqq I_n(x) - \tilde I_n(x) = I_n(x) - I_n(x_n) - \frac 12 \|x-x_n\|^2_{C^{-1}_n}$ and $r>0$ is an at the moment arbitrary radius which will be specified in the following first paragraph.

\paragraph{Bounding $J_1$.} 
Due to $\Phi_n, \log \pi_0 \in C^3(\S_0; \bbR)$, we have that for any $x\in B_r(x_n)$ there exists a $\xi_{x,n} \in B_r(x_n)$ such that
\[
	\left|R_n(x)\right| \leq \frac 16 \|\nabla^3 I_n(\xi_{x,n})\| \ \|x-x_n\|^3.
\]
Moreover, since $x_n \to x_\star$ there exists an $0<r_0<\infty$ such that
\[
	B_r(x_n) \subset B_{r_0}(0)\qquad \forall n\in\bbN.
\]
Hence, the local uniform boundedness of $\|\nabla^3 I_n(\cdot)\|$, see Assumption \ref{assum:LA_Conv}, implies the existence of a finite $c_3 > 0$ such that for sufficiently large $n$, i.e., $n\geq n_r$, we have
\[
	|R_n(x)| \leq c_3\ \|x-x_n\|^3
	\qquad \forall x\in B_r(x_n). 
\]
Thus, we obtain, due to $|\e^{-t} -1| = 1 - \e^{-t}  \leq \e^{t} -1$ for $t\geq0$,
\[
	\left|\e^{-n R_n(x)/p} -1\right|
	\leq
	\e^{n c_3 \ \|x-x_n\|^3/p} -1,
\]
which yields
\begin{align*}
	J_1(n) 
	& \leq \int_{B_r(x_n)} \left(\e^{n c_3 \ \|x-x_n\|^3/p} -1 \right)^p\, \LAn(\d x)\\
	& \leq \int_{B_r(x_n)} \left(1 - \e^{- n c_3 \ \|x-x_n\|^3/p} \right)^p\ \e^{-n (\frac 12 \|x-x_n\|^2_{C^{-1}_n} - c_3 \|x-x_n\|^3)}\ \frac{\d x}{\tilde Z_n}.
\end{align*}
Now, since $C^{-1}_n \to H_\star > 0$, there exists for sufficiently large $n$ a $\gamma >0$ such that
\[
	\frac 12 \|x-x_n\|^2_{C^{-1}_n} \geq \gamma \|x-x_n\|^2
	\qquad
	\forall x\in\bbR^d.
\]
Hence, for $x\in B_r(x_n)$, i.e., $\|x-x_n\|\leq r$, we get
\[
	\frac 12 \|x-x_n\|^2_{C^{-1}_n} - c_3\|x-x_n\|^3
	\geq (\gamma -c_3r) \|x-x_n\|^2.
\]
By choosing $r \coloneqq \frac \gamma{2c_3}$ we obtain further
\[
	J_1(n) 
	\leq \int_{B_r(x_n)} \left(1 - \e^{- n c_3 \ \|x-x_n\|^3/p} \right)^p\ \e^{- n \gamma \|x-x_n\|^2/2}\ \frac{\d x}{\tilde Z_n}.
\]
Let us now introduce the auxiliary Gaussian measure $\nu_{n} \coloneqq \mc N \left(0, \frac 1{n\gamma} I\right)$ with which we get
\begin{align*}
	J_1(n) 
	& \leq \frac{\sqrt{\det(2\pi (n\gamma)^{-1}\, I)}}{\tilde Z_n} \int_{\bbR^d} \left(1 - \e^{- nc_3\ \|x\|^3/p} \right)^p\ \nu_n(\d x).
\end{align*}
There holds now
\begin{align*}
	\lim_{n\to \infty}
	\frac{\sqrt{\det(2\pi (n\gamma)^{-1}\, I)}}{\tilde Z_n}
	& =
	\lim_{n\to \infty}
	\frac{n^{-d/2}\sqrt{\det(2\pi \gamma^{-1}I)}}{n^{-d/2}\sqrt{\det(2\pi C_n)}}
	=
	\frac{\sqrt{\det(\gamma^{-1}I)}}{\sqrt{\det\left(2\pi H_\star\right)}} < \infty
\end{align*}
due to the continuity of the determinant and $H_\star >0$.
Moreover, since $1-\e^{-t}\leq t$ for $t\geq0$ we obtain with $\xi \sim \mc N(0, I)$ that 
\begin{align*}
	\int_{\bbR^d} \left(1 - \e^{- nc_3 \ \|x\|^3/p} \right)^p\  \nu_n(\d x)
	&\leq
	\int_{\bbR^d} (nc_3/p)^p \ \|x\|^{3p}\  \nu_n(\d x)
	= n^p (c_3/p)^p \ \ev{\|(\gamma n)^{-1/2} \xi\|^{3p}}\\
	& = n^{-p/2} \frac {c^p_3}{p^p\gamma^{3p/2}} \ \ev{\|\xi\|^3} \in \mc O(n^{-p/2}).
\end{align*}
This yields $J_1(n)\in \mc O(n^{-p/2})$ for the particular choice $r = \frac \gamma{2c_3}$. In the following two paragraphs we will use exactly this particular radius.

\paragraph{Bounding $J_0$.}
Due to Assumption \ref{assum:LA_Conv}, we have $x_n \to x_\star$ as $n\to \infty$.
Hence, there exists an $n_0 < \infty$ such that $\|x_n - x_\star\| \leq r/2$ for all $n\geq n_0$.
This implies by Assumption \ref{assum:LA_Conv}
\[
	B_{r/2}(x_n) \subseteq 
	B_{r}(x_\star) \subseteq 
	\S_0
\]
and, hence, $\S_0^c \subseteq B^c_{r/2}(x_n)$.
By Proposition \ref{propo:con_LA_Phi_n}, we obtain
\begin{align*}
	J_0(n) & =\LAn(\S_0^c)
	\leq 
	\LAn( B^c_{r/2}(x_n))
	\in \mc O(\e^{-c_{r/2}n}),
\end{align*}
with a $0<c_{r/2}<\infty$.

\paragraph{Bounding $J_2$.}
We divide the set $B_r^c(x_n)\cap \S_0$ into several subsets in order to bound $J_2(n)$.
First, we define
\[
	\P_n
	\coloneqq
	\{x\in B^c_r(x_n) \cap \S_0 \colon R_n(x) > 0\},
	\qquad
	\P^c_n
	\coloneqq 
	\{x\in B^c_r(x_n) \cap \S_0 \colon R_n(x) \leq 0\},
\]
i.e., $B^c_r(x_n) \cap \S_0= \P_n \dot{\cup} \P^c_n$, and notice that
\[
	\left|\e^{-n R_n(x)/p} -1 \right|
	=
	1 - \e^{-n R_n(x)/p}
	\leq 1
	\qquad
	\forall x \in \P_n.
\]
Hence, due to Propostion \ref{propo:con_LA_Phi_n} there exists a $c_{r} \in (0, \infty)$ such that
\begin{align*}
	\int_{\P_n} \left|\e^{-n R_n(x)/p} -1 \right|^p\ \LAn(\d x)
	& \leq \LAn(\P_n) 
	\leq \LAn(B^c_r(x_n))
	\in \mc O(\e^{-c_{r}n}).
\end{align*}
\begin{align*}
	\int_{\PM^c_n} \left|\e^{-n R_n(x)/p} -1 \right|^p\ \LAn(\d x)
	& \leq \int_{\PM^c_n} \e^{-n R_n(x)} \ \LAn(\d x)
	\rev{ = \frac 1{\widetilde Z_n} \int_{\PM^c_n} \e^{-n I_n(x)} \ \d x}
\end{align*}

\rev{
We now prove that 
\begin{equation}\label{equ:tail_bound_proof}
	\int_{\PM^c_n} \e^{-n I_n(x)} \ \d x
	\in \mc O(\exp(- \epsilon \delta_r n ) )
\end{equation}
with $\epsilon\in(0,1)$ as in Assumption \ref{assum:LA_Conv}.
To this end, we observe that due to $I_n(x) \geq \delta_r$ for all $x\in \PM^c_n \subset B^c_r(x_n)$ the functions
\[
	g_n(x) := \e^{-n I_n(x)} \e^{n \epsilon \delta_r } \leq \e^{- n (1-\epsilon) \delta_r}
	\qquad
	x \in \PM^C_n,
\]
converge pointwise to zero as $n\to\infty$. 
Moreover, $I_n(x) \geq \delta_r$ yields
\[
	g_n(x) = 
	\e^{-n I_n(x)}  \e^{n \epsilon \delta_r}
	\leq
	 \e^{-n (1-\epsilon) I_n(x)}
	\leq q^{1-\epsilon}(x)
\]
with the bounding function $q$ introduced in Assumption \ref{assum:LA_Conv}.
Thus, since $q^{1-\epsilon}$ is integrable we obtain by Lebesgue's dominated convergence theorem
\[
	\e^{n \epsilon \delta_r} \int_{\PM^c_n} \e^{-n I_n(x)} \ \d x
	= \int_{\PM^c_n} g_n(x) \ \d x
	\to 0
\]
as $n\to\infty$.
Hence, \eqref{equ:tail_bound_proof} holds.
Since $\widetilde Z_n \in \mc O(n^{-d/2})$ we get in summary, 
\begin{align*}
	J_2(n) 
	& =
	\int_{\PM_n} \left|\e^{-n R_n(x)/p} -1 \right|^p\ \LAn(\d x)
	+ 
	\int_{\PM^c_n} \left|\e^{-n R_n(x)/p} -1 \right|^p\ \LAn(\d x)\\
	& \leq
	\int_{\PM_n} 1^p\ \LAn(\d x)
	+ 
	\frac 1{\widetilde Z_n} \int_{\PM^c_n} \e^{-n I_n(x)}\ \d x\\
	& \in \mc O(\e^{-c_{r}n} + \e^{- n \epsilon \delta_r} n^{d/2})
	\subseteq \mc O(\e^{- n c_{r,\epsilon}} n^{d/2} )
\end{align*}
with $c_{r,\epsilon} := \min\{c_r, \epsilon\delta_r\}>0$.
}

\subsection{Proof of Theorem \ref{theo:conv_H_star}}
\label{sec:conv_H_star_proof}
A straightforward calculation, see also \cite[Exercise 1.14]{Pardo2006}, yields that for Gaussian measures $\mc N_{a,Q} \coloneqq \mc N(a,Q)$, $\mc N_{b,Q} \coloneqq \mc N(b, Q)$ and $\mc N_{a,R} \coloneqq \mc N(a, R)$ we have
\begin{align*}
	d_\mathrm{H}(\mc N_{a,Q}, \mc N_{b,Q}) 
	& = 
	\sqrt{2 - 2 \exp\left( \frac 18 \|a-b\|^2_{Q^{-1}}\right)},\\
	d_\mathrm{H}(\mc N_{a,Q}, \mc N_{a,R}) 
	& = 
	\sqrt{2- \frac2{\det\left(\frac 12 (Q^{-1/2}R^{1/2} + Q^{1/2}R^{-1/2})\right)}},
\end{align*}

By \eqref{equ:Conv_mun_2} we obtain for $\tilde \mu_n \coloneqq \mc N(x_n, \frac 1n B_n)$ that
\[
		\lim_{n\to \infty} d_\mathrm{H}(\mu_n, \tilde \mu_n) = 0
		\quad
		\text{ iff }
		\quad
		\lim_{n\to \infty} \det\left(\frac 12 (C^{-1/2}_n B_n^{1/2} + C_n^{1/2} B_n^{-1/2})\right) = 1.
\]
Due to the local Lipschitz continuity of the determinant and $\|C_n^{-1} - H_\star\|\to0$ we obtain the first statement for $\tilde \mu_n \coloneqq \mc N(x_n, \frac 1n B_n)$.
Furthermore, by the triangle inequality 
\begin{align*}
	d_\mathrm{H}(\mu_n, \tilde \mu_n)
	& \leq d_\mathrm{H}(\mu_n, \LAn) + d_\mathrm{H}(\LAn, \tilde \mu_n)\\
	& \leq  c n^{-1/2} + \sqrt{2} \left(1- \frac1{\det\left(\frac 12 (C_n^{-1/2} B_n^{1/2} + C_n^{1/2}B_n^{-1/2})\right)}\right)^{1/2}
\end{align*}
and exploiting the local Lipschitz continuity of $f(x) = \frac 1x$ and of the determinant, we obtain
\begin{align*}
	\left|\frac 1{\det(I)}- \frac1{\det\left(\frac 12 (C_n^{-1/2}B_n^{1/2} + C_n^{1/2}B_n^{-1/2})\right)}\right|
	& \leq c \|I - 0.5(C_n^{-1/2}B_n^{1/2} + C_n^{1/2}B_n^{-1/2})\|\\
	& \leq c \left(\|I - C_n^{-1/2}B_n^{1/2}\|  + \|I - C_n^{1/2}B_n^{-1/2}\|\right)
\end{align*}
with a generic constant $c>0$.
Moreover, due to the local Lipschitz continuity of the square root of a matrix, we get
\[
	\|I - C_n^{-1/2}B_n^{1/2}\|
	\leq \|C_n^{-1/2}\|\ \|C_n^{1/2} - B_n^{1/2}\|
	\leq c \|C_n^{1/2} - B_n^{1/2}\|
	\leq c \|C_n - B_n\|
\]
where we used that $\|C_n^{-1/2}\| \to \|H_\star^{1/2}\|$.
Furthermore, we get analogously that $\|I - C_n^{1/2}B_n^{-1/2}\| \leq c \|B_n - C_n\|$ using that $\|B_n^{-1/2}\| \to \|H_\star^{1/2}\|$.
Thus, the second statement of the first item follows by
\begin{align*}
	d_\mathrm{H}(\mu_n, \tilde \mu_n)
	& \leq d_\mathrm{H}(\mu_n, \LAn) + d_\mathrm{H}(\LAn, \tilde \mu_n) 
	\leq  c n^{-1/2} + c \left(2\|C_n - B_n\| \right)^{1/2} \in \mc (n^{-1/2}).
\end{align*}

The second item follows by applying the triangle inequality, expressing the Hellinger distance between $\mc N(x_n, \frac 1n B_n)$ and $\mc N(a_n, \frac 1n B_n)$ by
\begin{align*}
	d_\mathrm{H}(\mc N_{a_n, n^{-1}B_n}, \mc N_{x_n, n^{-1}B_n}) 
	=	
	\sqrt{2} \sqrt{1 -\exp\left( \frac n8 \|x_n-a_n\|^2_{B^{-1}_n}\right)}
\end{align*}
and estimating
\[
\left|1 -\exp\left( \frac n8 \|x_n-a_n\|^2_{B_n^{-1}}\right)\right|
\leq
\frac n8 \|x_n-a_n\|^2_{B_n^{-1}}
\leq
c n \|x_n-a_n\|^2
\in \mc O(n^{-1}),
\]
where we used the fact that the spd matrices $B_n$ converge to the spd matrix $H_\star$, hence, the sequence of the smallest eigenvalue of $B_n$ is bounded away from zero.

\subsection{Proof of Lemma \ref{lem:Prior_QMC_norm}}
\label{sec:Prior_QMC_norm}
For the following proof we use the famous Faa di Bruno-formula for higher order derivatives of compositions given in \cite{Hardy2006}.
To this end, let $v\colon \bbR^d \to \bbR$ and $u \colon \bbR\to \bbR$ be sufficiently smooth functions and define $w\coloneqq u\circ v$, i.e., $w\colon \bbR^d\to\bbR$.
For a subset $\vnu \subset \{1,\ldots,d\}$ we consider the partial derivative $\frac{\partial^{|\vnu|}w}{\partial x_{\vnu}}$ where we set
\[
	\partial x_{\vnu}
	= \partial x_{\nu_1} \cdots \partial x_{\nu_{|\vnu|}},
	\quad
	\vnu = \{\nu_1,\ldots,\nu_{|\vnu|}\}
	\text{ with ordered }
	\nu_1 < \nu_2 < \cdots < \nu_{|\vnu|}.
\]
We then obtain (see \cite{Hardy2006})
\[
	\frac{\partial^{|\vnu|}}{\partial x_{\vnu}}w(x)
	=
	\sum_{P \in \Pi(\vnu)}
	\partial^{|P|} u(v(x))
	\cdot
	\prod_{B \in P}
	\frac{\partial^{|B|}}{\partial x_{B}} v(x),
\]
where $\Pi(\vnu)$ denotes the set of all partitions $P$ of the set $\vnu \subset \{1,\ldots,d\}$, $B\in P$ refers to running through the elements or blocks of the partition $P$ with $|B|$ denoting the cardinality of $B\subset \vnu$ and $|P|$ the number of blocks in $P$, and the same notational convention for $\partial x_B$ as above, i.e.,
\[
	\partial x_{B}
	= \partial x_{\nu_1} \cdots \partial x_{\nu_{|B|}},
	\quad
	B = \{\nu_1,\ldots,\nu_{|B|}\} \subset \vnu
	\text{ with }
	\nu_1 < \nu_2 < \cdots < \nu_{|B|}.
\]
By the application of the multivariate Faa di Bruno-formula to
\[
	\Theta_n(x)
	= \exp(-n \Phi(x))
	= u(v(x))
	\quad
	\text{ with }
	\quad
	u(t) = \exp(- n t),
	\;
	v(x) = \Phi(x),
\]
we obtain for $\vnu \subset \{1,\ldots,d\}$ that
\begin{align*}
	\frac{\partial^{|\vnu|}}{\partial v_\vnu} \Theta_n(x)
	& =
	\sum_{P \in \Pi(\vnu)}
	(-n)^{|P|} 
	\Theta_n(v)
	\cdot
	\prod_{B \in P}
	\frac{\partial^{|B|}}{\partial x_B} \Phi(x).
\end{align*}
By setting $x_{-\vnu} \coloneqq x_{\{1,\ldots,d\}\setminus\vnu}$ we get
\begin{align*}
	\|\Theta_n\|^2_\gamma
	& =
	\sum_{\vnu \subset \{1,\ldots,d\}}
	\frac 1{\gamma_\vnu} \int_{[-\frac 12,\frac 12]^{|\vnu|}} \left( \int_{[-\frac 12,\frac 12]^{d-|\vnu|}} \frac{\partial^{|\vnu|}}{\partial x_\vnu} \Theta_n(x) \ \d x_{-\vnu}  \right)^2 \ \d x_\vnu\\
	& \geq
	\sum_{\vnu = \{1,\ldots,d\}}
	\frac 1{\gamma_\vnu} \int_{[-\frac 12,\frac 12]^{|\vnu|}} \left( \int_{[-\frac 12,\frac 12]^{d-|\vnu|}} \frac{\partial^{|\vnu|}}{\partial x_\vnu} \Theta_n(x) \ \d x_{-\vnu}  \right)^2 \ \d x_\vnu\\
	& =
	\frac {1}{\prod_j \gamma_j} 
	\int_{[-\frac 12,\frac 12]^{d}} \left( \sum_{P \in \Pi_d}
	(-n)^{|P|} \Theta_n(v)
	\cdot
	\prod_{B \in P}
	\frac{\partial^{|B|}}{\partial x_B} \Phi(x)\right)^2  \ \d x,
\end{align*}
where we shortened $\Pi_d \coloneqq \Pi(\{1,\ldots,d\})$.
We will now investigate, how---i.e., to which power w.r.t.~$n$---
\begin{equation}\label{equ:Fn}
	F_d(n) \coloneqq \int_{[-\frac 12,\frac 12]^{d}} \left( \sum_{P \in \Pi_d}
	(-n)^{|P|} \Theta_n(x)
	\cdot
	\prod_{B \in P}
	\frac{\partial^{|B|}}{\partial x_B} \Phi(x)\right)^2  \ \d x
\end{equation}
decays as $n\to \infty$.
To this end, we write
\begin{align*}
	F_d(n) & =  
	\sum_{P \in \Pi_d} \sum_{\tilde P \in \Pi_d}
	(-n)^{|P|+|\tilde P|} \int_{[-\frac 12,\frac 12]^{d}} \Theta^2_n(x) \prod_{B \in P}
	\frac{\partial^{|B|}}{\partial x_B} \Phi(x) \ \prod_{\tilde B \in \tilde P}
	\frac{\partial^{|\tilde B|}}{\partial x_{\tilde B}} \Phi(x) \d x
\end{align*}
and apply in the following Laplace's method in order to derive the asymptotics of
\[
	\int_{[-\frac 12,\frac 12]^{d}}
	h_{P,\tilde P}(x)
	\e^{-2n \Phi(x)} 
	\d x,
	\qquad
	h_{P,\tilde P}(x)
	\coloneqq\prod_{B \in P}
	\frac{\partial^{|B|}}{\partial x_B} \Phi(x) \ \prod_{\tilde B \in \tilde P}
	\frac{\partial^{|\tilde B|}}{\partial x_{\tilde B}} \Phi(x).
\]
Since in the considered setting of a uniform prior on $[-\frac 12,\frac 12]^d$ we have $I_n (x) = \Phi(x)$ for $x \in [-\frac 12,\frac 12]^d$, there holds that $x_n = x_\star$, $\Phi(x_n) = 0$, and, by construction, also $\nabla \Phi(x_n) = 0$.
The latter may cause a faster decay of $\int_{[-\frac 12,\frac 12]^{d}} h_{P,\tilde P} \e^{-2n \Phi} \d x$ than the usual $n^{-d/2}$ depending on the partitions $P, \tilde P$.
For example, let $P = \tilde P = \{ \{1\},\ldots, \{d\} \}$, i.e., $P$ and $\tilde P$ consist only of single blocks $B = \{i\}$, $i = 1,\ldots,d$, then 
\[
	h_{P,\tilde P}(x) = \prod_{j=1}^d \left( \frac{\partial}{\partial x_j} \Phi(x)\right)^2.
\]
Exploiting \eqref{equ:Laplace_coef_2} for the coefficients in the asymptotic expansion of $\int_{[-\frac 12,\frac 12]^{d}} h_{P,\tilde P} \e^{-2n \Phi} \d x$ one can calculate that for these particular partitions $P = \tilde P=\{ \{1\},\ldots, \{d\} \}$ we have
\[
	c_k(h_{P,\tilde P}) = 0 \quad \text{ for } k=0,\ldots,d-1,
\]
but
\[
	c_d(h_{P,\tilde P}) = \prod_{j=1}^d \frac{\kappa_{2\ve_j}}{2} \frac{\partial^2}{\partial x^2_j}\Phi(x_\star) \neq 0,
\]
since $\frac{\partial^2}{\partial x^2_j}\Phi(x_\star) \neq 0$ due to $\nabla^2 \Phi(x_\star)$ being positive definite.
Hence, for these partitions $|P|=|\tilde P| = d$ we get
\[
	\int_{[-\frac 12,\frac 12]^{d}}
	h_{P,\tilde P}(x)
	\e^{-2n \Phi(x)} 
	\d x 
	= c_d(h_{P,\tilde P})  n^{-d/2 - d} + \mc O(n^{-3d/2 - 1}).
\]
We can extend this reasoning to arbitrary partitions $P,\tilde P \in \Pi_d$.
To this end, let $|P|_1 \coloneqq |\{B \in P\colon |B|=1 \} |$ denote the number of single blocks $|B|=1$ in $P \in \Pi_d$.
Then, we know by the definition of $h_{P,\tilde P}$ that $h_{P,\tilde P}$ posseses a zero of order $|P|_1+|\tilde P|_1$ in $x_\star$.
This in turn, implies that the first $\left\lfloor \frac{|P|_1+|\tilde P|_1}2\right\rfloor $ coefficients in the asymptotic expansion of $\int_{[-\frac 12,\frac 12]^{d}} h_{P,\tilde P} \e^{-2n \Phi} \d x$ are zero, hence, 
\[
	\int_{[-\frac 12,\frac 12]^{d}}
	h_{P,\tilde P}(x)
	\e^{-2n \Phi(x)} 
	\d x 
	\; \sim \;
	c_{P,\tilde P} n^{-d/2 - \left\lfloor |P|_1/2 + |\tilde P|_1/2 \right\rfloor}.
\]
Thus, for arbitrary $P,\tilde P \in \Pi_d$ we have
\[
	(-n)^{|P|+|\tilde P|} 
	\int_{[-\frac 12,\frac 12]^{d}} \Theta^2_n(x) \prod_{B \in P}
	\frac{\partial^{|B|}}{\partial x_B} \Phi(x) \prod_{\tilde B \in \tilde P}
	\frac{\partial^{|\tilde B|}}{\partial x_{\tilde B}} \Phi(x) \d x
	\; \sim 
	\; c_{P,\tilde P} n^{-d/2 + |P| + |\tilde P| - \left\lfloor |P|_1/2 + |\tilde P|_1/2 \right\rfloor}.
\]
If we maximize the exponent on the righthand side we get that
\[
	\max_{P, \tilde P \in \Pi_d} |P| + |\tilde P| - \left\lfloor |P|_1/2+ |\tilde P|_1/2 \right\rfloor
	=
	d
\]
where the maximum is obtained, e.g., for the above choice of $P = \tilde P=\{ \{1\},\ldots, \{d\} \}$.
This means that certain summands in $F_d(n)$ grow like $n^{d/2}$ whereas the other ones grow slower w.r.t.~$n$.
Thus, we get that $F_d(n) \sim c n^{d/2}$ which yields the statement.

\subsection{Proof of Lemma \ref{lem:Laplace_QMC_norm}}
\label{sec:Laplace_QMC_norm}
Since the transformation
\[
	g_n(x) \coloneqq x_\star+ n^{-1/2} Ax,
	\qquad
	A\coloneqq\sqrt{2\ln|\tau|} Q D^{-1/2},
\]
is linear, we have for $j \in \{1,\ldots,d\}$
\[
	\frac{\partial}{\partial x_j} \Phi(g_n(x))
	=
	n^{-1/2} \left( \nabla \Phi\right)(g_n(x))^\top A_{\cdot j} 
\] 
where $A_{\cdot j}$ denotes the $j$th column of $A$.
Thus, for a $\vnu \subset \{1,\ldots,d\}$ we get
\[
	\frac{\partial^{|\vnu|}}{\partial x_{\vnu}} \Phi(g_n(x))
	=
	n^{-|\vnu|/2} \left( \nabla^{|\vnu|} \Phi\right)(g_n(x))[A_{\cdot \nu_1}, \ldots, A_{\cdot \nu_{|\vnu|}} ]
\] 
where the last term on the righthand side denotes the application of the multilinear form $(\nabla^{|\vnu|} \Phi)(g_n(x))\colon \bbR^{d\times |\vnu|} \to \bbR$ to the $|\vnu|$ arguements $A_{\cdot \nu_j} \in \bbR^d$.
To keep the notation short, we denote by $\nabla^{|\vnu|}\Phi(g_n(x))[A_\vnu]$ the term $(\nabla^{|\vnu|} \Phi)(g_n(x))[A_{\cdot \nu_1}, \ldots, A_{\cdot \nu_{|\vnu|}} ]$.
By Faa di Bruno we obtain now for any $\vnu \subset \{1,\ldots,d\}$ that
\begin{align*}
	\frac{\partial^{|\vnu|}}{\partial v_\vnu} \Theta_n(g_n(x))
	& =  \sum_{P \in \Pi(\vnu)}
	(-n)^{|P| - |\vnu|/2} \Theta_n(g_n(x))
	\cdot
	\prod_{B \in P}
	\nabla^{|B|}\Phi(g_n(x))[A_B]
\end{align*}
which yields
\begin{align*}
	\|\Theta_n\circ g_n\|^2_\gamma
	& =
	\sum_{\vnu \subset \{1,\ldots,d\}}
	\frac 1{\gamma_\vnu} \int_{[-\frac 12,\frac 12]^{|\vnu|}}  \left( \int_{[-\frac 12,\frac 12]^{d-|\vnu|}}  \frac{\partial^{|\vnu|}}{\partial v_\vnu} (\Theta_n\circ g_n)(x) \ \d x_{-\vnu}  \right)^2 \ \d x_\vnu
	\leq
	\sum_{\vnu \subset \{1,\ldots,d\}}
	\frac {F_\vnu(n)}{\gamma_\vnu} 
\end{align*}
where
\begin{equation}\label{equ:Fn_2}
	F_\vnu(n)
	\coloneqq
	\int_{[-\frac 12,\frac 12]^{|\vnu|}} \left( \int_{[-\frac 12,\frac 12]^{d-|\vnu|}} \sum_{P \in \Pi(\vnu)}
	n^{|P|-|\vnu|/2} \left| \Theta_n(g_n(x)) \right|
	\prod_{B \in P}
	\left|\nabla^{|B|}\Phi(g_n(x))[A_B]\right| 
	\ \d x_{-\vnu}  \right)^2 \ \d x_\vnu.
\end{equation}
Note, that we can bound 
\[
	\left|\nabla^{|B|}\Phi(g_n(x))[A_B]\right| 
	\leq 
	\left\|\nabla^{|B|}\Phi(g_n(x))\right\| \ \prod_{j\in B} \|A_j\| 
\]
where we set the ``norm'' of the multilinear form $\nabla^{|B|}\Phi(g_n(x)) \colon \bbR^{d \times |B|} \to \bbR$ as
\[
	\left\| \nabla^{|B|}\Phi(g_n(x))\right\|
	\coloneqq
	\sup_{\|x_j\| \leq 1} \left| \nabla^{|B|}\Phi(g_n(x))[x_1,\ldots,x_{|B|}] \right|.
\]
Setting $c_A \coloneqq \prod_{j=1}^d \|A_j\|$ we get
$
	\left|\prod_{B \in P} \nabla^{|B|}\Phi(g_n(x))[A_B] \right|
	\leq c_A \prod_{B \in P} \left\|\nabla^{|B|}\Phi(g_n(x))\right\|
$
and obtain by multiplication---omitting the integral domains for a moment---that
\begin{align*}
	F_\vnu(n)
	& = 
	\sum_{P \in \Pi(\vnu)} \sum_{\tilde P \in \Pi(\vnu)}
	n^{|P|+|\tilde P|-|\vnu|} 
	\int 
	\left(\int \left| \Theta_n(g_n(x)) \prod_{B \in P}
	\nabla^{|B|}\Phi(g_n(x))[A_B]
	\right| \d x_{-\nu} \right)\\
	& \qquad \qquad
	\left(\int \left| \Theta_n(g_n(x)) \prod_{\tilde B \in \tilde P}
	\nabla^{|\tilde B|}\Phi(g_n(x))[A_{\tilde B}]
	\right| \d x_{-\nu} \right)
	\d x_\nu\\
	& \leq
	\sum_{P,\tilde P \in \Pi(\vnu)} 
	n^{|P|+|\tilde P|-|\nu|} 
	\left\| (\Theta_n \circ g_n)\ c_A \ \prod_{B \in P} \ \left\| \left(\nabla^{|B|}\Phi\right) \circ g_n\right\| \right\|_{L^2([-\frac 12,\frac 12]^d)} \\
	& \qquad 
	\left\| (\Theta_n\circ g_n) \ c_A \ \prod_{\tilde B \in \tilde P} \left\| \left(\nabla^{|\tilde B|}\Phi\right) \circ g_n)\right\|\right\|_{L^2([-\frac 12,\frac 12]^d)}\\
	& \leq c_A^2\left( \sum_{P\in \Pi(\vnu)}  
	n^{|P| - |\vnu|/2} \left\| (\Theta_n \circ g_n) \prod_{B \in P}  \left\| \left(\nabla^{|B|}\Phi\right) \circ g_n \right\| \right\|_{L^2([-\frac 12,\frac 12]^d)} \right)^2
\end{align*}
where we used the Cauchy--Schwarz and Jensen's inequality in the second line.
We apply Laplace's method in order to examine the $L^2$-norm above:
\begin{align*}
	\left\| (\Theta_n \circ g_n) (h_P \circ g_n)\right\|^2_{L^2([-\frac 12,\frac 12]^d)} 
	& = \int_{[-\frac 12,\frac 12]^d} \e^{- 2n \Phi(x_\star + n^{-1/2} A x)} \ h_P^2(x_\star + n^{-1/2}Ax) \ \d x
\end{align*}
where $h_P = \prod_{B \in P}  \left\| \left(\nabla^{|B|}\Phi\right) \circ g_n \right\|$.
By the substitution $y \coloneqq g_n(x) = x_\star + n^{-1/2} Ax$ we get 
\begin{align*}
	\left\| (\Theta_n \circ g_n) (h_P \circ g_n)\right\|^2_{L^2([-\frac 12,\frac 12]^d)} 
	& 
	= \frac{1}{C_{J_{Trans}}(n)}\int_{[-\frac 12,\frac 12]^d} \exp(- 2n \Phi(y)) \ h_P^2(y) \ \d y,
\end{align*}
where $C_{J_{Trans}}(n) = \det(n^{-1/2}A) = n^{-d/2}\det(A)$. 
Since also $2\Phi$ satisfies the assumptions of Theorem \ref{theo:laplace_method} and $\Phi(x_\star)=0$ we obtain
\[
	\left\| (\Theta_n \circ g_n) (h_P \circ g_n)\right\|^2_{L^2([-\frac 12,\frac 12]^d)} 
	=
	\frac{(2P)^{d/2}h_P^2(x_\star)}{\det(\nabla^2\Phi(x_\star))}
	+ \mc O(n^{-1}).
\]
However, since $\nabla \Phi(x_\star) = 0$ there holds $h_P^2(x_\star)=0$ if there exists a single block $|B|=1$ in $P$ which then yields to a decay of the $L^2$-norm as least as fast as $n^{-1}$.
In particular, denoting by $|P|_1$ the number of single blocks in $P$ we obtain by the same reasoning as in Section \ref{sec:Prior_QMC_norm} that
\[
	\frac{1}{C_{J_{Trans}}(n)} \int_{[-\frac 12,\frac 12]^d} \exp(- 2n \Phi) \prod_{B \in P} \left(\frac{\partial^{|B|}}{\partial x_B} \Phi\right)^2 \ \d y
	\sim c_P \ n^{-|P|_1}
\]
where $c_P = c_{|P|_1}(h^2_P)$ as in \eqref{equ:Laplace_coef}.
Hence,
\begin{align*}
	F_\vnu(n)
	& 
	\sim
	c_A^2
	 \left( \sum_{P\in \Pi(\vnu)} c_P n^{|P|- |\nu|/2 -|P|_1/2} \right)^2.
\end{align*}
Similarly to Section \ref{sec:Prior_QMC_norm} we can derive $\max_{P \in \Pi(\vnu)} |P|- |P|_1/2 =  |\vnu|/2$ which yields that $	F_\vnu(n) \in \mc O(1)$ as $n\to \infty$ and concludes the proof.

%
\subsection{Proof of Equation \eqref{equ:Var_mu_2}}
\label{sec:Var_mu_n}
We show in the following that for $f,\pi_0\in C^4(\bbR^d, \bbR)$ and $\Phi\in C^5(\bbR^d, \bbR)$ we have
\[
	\Var_{\mu_n}(f)
	=
	n^{-1} \|\nabla f(x_\star)\|^2_{H_\star^{-1}} + \mc O(n^{-2}).
\]
To this end, we use $\Var_{\mu_n}(f) = \evalt{\mu_n}{f^2} - \evalt{\mu_n}{f}^2$ and the asymptotics
\begin{align*}
	\evalt{\mu_n}{f} & = f(x_\star) + \tilde c_1(f,\pi_0) n^{-1} + \mc O(n^{-2}),
	& \evalt{\mu_n}{f^2} & = f^2(x_\star) + \tilde c_1(f^2,\pi_0) n^{-1} + \mc O(n^{-2}),
\end{align*}
where $\tilde c_1(f,\pi_0) = \frac{1}{c_0(\pi_0)} c_1(f\pi_0) - \frac{c_1(\pi_0)}{c^2_0(\pi_0)}c_0(f\pi_0)$, analogously for $\tilde c_1(f^2,\pi_0)$.
Thus, we have
\begin{align*}
	\Var_{\mu_n}(f)
	& =
	f^2(x_\star) + \tilde c_1(f^2,\pi_0) n^{-1} + \mc O(n^{-2})
	- \left[f(x_\star) + \tilde c_1(f,\pi_0) n^{-1} + \mc O(n^{-2})\right]^2\\
	& = \left[\tilde c_1(f^2,\pi_0) - 2f(x_\star)\tilde c_1(f,\pi_0)  \right] n^{-1} + \mc O(n^{-2}).
\end{align*}
Since $\Var_{\mu_n}(f) = \Var_{\mu_n}(f-\mathrm{const})$ we may assume w.l.o.g.~that $f(x_\star) = 0$.
Hence,
\[
	\Var_{\mu_n}(f)
	= \tilde c_1(f^2,\pi_0) n^{-1} + \mc O(n^{-2})
	= \frac{c_1(f^2\pi_0)}{c_0(\pi_0)} n^{-1} + \mc O(n^{-2}),
\]
where the last equality follows due to $c_0(f^2\pi_0) \propto f^2(x_\star) = 0$.
We recall that
\[
	c_1(f^2\pi_0)
	=
	\sum_{\valpha \in \bbN_ 0^d \colon |\valpha| = 1}
	\frac{\kappa_{2\valpha}}{(2\valpha)!} D^{2\valpha} \widetilde F_0(0)
	=
	\sum_{i=1}^d
	\frac{\kappa_{2\ve_i}}{2} \frac{\partial^2}{\partial x_i^2} \widetilde F_0(0),
\]
where
\[
	\widetilde F_0(x) := f^2(h(x))\, \pi_0(h(x))\,\det(\nabla h(x))
\]
with $h\colon \Omega \to U(x_\star)$ denoting the diffeomorphism between $0 \in \Omega \subset \bbR^d$ and a particular neighborhood $U(x_\star)$ of $x_\star$ mapping $h(0) = x_\star$ and such that $\det(\nabla h(0)) = 1$.
Moreover, as outlined in \cite[Section IX.2]{Wong2001}, we have that 
$
	 \nabla h(0)^\top H_\star  \nabla h(0) = \Lambda
$
where $\Lambda := \diag(\lambda_1,\ldots,\lambda_d)$ denotes the diagonal matrix containing the eigenvalues $\lambda_i$ of $H_\star$.
This is equivalent to
\[
	\nabla h(0)\Lambda^{-1}\nabla h(0)^\top = H_\star^{-1}
\]
which we will use later on. 
Furthermore, since $\kappa_{\valpha} = \kappa_{\alpha_1}\cdots \kappa_{\alpha_d} \in\bbR$ and $\kappa_{\alpha_i} = (2/\lambda_i)^{(\alpha_i +1)/2} \Gamma((\alpha_i+1)/2)$ for even $\alpha_i$, we get
\[
	\kappa_{2\ve_i} 
	= 
	\frac{\sqrt{\pi}}{2}\left(\frac {2}{\lambda_i}\right)^{3/2}
	\prod_{j\neq i} \sqrt{\frac {2\pi}{\lambda_j}}
	=
	\frac{\sqrt{\det(2\pi H^{-1}_\star)}}{\lambda_i}
	=
	\frac{\kappa_{\boldsymbol 0}}{\lambda_i},
\]
since $\kappa_{\boldsymbol 0} = \det(2\pi H^{-1}_\star)^{1/2}$.
Thus,
\[
	\frac{c_1(f^2\pi_0)}{c_0(\pi_0)}
	=
	\frac{\kappa_{\boldsymbol 0}\sum_{i=1}^d \frac{1}{2\lambda_i} \frac{\partial^2}{\partial x_i^2} \widetilde F_0(0)}{\kappa_{\boldsymbol 0} \pi_0(x_\star)}
	=
	\sum_{i=1}^d \frac{1}{2\lambda_i \pi_0(x_\star)} \frac{\partial^2}{\partial x_i^2} \widetilde F_0(0).
\]
We now compute $\frac{\partial^2}{\partial x_i^2} \widetilde F_0(0)$.
To this end, let us introduce $G_0(x) := \pi_0(h(x))\,\det(\nabla h(x))$, i.e., $\widetilde F_0(x) = f^2(h(x)) G_0(x)$ as well as $G_0(0) = \pi_0(x_\star)$.
We first derive by the product and chain rule
\[
	\frac{\partial}{\partial x_i} \widetilde F_0(x)
	=
	2f(h(x))\ \nabla f(h(x))^\top \left[\frac{\partial}{\partial x_i}h(x)\right]\ G_0(x)
	+
	f^2(h(x))\frac{\partial}{\partial x_i} G_0(x).
\]
For the second derivative $\frac{\partial^2}{\partial x_i^2} \widetilde F_0(x)$ evaluated at $x=0$ we can exploit once more that w.l.o.g.~$f(x_\star)=0$ which then yields
\[
	\frac{\partial^2}{\partial x^2_i} \widetilde F_0(0)
	=
	2 \left| \nabla f(x_\star)^\top \frac{\partial}{\partial x_i}h(0)\right|^2 \ \pi_0(x_\star),
\]
i.e., all other terms which we obtain by applying the product (and chain) rule are zero, since they contain $f(h(0)) = f(x_\star) = 0$ as a factor.
In summary, we end up with
\begin{align*}
	\frac{c_1(f^2\pi_0)}{c_0(\pi_0)}
	& = \sum_{i=1}^d \frac{1}{2\lambda_i \pi_0(x_\star)} \frac{\partial^2}{\partial x_i^2} \widetilde F_0(0)
	= \sum_{i=1}^d \frac{1}{\lambda_i} \left| \nabla f(x_\star)^\top \frac{\partial}{\partial x_i}h(0)\right|^2\\
	& = \left[\nabla h(0)^\top \nabla f(x_\star)\right]^\top \Lambda^{-1} \left[\nabla h(0)^\top \nabla f(x_\star)\right]\\
	& = \nabla f(x_\star)^\top H_\star^{-1} \nabla f(x_\star).
\end{align*}
This concludes the proof.


\rev{\section{Visualization of the necessity of the third term in Assumption \ref{assum:LA_Conv}}
%
Consider the following setting: 
Let $\pi_0(x) \propto \exp(- x^2)$, $x\in\bbR$, (with suitable normalization) and let $\Phi_n\colon \bbR\to \bbR$ be a smooth function satisfying the following conditions:
\begin{itemize}
\item
For some $0 < r_0 < \frac 12$ we have $\Phi_n(x) = x^2$ for $|x|\leq r_0$.
\item
For some $0 < \delta < 1$ and $r_n := \sqrt{n} > r_0$ as well as $r_n^+ := r_n + \exp(n\delta)$ we have for all  $r_n  < |x| \leq r_n^+$ that
\[
	\Phi_n(x) = - \frac{x^2}n + \delta \leq -1 + \delta < 0.
\]
\item
For a $r'_n > r_n^+$ we have $\Phi_n(x) \equiv 0$ for all $|x| \geq r'_n$.
\item
On $r_0 < |x| < r_n$ and $r_n^+ < |x| < r'_n$ the function $\Phi$ is smooth and monotone.
\end{itemize}
This means that $I_n(x) = \Phi_n(x) - \frac 1n \log(\pi_0(x))$ has the following properties:
\begin{enumerate}
\item
$x_n := \argmin I_n = 0$ with $I_n(x_n) = 0$ and $I_n(x) = \frac {n+1}n x^2$ f\"ur $|x| \leq r_0$. 
\item For $r_n  < |x| \leq r_n^+$ we have $I_n(x) \equiv \delta$.
\item The quadratic approximation used in the Laplace approximation is then $\widetilde I_n(x) = \frac {n+1}n x^2$ for all $x$ which yields for $r_n  < |x| \leq r_n^+$
\[
	R_n(x) 
	= I_n(x) - \widetilde I_n(x)
	= \delta - \frac {n+1}n x^2 
	\leq \delta - (n+1) < 0.
\]
\end{enumerate}
\begin{figure}[hbtp]
\centering
\includegraphics[width=0.7\textwidth]{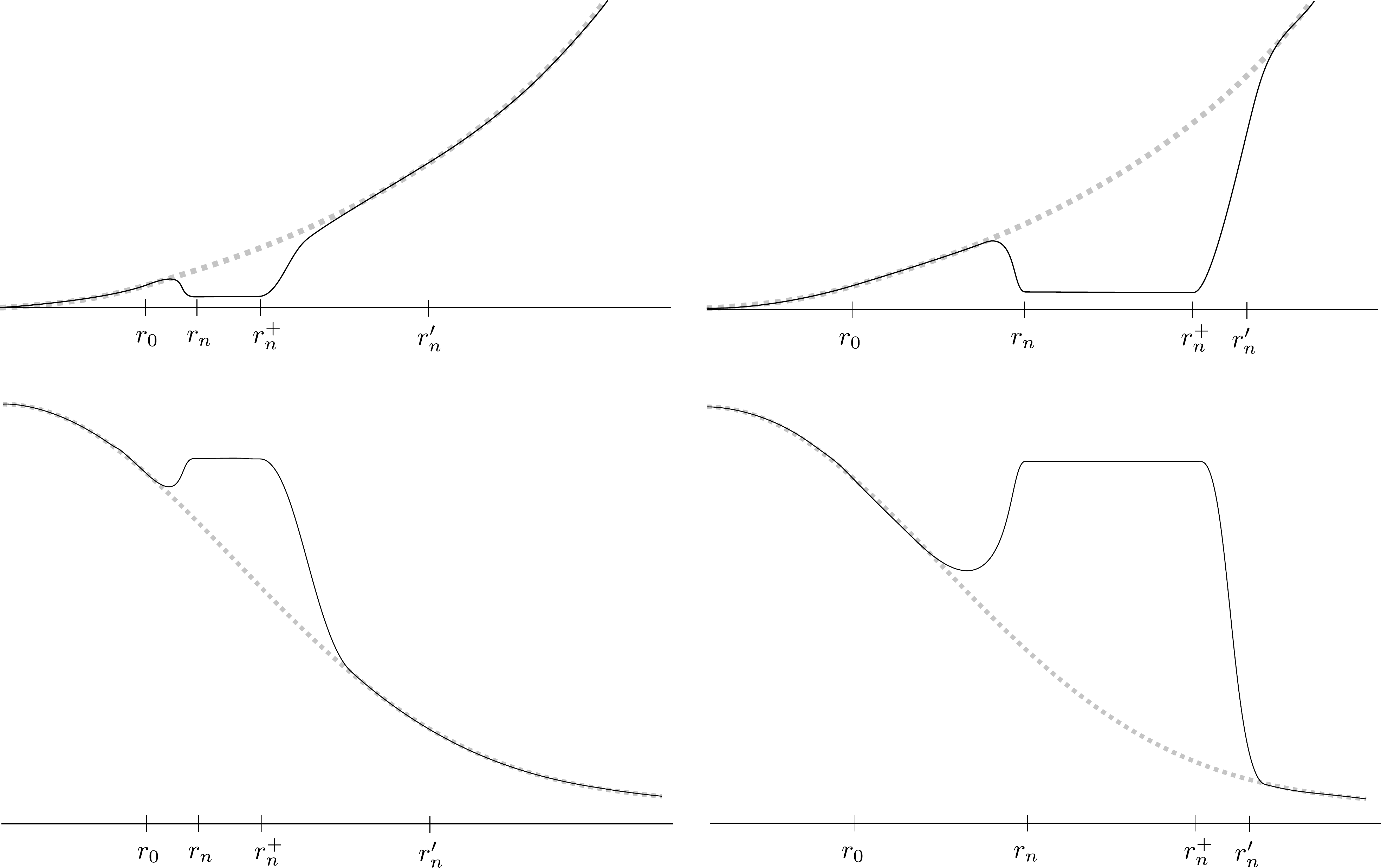}
\caption{Sketch of $I_n$ (top) and $\mu_n$ (bottom) along with quadratic approximations $\tilde I_n$ and $\LAn$ (dotted lines). Right column shows the same for a larger value of $n$.}
\end{figure}
Note that Assumption \ref{assum:LA_0} and items 1 and 2 of Assumption \ref{assum:LA_Conv} are fulfilled.
However, we will now show that the resulting sequence of measures $\mu_n$ with normalized Lebesgue density $\frac 1{Z_n} \exp(-n I_n)$ is not tight which by Prokhorov's theorem precludes the existence of a weakly converging subsequence.
In particular, the $\mu_n$ are not a Cauchy sequence in total variation or Hellinger distance.
To this end, we first obtain for the normalizing constants $Z_n$ that
\[
	Z_n = \int_{\bbR} \e^{-n I_n(x)} \ \d x
	\geq
	\int_{\{ x\colon r_n < |x| < r_n^+\}} \e^{-n \delta} \ \d x
	= 2\e^{n \delta}
	\e^{-n \delta}
	= 2.
\]
Moreover, we have
\[
	 \int_{|x| < r_0}  \e^{- n I_n(x) }\ \d x 
	 = \int_{|x| < r_0}  \e^{- (n+1) x^2 }\ \d x 
	\leq 	2r_0 
	\leq 1.
\]
Hence, we get for the ball $B_{r_0}(x_n) = \{x\colon |x| < r_0\}$
\[
	\mu_n(B_{r_0}(x_n))
	= \frac1{Z_n}\int_{|x| < r_0}  \e^{-(n+1) x^2}\ \d x
	\leq \frac{2r_0}{Z_n} \leq r_0 \leq \frac 12.
\]
This means, that for $n\to \infty$ the fixed ball $B_{r_0}(x_n)$ around global peak of the density of $\mu_n$ has $\mu_n$-probability mass less than $\frac 12$. 
In other words, the measures $\mu_n$ do \emph{not} concentrate around $x_n = 0$.
Additionally, we can assume that $\Phi_n$ is such that for a constant $c < \infty$ we have
\begin{align*}
	\int_{ r_0 < |x| < r_n}  \e^{-n \Phi_n(x) - x^2} \ \d x 
	+ \int_{ r^+_n < |x| < r'_n} \e^{-n \Phi_n(x) - x^2} \ \d x  
	+ \int_{ r_n'  < |x|}  \e^{-x^2} \d x \leq c
\end{align*}
and, thus, $Z_n \leq 3+c$. 
But then, for all $n\in\bbN$
\[
	\mu_n( [n, + \infty]) \geq \frac 2{Z_n} \geq \frac 2{3+c} > 0, 
\]
which means that the sequence $\{\mu_n\}_n$ is not tight.
This should demonstrate sufficiently why the third term in Assumption \ref{assum:LA_Conv} is needed for a suitable concentration of the measures $\mu_n$.}

\subsection{A sufficient condition for the third term in Assumption \ref{assum:LA_Conv}}
We provide now a useful sufficient condition on the functions $\Phi_n$ and $\pi_0$ such that the third term in Assumption \ref{assum:LA_Conv} is fulfilled.
In order to satisfy $I_n(x_n) = 0$ as required in Assumption \ref{assum:LA_0} we set
\[
	\iota_n := \Phi_n(x_n) - \frac 1n \log \pi_0(x_n)
\]
consider $\widehat{\Phi}_n(x) := \Phi_n(x) - \iota_n$ and represent the measures $\mu_n$ by
\[
	\mu_n(\d x) = \frac 1{Z_n} \exp(- n \widehat{\Phi}_n(x)) \ \pi_0(\d x),
	\qquad
	Z_n = \int_{\S_0} \exp(- n \widehat{\Phi}_n(x)) \ \pi_0(\d x),
\]
with $I_n(x) = \widehat{\Phi}_n(x) - \frac 1n \log \pi_0(x)$ in accordance to Remark \ref{rem:In_zero}.
We then have
\begin{align*}
	\exp(-n I_n(x)) 
	 & = \exp\left( -n \widehat{\Phi}_n(x) \right) \ \pi_0(x) 
	=  \exp\left( n \iota_n  - n\inf_{S_0} \Phi_n - n(\Phi_n(x) - \inf_{S_0} \Phi_n)  \right) \pi_0(x)\\
	& \leq \exp\left( n (\iota_n  - \inf_{S_0} \Phi_n)\right) \pi_0(x)
	= \exp\left( n (\Phi_n(x_n)  - \inf_{S_0} \Phi_n)\right) \frac{\pi_0(x)}{\pi_0(x_n)}.
\end{align*}
By the first term in Assumption \ref{assum:LA_Conv}, we can conlude that $\pi_0(x_n) \to \pi_0(x_\star) >0$, i.e.,
the third term of Assumption \ref{assum:LA_Conv} follows if there exists an $\epsilon >0$ such that
\[
	\pi_0 \in L^{1-\epsilon}(\S_0;\bbR), 
\]
and if
\[
	\Phi_n(x_n) - \inf_{\S_0} \Phi_n \in \mc O(n^{-1}).
\]	
Note that for $\Phi_n \equiv \Phi$ we have $\Phi(x_n) - \inf_{\S_0} \Phi\in \mc O(n^{-2})$, since then $x_\star = \lim_{n\to\infty} x_n$ is a minimizer of $\Phi$ and $\Phi$ is strongly convex in a neighborhood of $x_\star$, cf. the paragraph ``Preliminaries'' in Section \ref{sec:numerics} and Remark \ref{rem:Conv_xn}.

\bibliographystyle{abbrv}
\bibliography{literature}

\end{document}